\RequirePackage{etoolbox}
\csdef{input@path}{{style/}{graphics/}}
\documentclass[aap,MSNbibl,seceqn,dvips]{arximspdf}

\doi{10.1214/14-AAP1027} 
\volume{25}
\issue{3}
\pubyear{2015}
\firstpage{1420}
\lastpage{1474}
\docsubty{FLA}

\makeatletter
\renewcommand{\lll}{\bigl\langle\bigl\langle}
\newcommand{\rrr}{\bigr\rangle\bigr\rangle}
\renewcommand{\ll}{\langle\!\langle}
\renewcommand{\gg}{\rangle\!\rangle}

\newcommand{\rrvert}{\vert}

\newcommand{\rrVert}{\Vert}
\newcommand{\llvert}{\vert}
\newcommand{\llVert}{\Vert}
\newtheorem{teo}{Theorem}[section]
\newtheorem{prop}[teo]{Proposition}
\newproclaim{assumption}[teo]{Assumption}
\newtheorem{lem}[teo]{Lemma}
\newproclaim{definition}[teo]{Definition}
\newproclaim{rem}[teo]{Remark}
\makeatother

\begin{document}
\begin{frontmatter}

\title{Distribution-valued heavy-traffic limits for the~$G/\mathit{GI}/\infty$ queue}
\runtitle{Heavy-traffic limits for the $G/\mathit{GI}/\infty$ queue}

\begin{aug}
\author[A]{\fnms{Josh}~\snm{Reed}\corref{}\ead[label=e1]{jreed@stern.nyu.edu}}
\and
\author[B]{\fnms{Rishi}~\snm{Talreja}\ead[label=e2]{rtalreja@kcg.com}}
\runauthor{J. Reed and R. Talreja}
\affiliation{New York University and KCG Holdings, Inc.}
\address[A]{Stern School of Business\\
New York University \\
New York, New York 10012\\
USA\\
\printead{e1}}
\address[B]{KCG Holdings, Inc.\\
Jersey City, New Jersey 07310\\
USA\\
\printead{e2}}
\end{aug}

\received{\smonth{12} \syear{2010}}
\revised{\smonth{3} \syear{2014}}

%
\begin{abstract}
We study the $G/\mathit{GI}/\infty$ queue in heavy-traffic using
tempered distri\-bution-valued processes which track the age and residual service
time of each customer in the system. In both cases, we use the
continuous mapping theorem together with functional central limit
theorem results
in order to obtain fluid and diffusion limits for these
processes in the space of tempered distribution-valued processes. We
find that our diffusion limits are tempered
distribution-valued Ornstein--Uhlenbeck processes.
\end{abstract}

%
\begin{keyword}[class=AMS]
\kwd[Primary ]{60F05}
\kwd{60K25}
\kwd[; secondary ]{60J25}
\end{keyword}
\begin{keyword}
\kwd{Queueing}
\kwd{central limit theorem}
\kwd{Markov process}
\kwd{infinite server}
\end{keyword}
\end{frontmatter}

\section{Introduction}\label{SecIntro}

Limit theorems for the infinite-server queue in heavy-traffic have a rich
history starting with the seminal work of Iglehart \cite{Ig65} on the
$M/M/\infty$ queue. This work then inspired a line of research aimed at
extending the results of \cite{Ig65} to additional classes of service time
distributions. Whitt \cite{WhittInfinite} studies the $\mathit{GI}/\mathit{PH}/\infty$
queue, having phase-type service-time distributions, and Glynn and
Whitt \cite{GlynnWhitt} consider the $\mathit{GI}/\mathit{GI}/\infty$ queue with service
times taking values in a finite set. In \cite{Borovkov},
\cite{KrPu97} and \cite{PuhalskiiReed08}, the $G/\mathit{GI}/\infty$ queue is
studied with general service time distributions. Pang et al.
\cite{PTW} gives a
survey of these results.

In this paper, we study two Markov processes associated with the
$G/\mathit{GI}/\infty$ queue. The first process which we study is a tempered
distribution-valued process which tracks the age of
each customer in the system. We refer to this process as the age
process. The second process
which we study is also a tempered distribution-valued
process and it tracks the residual service time of each customer in the
system as well as the amount of time since departure for each customer
who has left the system.
We refer to this process as the residual service time process. Although
analyzing
either of these processes might at first appear to
be a difficult task, one of the key themes that runs throughout the
present paper is that techniques originally developed for establishing
heavy-traffic limits in the finite-dimensional setting may also be
successfully applied in the more abstract
infinite-dimensional setting.

Our main results in this paper are to obtain fluid and diffusion limits
for both the age and residual service time processes. In particular,
for both the age and the residual service time processes, we use the
continuous mapping theorem together with functional central limit
theorem results in order to establish our main results. The
corresponding diffusion limits that we obtain for both the age and
residual time
processes may be characterized as tempered distribution-valued
Ornstein--Uhlenbeck processes.

The tempered distribution-valued representation which we use for the
residual service time process was also used by Decreusefond and Moyal
\cite{DM08b} in order to analyze the $M/G/\infty$ queue. However, in
the current
paper we go beyond analyzing the residual service time process and also
analyze the age process, which was not treated in \cite{DM08b}. In
particular, we provide a diffusion limit result
for the tempered distribution-valued age process which is fundamentally
different from the diffusion limit for the age process obtained in
\cite{DM08b}. Moreover, our general methodology for
proving our main results differs from that employed in \cite{DM08b}.
Specifically, while the approach in \cite{DM08b} has as its starting
point the infinitesimal generator of the residual service
time process, in the present paper we begin by defining the model
primitives and setting up a governing system equation for both the age
and residual service time processes. As alluded to above, we then
rely upon the continuous mapping theorem and functional central limit
theorem results when proving our main results.

A second major contribution of our work is to make a connection
\mbox{between} the literature on infinite-dimensional heavy-traffic limits
for queueing \mbox{systems} \mbox{\cite{DM08,DM08b,DGP,Gromoll04,GPW02,GRZ08,KR07}}
and the vast
literature on infinite-dimensional Ornstein--Uhlen\-beck processes
motivated by applications to interacting particle systems
\cite{Langevin,BojdeckiGorostiza,Bojdecki,Hitsuda,HolleyStroock,KallianpurPerezAbreuMap,KallianpurPerezAbreuCont,LangevinMitoma,MitomaGeneralizedOU}.
Our work especially relies upon
\cite{KallianpurPerezAbreuMap} and \cite{KallianpurPerezAbreuCont} in
order to
prove continuity of a particular regulator map.

Another set of papers related to ours are those of Kaspi and Ramanan
\cite{kaspi2010spde,KR07}. Although these works
analyze the many-server queue with general service time distributions,
their infinite-dimensional representation of the system is similar to ours.
Fluid limits are established for the system in \cite{KR07} in the
space of Radon measure-valued processes. However, when
establishing corresponding diffusion limits, the limit process
evidently falls out of the space of Radon measure-valued
processes and distribution-valued processes are used instead in \cite
{kaspi2010spde}. In the present
paper, we follow the work of \cite{DM08b} in which the space of
tempered distributions is used. This
space may be characterized as the topological dual of Schwartz space,
the space of rapidly decreasing, infinitely differentiable functions.

We also mention the work \cite{PangWhitt}, where the authors build
upon the
work of \cite{GlynnWhitt} and \cite{KrPu97} in order to prove heavy-traffic
limits for the $G/\mathit{GI}/\infty$ queue in a two-parameter function
space.
They analyze both the age and the residual service time process as we
do. The main
difference between the present work and \cite{PangWhitt} is that in
the present work
tempered distribution-valued processes are used which allows one to
apply the
continuous mapping theorem and other standard results in order to obtain
heavy-traffic limits.

The remainder of this paper is now organized as follows. In
Section~\ref{SecSystemEquations}, we derive basic system equations for
both the
age process and the residual service time process. These equations
serve as the starting point
for our analysis in the remainder of the paper. In Section~\ref{SecRegulatorMap}, we
present a regulator map result to be used in conjunction with the
continuous mapping
theorem in order to prove our main results. In Section~\ref{SecMartingale}, we provide martingale results that are used together
with the regulator map of Section~\ref{SecRegulatorMap} in order to obtain
our fluid and diffusion limits. In Sections~\ref{SecFluidLimits}~and~\ref{SecDiffusionLimits}, we prove our fluid and diffusion limits,
respectively. In the \hyperref[appen]{Appendix}, we provide the proofs of several
technical lemmas that are
used throughout the paper.

\subsection{Technical background}\label{SubsecNotation}

We now provide some technical background which is useful for the
remainder of the paper. We begin with some preliminary details.

\subsubsection{Preliminaries}

All random variables and processes in this paper are assumed to be
defined on a common probability space $(\Omega,\mathcal{F},\mathbb
{P})$ and are measurable maps from $(\Omega,\mathcal{F},\mathbb{P})$
to an arbitrary topological space with an associated Borel $\sigma
$-algebra. It turns out that many of the random quantities which we
study in this paper take values in a topological space which is not
metrizable and so we now provide the definition of weak convergence on
an arbitrary topological space. We follow the approach of \cite
{KallianpurXiong}.
Let $X$ be an arbitrary topological space with associated Borel $\sigma
$-algebra $\mathcal{B}(X)$. We say that a sequence of probability
measures $(\mathbb{P}_n)_{n \geq1}$ on $\mathcal{B}(X)$ weakly
converges to a probability measure $\mathbb{P}$ on $\mathcal{B}(X)$,
abbreviated as $\mathbb{P}_n \Rightarrow\mathbb{P} $, if $\int f\,d
\mathbb{P}_n \rightarrow\int f\,d \mathbb{P}$ for every bounded,
continuous real functional $f$ on~$X$ (see Definition 2.2.1 of \cite
{KallianpurXiong}). We also use the notation $\stackrel{\mathbb
{P}}{\rightarrow}$ to denote convergence in probability. For any two
topological spaces $X$ and $Y$, we denote by $X \times Y$ the cartesian
product of $X$ and $Y$ and we associate with $X \times Y$ the product
topology. Note that using the above definition of weak convergence, it
is straightforward to show that the continuous mapping theorem
continues to hold (see, e.g., the proof of Theorem~3.4.1 of
\cite{WhittBook}). In particular, we have the following.

%
\begin{prop}\label{cmtprop}
Let $X$ and $Y$ be two topological spaces and let $(x^n)_{n \geq
1}$ be a sequence of random elements of $X$ such that $x^n \Rightarrow
x$. If $g\dvtx X \mapsto Y$ is a continuous function, then $g(x^n)
\Rightarrow g(x)$ in $Y$.
\end{prop}

Next, for each $0 < T < \infty$, let
$\mathbb{D}([0,T],\mathbb{R})$ denote the space of functions from
$[0,T]$ to $\mathbb{R}$ that are
right-continuous on $[0,T)$ with left limits everywhere on $(0,T]$. We
equip $\mathbb{D}([0,T],\mathbb{R})$ with the Skorokhod
$J_1$-topology \cite{Billingsley99}. We also note that we will
commonly abbreviate the notation $\mathbb{D}([0,T],\mathbb{R})$ by
simply writing $\mathbb{D}$. Next, let $\mathbb{D}([0,T],\mathbb
{D})$ denote the space of functions from $[0,T]$ to $\mathbb{D}$ that
are right-continuous on $[0,T)$ with left limits everywhere on $(0,T]$.
We equip $\mathbb{D}([0,T],\mathbb{D})$ with the Skorokhod
$J_1$-topology \cite{Billingsley99} as well. For an element $x \in
\mathbb{D}([0,T],\mathbb{R})$, we set
\begin{eqnarray*}
\llVert x\rrVert_T &=& \sup_{0 \leq t \leq T}\bigl\llvert
x(t)\bigr\rrvert.
\end{eqnarray*}
We denote by $e=(t, t \in[0,T])$, the identity process on $[0,T]$.

\subsubsection{Schwartz space}

The space of rapidly decreasing functions, also known as Schwartz
space, plays an important rule in this paper and so we now provide a
brief review of some of the relevant facts concerning this space. Much
of the material found in this subsection may also be found in \cite
{KallianpurXiong}.

Let $\mathbb R$, $\mathbb R_+$ and $\mathbb R_-$ denote the set of
reals, nonnegative reals and nonpositive reals, respectively. Also,
denote by $\mathbb{N}=\{0,1,2,\ldots\}$ the set of nonnegative
integers. Let
$C^\infty(\mathbb{R})$ denote the set of infinitely differentiable
functions from $\mathbb R$
to $\mathbb R$ and let $C^\infty(\mathbb{R}_{+})$ denote the set of
infinitely differentiable functions from $\mathbb R_{+}$
to $\mathbb R$. Also, let $C^\infty_b(\mathbb{R})$ denote the set of
infinitely differentiable, bounded functions from $\mathbb R$
to $\mathbb R$ whose derivatives of all orders are bounded and,
similarly, let $C^\infty_b(\mathbb{R}_{+})$ denote the set of
infinitely differentiable, bounded functions from $\mathbb R_{+}$ to
$\mathbb{R}$
whose derivatives of all orders are bounded.

Now define
%
%
\begin{equation}
\label{schwartz} \mathcal{S}\equiv\bigl\{ \varphi\in C^\infty(\mathbb{R})\dvtx
\llVert\varphi\rrVert_{\alpha, \beta} < \infty\mbox{ for all }
\alpha,
\beta\in\mathbb N\bigr\},
\end{equation}
where
%
%
\begin{equation}
\label{schwartzseminorm} \llVert\varphi\rrVert_{\alpha, \beta} \equiv
\sup
_{x \in
\mathbb R} \bigl\llvert x^\alpha\varphi^{(\beta)}(x)
\bigr\rrvert
\end{equation}
and $\varphi^{(\beta)}$ denotes the $\beta$th derivative of $\varphi
$. The space $\mathcal{S}$ is commonly referred to as Schwartz space
or the space of rapidly decreasing functions \cite{KallianpurXiong}.

The topology of $\mathcal{S}$ is given by the family
of seminorms \mbox{$\{ \llVert \cdot\rrVert _{\alpha,\beta}\dvtx \alpha,
\beta\in\mathbb N\}$} defined in (\ref{schwartzseminorm}). In
particular, $\varphi_n \rightarrow\varphi$ in $\mathcal{S}$ if
$\llVert \varphi_n - \varphi\rrVert _{\alpha,\beta} \rightarrow
0$ for all $\alpha,\beta\in\mathbb{N}$.
We note that by Lemma 1.3.2 and Theorem 1.3.2 of \cite
{KallianpurXiong}, the space $\mathcal{S}$ is a nuclear Fr\`{e}chet space.
Moreover, we also note that one may construct a sequence of seminorms
\mbox{$\{\llVert \cdot\rrVert _p \dvtx  p \in\mathbb{N}\}$} with the property
that $\llVert \cdot\rrVert _p \leq\llVert \cdot\rrVert _{p+1}$
for each $p \in\mathbb{N}$ and which also induce the same the
topology on $\mathcal{S}$ as the seminorms $\{ \llVert \cdot\rrVert
_{\alpha,\beta} \dvtx \alpha, \beta\in\mathbb N\}$ given above. In
particular, by
Lemma 1.3.3 of \cite{KallianpurXiong}, one has that for each $p \in
\mathbb{N}$, there exists a $k \in\mathbb{N}$ and $C>0$ such that
%
%
\begin{eqnarray}
\llVert\varphi\rrVert_p &\leq&C \max_{0 \leq\alpha, \beta
\leq k+1}
\llVert\varphi\rrVert_{\alpha,\beta}\qquad\mbox{for~all }\varphi\in
\mathcal{S}, \label{lem133}
\end{eqnarray}
and, by Lemma 1.3.4 of \cite{KallianpurXiong}, one has that for each
$\alpha,\beta\in\mathbb{N}$, there exists a $p \in\mathbb{N}$ and
$M>0$ such that
%
%
\begin{eqnarray}
\max_{0 \leq\alpha, \beta\leq k+1}\llVert\varphi\rrVert_{\alpha,\beta
} &\leq& M
\llVert\varphi\rrVert_p\qquad\mbox{for~all }\varphi\in
\mathcal{S}. \label{lem134}
\end{eqnarray}
For a precise construction of $\llVert \cdot\rrVert _p$ for each $p
\in\mathbb{N}$, one may consult page 24 of~\cite{KallianpurXiong}.

The set of all linear maps from $\mathcal{S}$ to $\mathcal{S}$ is
denoted by $L(\mathcal{S},\mathcal{S})$ and the strong topology on
$L(\mathcal{S},\mathcal{S})$ is defined in the following manner.
A subset $B$ of $\mathcal{S}$ is said to be bounded if for any
neighborhood $U$ of $\varphi\equiv0 \in\mathcal{S}$, there exists a
constant $\alpha> 0$ such that $\alpha^{-1}B \subset U$ (see
Definition 1.1.7 of \cite{KallianpurXiong}). The strong topology on
$L(\mathcal{S},\mathcal{S})$ is then
given by the following definition (see Theorem 1.2.1 of \cite
{KallianpurXiong}).

\begin{definition}
For each bounded subset $B$ of $\mathcal{S}$ and $p \in\mathbb{N}$, let
\begin{eqnarray*}
q_{B,p}(T) &\equiv& \sup_{\varphi\in B}\llVert T \varphi
\rrVert_p\qquad\mbox{for~all }T \in L(\mathcal{S},\mathcal{S}).
\end{eqnarray*}
Then $\{q_{B,p} \}$ constitutes a family of seminorms on $L(\mathcal
{S},\mathcal{S})$ and the topology given by these seminorms is
referred to as the strong topology on
$L(\mathcal{S},\mathcal{S})$.
\end{definition}

\subsubsection{The space of tempered distributions}

Many of the processes studied in this paper take values in the
topological dual of
$\mathcal{S}$, which we denote by $\mathcal{S}'$. Recall that
$\mathcal{S}'$ is the space of all
continuous linear functionals on $\mathcal{S}$. Elements of $\mathcal
{S}'$ are referred to as
\textit{tempered distributions} and we now review some relevant facts
concerning tempered distributions
as well as tempered distribution-valued processes.

For each $\mu\in\mathcal{S}'$ and $\varphi\in\mathcal{S}$, we
denote the
\textit{duality product} of $\mu$ and $\varphi$ by $\langle\mu,
\varphi\rangle\equiv\mu(\varphi)$. The \textit{distributional
derivative} of $\mu\in\mathcal{S}'$ is
denoted by $\mu'$ and is defined to be the
unique element of $\mathcal{S}'$ such that
\[
\bigl\langle\mu', \varphi\bigr\rangle= - \bigl\langle\mu,
\varphi' \bigr\rangle\qquad\mbox{for~all }\varphi\in\mathcal{S}.
\]
It is clear by the definition of
$\mathcal{S}$ that $\mu'$ is well defined. For each $\mu\in
\mathcal{S}'$ and $t \in\mathbb R$, we also define
$\tau_t \mu$ as the unique element of $\mathcal{S}'$ such that
\[
\langle\tau_t \mu, \varphi\rangle= \langle\mu, \tau_t
\varphi\rangle\qquad\mbox{for all }\varphi\in\mathcal{S},
\]
where $\tau_t \varphi\in\mathcal{S}$ is the function defined by
$\tau_t
\varphi(\cdot) \equiv\varphi( \cdot- t )$.

All statements in this paper regarding convergence in $\mathcal{S}'$
are with respect to the strong topology on $\mathcal{S}'$, which we
now define. One may consult Section~1.1 of \cite{KallianpurXiong}
for further details.

%
\begin{definition}\label{defstrongdual}
For each bounded subset $B \subset\mathcal{S}$, let
%
%
\begin{equation}
\label{strongdual} q_B(\mu) \equiv\sup_{\varphi\in B} \bigl
\llvert\langle\mu, \varphi\rangle\bigr\rrvert\qquad\mbox{for all }\mu
\in
\mathcal{S}'.
\end{equation}
Then the strong topology on $\mathcal{S}'$ is the topology
induced by the family of seminorms~$\{ q_B \}$.
\end{definition}

Unfortunately, the space $\mathcal{S}'$ is not metrizable
with respect to the strong topology (see Section~2 of \cite{KallianpurXiong}).
Nevertheless, as we discuss below, one may still usefully speak of weak
convergence of $\mathcal{S}'$-valued random elements and processes
taking values in $\mathcal{S}'$.

Let $\mathbb{D}([0,T],\mathcal{S}')$ denote the space of functions
from $[0,T]$ to $\mathcal{S}'$ that are right-continuous on $[0,T)$
with left limits everywhere\vspace*{1pt} on $(0,T]$. If $(\mu_t)_{t \geq0} \in
\mathbb{D}([0,T],\mathcal{S}') $ and $t \in[0,T]$,
we then define the tempered distribution $\int_0^t \mu_s \,ds $ to
be the unique element of
$\mathcal{S}'$ (see Section~2 of \cite{KallianpurPerezAbreuMap}) such that
\[
\biggl\langle\int_0^t \mu_s \,ds,
\varphi\biggr\rangle= \int_0^t \langle
\mu_s, \varphi\rangle \,ds\qquad\mbox{for all }t \geq0, \varphi\in
\mathcal{S}.
\]

As noted immediately following Definition~\ref{defstrongdual} above,
the space $\mathcal{S}'$ equipped with the strong topology is not
metrizable and so the Skorokhod metric and ensuing Skorokhod
$J_1$-topology may not defined on $\mathbb{D}([0,T],\mathcal{S}')$ in
the usual manner. We therefore follow the approach of \cite
{KallianpurXiong,Mitoma} in defining an appropriate topology on
$\mathbb{D}([0,T],\mathcal{S}')$. Let $\Lambda$ be the set of
strictly increasing continuous maps from $[0,T]$ onto itself such that
for each $\lambda\in\Lambda$,
\begin{eqnarray*}
\gamma(\lambda) &=& \sup_{0 \leq s <t \leq T} \biggl\llvert\ln\biggl(
\frac{\lambda_t-\lambda_s}{t-s} \biggr) \biggr\rrvert<\infty.
\end{eqnarray*}
We then have the following definition (see \cite{KallianpurXiong,Mitoma}).

%
\begin{definition}\label{defstrongskor}
For each seminorm $q_B$ defining the strong topology on $\mathcal
{S}'$, let
\begin{eqnarray*}
d_{q_B}^{o}(\mu,\nu) &=& \inf_{\lambda\in\Lambda}
\Bigl(\sup_{0
\leq t \leq T}\bigl\llvert q_B(
\mu_t-\nu_{\lambda_t}) +\gamma(\lambda)\bigr\rrvert\Bigr)\qquad
\mbox{for all }\mu,\nu\in\mathcal{S}'.
\end{eqnarray*}
The topology on $\mathbb{D}([0,T],\mathcal{S}')$ is then defined by
the family of pseudometrics~$\{d_{q_B}^{o}\}$.
\end{definition}

By part~(c) of Theorem 2.4.1 of \cite{KallianpurXiong}, the
topology given
in Definition~\ref{defstrongskor} above is equivalent to the
topology defined by the family of
pseudometrics $\{d_{q_B}\}$, where
\begin{eqnarray*}
d_{q_B}(\mu,\nu) &=& \inf_{\lambda\in\Lambda} \Bigl(\sup
_{0
\leq t \leq T}\bigl\llvert q_B(\mu_t-
\nu_{\lambda_t})\bigr\rrvert+\sup_{0
\leq t \leq T} \llvert
\lambda_t-t\rrvert\Bigr)\qquad\mbox{for all }\mu,\nu\in
\mathcal{S}'.
\end{eqnarray*}
We also note that under this topology, $\mathbb{D}([0,T],\mathcal
{S}')$ is a completely regular topological space \cite{Mitoma}.

The following result is an important consequence of Proposition 5.2 of
\cite{Mitoma} regarding weak convergence of processes taking values in
the dual of a nuclear Fr\`{e}chet space. It provides a convenient
characterization of weak convergence of processes taking values in
$\mathcal{S}'$.

%
\begin{teo}[(Mitoma's theorem)]\label{Mitoma}
Let $(\mu^n)_{n \geq
1}$ be a sequence of random elements of $\mathbb{D}([0,T], \mathcal
{S}')$. Then
\[
\mu^n \Rightarrow\mu\qquad\mbox{in } \mathbb{D}\bigl([0,T],
\mathcal{S}'\bigr)
\]
if the following two statements hold:
\begin{longlist}[(2)]
\item[(1)] For each $\varphi\in\mathcal{S}$, the sequence $\{ \langle\mu^n,
\varphi\rangle\}_{n \geq1}$ is tight in $\mathbb{D}([0,T], \mathbb R)$.
\item[(2)] For $\varphi_1, \ldots, \varphi_m \in\mathcal{S}$ and
$t_1,\ldots,t_m \in[0,T]$,
\begin{eqnarray*}
\bigl( \bigl\langle\mu^n_{t_1}, \varphi_1
\bigr\rangle, \ldots, \bigl\langle\mu^n_{t_m},
\varphi_m \bigr\rangle\bigr) &\Rightarrow& \bigl( \langle
\mu_{t_1}, \varphi_1 \rangle, \ldots, \langle
\mu_{t_m}, \varphi_m \rangle\bigr) \qquad\mbox{in }
\mathbb{R}^m.
\end{eqnarray*}
\end{longlist}
\end{teo}

We now conclude the technical background section with some comments
regarding martingales and, in particular, $\mathcal{S}'$-valued
martingales. Let $(\mathcal{F}_t )_{t \geq0}$ be a filtration on an underlying
probability space $(\Omega, \mathcal{F}, \mathbb{P})$ and let $M$
and $N$ be two $\mathbb{R}$-valued
$\mathcal{F}_t $-martingales. The quadratic
covariation of $M$ and $N$ is
denoted by $(\langle M,N\rangle_t)_{t \geq0}$ and the quadratic
variation of $M$ is
denoted by \mbox{$(\ll M\gg_t)_{t \geq
0} \equiv(\langle M,M\rangle_t)_{t \geq0}$}. An $\mathcal{S}'$-valued
process $M$
is said to be an \textit{$\mathcal{S}'$-valued
$\mathcal{F}_t$-martingale} if for all $\varphi\in\mathcal{S}$,
$(\langle
M_t, \varphi\rangle)_{t \geq0}$ is an $\mathbb R$-valued
$\mathcal{F}_t$-martingale. For two $\mathcal{S}'$-valued martingales
$M$ and $N$, their \textit{tensor quadratic
covariation} $(\langle M,N\rangle_t)_{t \geq0}$ is given for all $t
\geq0$ and
all $\varphi, \psi\in\mathcal{S}$ by
\[
\langle M,N\rangle_t(\varphi, \psi) \equiv \bigl\langle\langle
M_\cdot, \varphi\rangle, \langle N_\cdot, \psi
\rangle\bigr\rangle_t,
\]
and the \textit{tensor quadratic variation} \mbox{$(\ll M\gg_t)_{t \geq0}$} of
an $\mathcal{S}'$-valued martingale is given by
\mbox{$(\ll M\gg_t)_{t \geq0} \equiv(\langle M,M\rangle_t)_{t \geq0}$}. Two
$\mathcal{S}'$-valued martingales, $M$ and $N$, are said to be
\textit{orthogonal} if $\langle M, N\rangle=0$ identically.
Corresponding notions
for the optional quadratic variation process $[M]$ are defined
analogously.\looseness=1

\section{System equations}\label{SecSystemEquations}

In this section, we obtain semi-martingale decompositions of the
tempered distribution-valued age process $\mathcal{A}\equiv( \mathcal
{A}_t)_{ t
\geq0 }$ and the tempered distribution-valued residual service time
process $\mathcal{R}\equiv(
\mathcal{R}_t )_{t \geq0 }$. We begin in Section~\ref{SubsecAgeEquations} by treating the
age process $\mathcal{A}$ and then move
on in Section~\ref{SubsecResidualEquations} to treating the residual
service time process $\mathcal{R}$.

\subsection{Age process}\label{SubsecAgeEquations}

We consider a $G/\mathit{GI}/\infty$ queue with general arrival process $(E_t) _{t
\geq0} \in\mathbb{D}([0,\infty),\mathbb R)$. We assume that
$E_0=0$, $\mathbb{P}$-a.s., and, for convenience in our proofs, we
also define $E_{t}=0$ for $t < 0$. We also make the assumption that for
each $t \geq0$, we have that $\mathbb{E}[E_t^2] < \infty$. Next, for
each $i \geq1$, we denote by
\begin{eqnarray*}
\tau_i &=& \inf\{t \geq0 \dvtx  E_t \geq i\}
\end{eqnarray*}
the time of the arrival of the $i$th customer to the system after time
$t=0$. We assume that $\mathbb{E}[\tau_i] < \infty$ for each
$i=1,2,\ldots.$ We denote by $\eta_i$ the service time of the $i$th customer
to arrive to the system after time $t=0$ and we assume that $\{\eta_i,
i \geq1\}$ is an i.i.d. sequence of nonnegative, mean 1 random
variables with cumulative distribution function (c.d.f.) $F$,
complementary cumulative distribution function (c.c.d.f.) $\bar{F} = 1 -
F$, and
probability density function (p.d.f.) $f$. We also assume that the hazard
rate function $h$ of $F$ satisfies the following assumption.

\begin{assumption}\label{boundedhazardassumption}
The function $h \in C_b^{\infty}(\mathbb{R}_{+})$.
\end{assumption}

%

Now let $(A_t)_{t \geq0} \in\mathbb{D}([0,\infty),\mathbb{D})$ be
such that for each $t \geq0$
and $y \geq0$, the quantity $A_t(y)$ represents
the number of customers in the system at time $t \geq0$ that have
been in the system for less than or equal to $y$ units of time
at time $t$. For $y < 0$, we set $A_t(y)=0$. At time $t=0$, we assume
that there are $A_{0}(y)$
customers present who have been in the system for less than or equal to
$y \geq0$
units of time and that there are a total of $A_0(\infty)$ customers
present. We assume that $\mathbb{E}[A_0^2(\infty)] < \infty$. For
each $i = 1,\ldots,A_0(\infty)$, we denote by
\[
\tilde{\tau}_i =-\inf\bigl\{ y \geq0 \dvtx  A_0(y) \geq i
\bigr\}
\]
the ``arrival'' time of the $i$th initial customer to the system. We
denote by
$\tilde{\eta}_i$ the remaining service time at time $t=0$ of the $i$th
initial customer in the system. The distribution of $\tilde{\eta}_i$,
conditional on the arrival time $\tilde{\tau}_i$, is given by
%
%
\begin{equation}
\label{initialservicetimes} \mathbb P( \tilde{\eta}_i > x | \tilde{
\tau}_i ) = \frac
{1-F(-\tilde{\tau}_i + x)}{1-F(-\tilde{\tau}_i) },\qquad x \geq0.
\end{equation}
We denote by $f_{\tilde{\tau}_i}$ the
conditional p.d.f. associated with this
distribution and we set $h_{\tilde{\tau}_i}(\cdot)=h(\cdot-\tilde
{\tau}_i)$.

We now derive a convenient representation for the system equations for
$(A_t)_{t \geq0}$
and its tempered distribution-valued counterpart, $\mathcal{A}$, which
we define shortly. We begin by noting that by first
principles we have that for each $t \geq0$ for $y \geq0$,
%
%
\begin{equation}
\label{QueueFirstPrinciples} A_t(y) = \sum_{i=1}^{A_0(\infty)}
\mathbf{1}_{ \{ t - \tilde{\tau
_i} \leq y \} } \mathbf{1}_{ \{ t < \tilde{\eta}_i \} } + \sum
_{i=1}^{E_t} \mathbf{1}_{ \{ t - \tau_i \leq y \} }
\mathbf{1}_{ \{ t - \tau_i < \eta_i \} }.
\end{equation}
Our first result provides an alternative way to write
(\ref{QueueFirstPrinciples}). In the following, we set~$\sum_{i=1}^{0}=0$.

%
\begin{prop}\label{PropEquivalentAges}
For each $t \geq0$ and $y \geq0$,
%
%
\begin{eqnarray} \label{AgesDDForm}
A_t(y) &=& A_0(y) - \sum_{i=1}^{A_0(\infty)}
\mathbf{1}_{ \{ \tilde
{\eta}_i \leq t \wedge(y+\tilde{\tau}_i) \} } - \sum_{i=A_0(y-t)+1}^{A_0(y)}
\mathbf{1}_{ \{ \tilde{\eta}_i > y +\tilde
{\tau}_i \} }
\nonumber\\[-8pt]\\[-8pt]
&&{}+ E_t -\sum_{i=1}^{E_t}
\mathbf{1}_{ \{ \eta_i \leq(t - \tau_i) \wedge y \} } - \sum_{i=1}^{E_{(t-y)-}}
\mathbf{1}_{ \{ \eta_i > y \} }.\nonumber
\end{eqnarray}
\end{prop}

\begin{pf}
By (\ref{QueueFirstPrinciples}), we have that
%
%
\begin{eqnarray} \label{finalthirdqueue}
A_t(y) &=& \sum_{i=1}^{A_0(y)}
\mathbf{1}_{ \{ t - \tilde{\tau}_i
\leq y \} } \mathbf{1}_{ \{ t < \tilde{\eta}_i \} } + \sum
_{i=1}^{E_t} \mathbf{1}_{ \{ t - \tau_i \leq y \} }
\mathbf{1}_{ \{
t - \tau_i < \eta_i \} }
\nonumber
\\
&=& A_0(y) + \sum_{i=1}^{A_0(y)}(
\mathbf{1}_{ \{ t - \tilde{\tau
}_i \leq y \} } \mathbf{1}_{ \{ t < \tilde{\eta}_i \} }-1) + E_t
\\
&&{}+ \sum_{i=1}^{E_t}( \mathbf{1}_{ \{ t - \tau_i \leq y \} }
\mathbf{1}_{ \{ t - \tau_i < \eta_i \} }-1).\nonumber
\end{eqnarray}
However,
%
%
\begin{eqnarray} \label{doubledecomptwo}
&& 1- \mathbf{1}_{ \{ t - \tilde{\tau}_i \leq y \} }\mathbf{1}_{ \{
t < \tilde{\eta}_i \} }\nonumber
\\
&&\qquad = \mathbf{1}_{ \{ t - \tilde{\tau}_i \leq y \} }\mathbf{1}_{ \{
\tilde{\eta}_i \leq t \} } + \mathbf{1}_{ \{ t-\tilde{\tau}_i > y
\} }
\\
&&\qquad = (\mathbf{1}_{ \{ t - \tilde{\tau}_i \leq y \} }\mathbf{1}_{ \{
\tilde{\eta}_i \leq t \} } +
\mathbf{1}_{ \{ t - \tilde
{\tau}_i > y \} }\mathbf{1}_{ \{ - \tilde{\tau}_i+ \tilde{\eta}_i
\leq y \} } ) + \mathbf{1}_{ \{ t - \tilde{\tau}_i > y \}
}\mathbf{1}_{ \{ -\tilde{\tau}_i + \tilde{\eta}_i > y \} },
\nonumber
\end{eqnarray}
and, similarly,
%
%
\begin{eqnarray}\label{doubledecompone}
&& 1- \mathbf{1}_{ \{ t - \tau_i \leq y \} } \mathbf{1}_{ \{ t -
\tau_i < \eta_i \} } \nonumber
\\
&&\qquad = \mathbf{1}_{ \{ t - \tau_i \leq y \} } \mathbf{1}_{ \{ \eta_i
\leq t - \tau_i \} } +
\mathbf{1}_{ \{ t - \tau_i > y \} }
\\
&&\qquad = (\mathbf{1}_{ \{ t - \tau_i \leq y \} } \mathbf{1}_{ \{
\eta_i \leq t - \tau_i \} } +
\mathbf{1}_{ \{ t - \tau_i > y \} } \mathbf{1}_{ \{ \eta_i \leq y \} }
) + \mathbf{1}_{ \{ t -
\tau_i > y \} }
\mathbf{1}_{ \{ \eta_i > y \} }.
\nonumber
\end{eqnarray}
Substituting (\ref{doubledecompone}) and (\ref{doubledecomptwo})
into (\ref{finalthirdqueue}) and summing over $A_0(y)$ and $E_t$
completes the proof.
\end{pf}

We now provide an intuitive description of each of the terms appearing in
(\ref{AgesDDForm}). The first term represents the number of customers
in the system at time $t=0$ that have been in the system for less than
or equal to $y$ units of time, the second term represents the number
of departures by time $t \geq0$ of those initial customers that had total
service less than or equal to $y$ units of time at time $t=0$, and the third
term represents the number of initial customers whose total service
time is greater than $y$
units of time and had been in the system for less than or equal to $y$
units of time at time $t=0$ but have been in the system for greater than $y$ units of time at
time $t \geq0$. The fourth, fifth and sixth terms represent similar
quantities but for those customers that arrived to the system after time
$t=0$.

Now let $ D^0 = (D_t^0)_{t \geq0} \in
\mathbb{D}([0,\infty),\mathbb{D})$ be defined by setting
%
%
\begin{equation}
\label{D0def} D_t^0(y) = \sum
_{i=1}^{A_0(\infty)} \biggl( \mathbf{1}_{ \{ \tilde
{\eta}_i \leq t \wedge(y + \tilde{\tau}_i) \} } - \int
_0^{\tilde
{\eta}_i \wedge t \wedge(y + \tilde{\tau}_i)} h_{\tilde{\tau
}_i}(u) \,du \biggr),
\end{equation}
for $t \geq0$, $y \geq0$, and set $D_t^0(y)=0$ for $y < 0$. Also, let
$D = (D_t)_{t \geq0} \in\mathbb{D}([0,\infty),\mathbb{D})$ be
defined by setting
%
%
\begin{equation}
\label{Ddef}
\qquad D_t(y) = \sum_{i=1}^{E_t}
\biggl( \mathbf{1}_{ \{ \eta_i \leq(t -
\tau_i) \wedge y \} } - \int_0^{ \eta_i \wedge(t-\tau_i) \wedge y
}
h(u) \,du \biggr),\qquad t \geq0, y \geq0,
\end{equation}
and set $D_t(y)=0$ for $y < 0$.
It then follows from (\ref{AgesDDForm}) that for each $t \geq0$ and
$y \geq0$, we may write
%
%
\begin{eqnarray}
A_t(y) &=& A_0(y) + E_t -
D_t^0(y)- D_t(y)
\nonumber
\\
&&{}- \sum_{i=1}^{A_0(\infty)} \int
_{0}^{\tilde{\eta}_i \wedge t
\wedge(y + \tilde{\tau_i}) } h_{\tilde{\tau_i}}(u) \,du -\sum
_{i=1}^{E_t} \int_0^{\eta_i \wedge(t - \tau_i) \wedge y}
h(u) \,du \label{AgesDDForm2}
\\
&&{}- \sum_{i=A_0(y-t)+1}^{A_0(y)} \mathbf{1}_{ \{ -\tilde{\tau}_i +
\tilde{\eta}_i > y \} }
- \sum_{i=1}^{E_{(t-y)-}} \mathbf{1}_{ \{
\eta_i > y \} }.
\nonumber
\end{eqnarray}

The above expression for $A_t(y)$ will become useful in a moment.
However, we next move on to expressing the age process as a tempered
distribution-valued process using the
Schwartz space $\mathcal{S}$ defined in (\ref{schwartz}). In
particular, we associate
with the process $A$ defined in (\ref{QueueFirstPrinciples}) the
$\mathcal{S}'$-valued process $\mathcal{A}=(\mathcal{A}_{t})_{t \geq
0}$ such that for each $t \geq0$ and $\varphi\in
\mathcal{S}$ we set
%
%
\begin{equation}
\label{DPhiAssociation} \langle\mathcal{A}_t,\varphi\rangle= \int
_{\mathbb{R}} \varphi(y) \,dA_t(y).
\end{equation}
In a similar manner, we associate the $\mathcal{S}'$-valued processes
$\mathcal{D}^0=(\mathcal{D}_{t}^0)_{t \geq0}$ and $\mathcal
{D}=(\mathcal{D}_{t})_{t \geq0}$ with $D^{0}$ and $D$, respectively.
That is, for each $t \geq0$ and $\varphi\in\mathcal{S}$ we set
\[
\bigl\langle\mathcal{D}_t^0,\varphi\bigr\rangle= \int
_{\mathbb{R}} \varphi(y) \,dD_t^0(y)\quad
\mbox{and}\quad\langle\mathcal{D}_t,\varphi\rangle= \int
_{\mathbb{R}} \varphi(y) \,dD_t(y).
\]
We also associate the $\mathcal{S}'$-valued random
variable $\mathcal{A}_0$ with $A_0$ by setting
\[
\langle\mathcal{A}_0,\varphi\rangle= \int_{\mathbb{R}}
\varphi(y) \,dA_0(y),\qquad\varphi\in\mathcal{S}.
\]
It is straightforward to see that for each $t \geq0$,
the quantities $\mathcal{A}_t$, $\mathcal{D}^0_t$ and $\mathcal{D}_t$
are well-defined elements of $\mathcal{S}'$. Moreover, since for each
fixed $\varphi\in\mathcal{S}$
the sample paths of $(\langle\mathcal{A}_t, \varphi\rangle)_{t \geq
0}$, $(\langle\mathcal{D}^0_t, \varphi\rangle)_{t \geq0}$ and
$(\langle\mathcal{D}_t, \varphi\rangle)_{t \geq0}$ all lie in
$\mathbb{D}([0,\infty),\mathbb{R})$, $\mathbb{P}$-a.s., it follows
that $\mathcal{A},\mathcal{D}^0,\mathcal{D} \in
\mathbb{D}([0,\infty),\mathcal{S}')$, $\mathbb{P}$-a.s.

For\vspace*{1pt} the remainder of the paper, we now replace the hazard rate function
$h\dvtx \mathbb{R}_{+} \mapsto\mathbb{R}$ with a function $\tilde
{h}\dvtx \mathbb R\rightarrow\mathbb R$ such that $\tilde{h}(x) = h(x)$
for $x \geq0$ and $\tilde{h} \in C_b^\infty(\mathbb{R})$. In a\vspace*{-2pt}
similar manner, we replace the c.d.f. $F$ and c.c.d.f. $\bar{F}$ with
corresponding functions $\tilde{F}$ and $\hspace*{2pt}\tilde{\hspace*{-2pt}\bar{F}}$ such that
$\tilde{F},\hspace*{2pt}\tilde{\hspace*{-2pt}\bar{F}} \in C_b^\infty(\mathbb{R})$. For ease
of notation,\vspace*{-1pt} we continue to refer to $\tilde{h}$, $\tilde{F}$ and
$\hspace*{2pt}\tilde{\hspace*{-2pt}\bar{F}}$ as $h$, $F$ and $\bar{F}$, respectively.

Our next step is to use the expression (\ref{AgesDDForm2}) in order to
provide a convenient
expression for the tempered distribution-valued process $\mathcal{A}$.
We begin by noting that integrating test functions $\varphi\in
\mathcal{S}$ term-by-term in (\ref{AgesDDForm2}) it follows that for
each $t \geq0$ one has that
%
%
\begin{eqnarray}
\langle\mathcal{A}_t, \varphi\rangle&=& \langle\mathcal{A}_0,
\varphi\rangle- \bigl\langle\mathcal{D}_t^0 +
\mathcal{D}_t, \varphi\bigr\rangle
\nonumber
\\
&&{}- \sum_{i=1}^{A_0(\infty)} \int
_{-\tilde{\tau_i}}^{-\tilde{\tau
_i} + ( \tilde{\eta}_i \wedge t ) } \varphi(y) h(y) \,dy - \sum
_{i=1}^{E_t} \int_0^{\eta_i \wedge(t - \tau_i)}
\varphi(y) h(y) \,dy \label{AgesPhiForm}
\\
&&{}- \int_{\mathbb R_+} \varphi(y) \,d \Biggl(\sum
_{i=A_0(y-t)+1}^{A_0(y)} \mathbf{1}_{ \{ -\tilde{\tau}_i + \tilde
{\eta}_i > y \} } + \sum
_{i=1}^{E_{(t-y)-}} \mathbf{1}_{ \{ \eta_i
> y \} }
\Biggr).
\nonumber
\end{eqnarray}

The following two propositions now allow us to further
simplify the expression in (\ref{AgesPhiForm}). We first have the following.

\begin{prop}\label{PropositionHazard}
For each $t \geq0$,
\[
\sum_{i=1}^{A_0(\infty)} \int_{-\tilde{\tau_i}}^{-\tilde{\tau_i}
+ ( \tilde{\eta}_i \wedge t ) }
\varphi(y) h(y) \,dy + \sum_{i=1}^{E_t} \int
_0^{\eta_i \wedge(t - \tau_i)} \varphi(y) h(y) \,dy = \int
_0^{t} \langle\mathcal{A}_s,
\varphi h \rangle \,ds.
\]
\end{prop}

\begin{pf}For each $t \geq0$,
\begin{eqnarray*}
&&\sum_{i=1}^{A_0(\infty)} \int
_{-\tilde{\tau_i}}^{-\tilde{\tau
_i} + ( \tilde{\eta}_i \wedge t ) } \varphi(y) h(y) \,dy + \sum
_{i=1}^{E_t} \int_0^{\eta_i \wedge(t - \tau_i)}
\varphi(y) h(y) \,dy
\\
&&\qquad =\sum_{i=1}^{A_0(\infty)} \int
_0^t \mathbf{1}_{ \{ 0 \leq s \leq
\tilde{\eta}_i \} } \varphi(s -
\tilde{\tau}_i ) h(s-\tilde{\tau}_i) \,ds
\\
&&\quad\qquad{}+ \sum_{i=1}^{E_t} \int
_0^t \mathbf{1}_{ \{ 0 \leq s - \tau_i
\leq\eta_i \} } \varphi( s -
\tau_i ) h(s-\tau_i) \,ds
\\
&&\qquad = \int_0^t \Biggl(\sum
_{i=1}^{A_0(\infty)} \mathbf{1}_{ \{ 0 \leq
s \leq\tilde{\eta}_i \} } \varphi(s -
\tilde{\tau}_i) h(s-\tilde{\tau}_i)
\\
&&\hspace*{51pt}{} + \sum
_{i=1}^{E_t} \mathbf{1}_{ \{ 0 \leq s - \tau_i
\leq\eta_i \} } \varphi(s -
\tau_i) h(s-\tau_i) \Biggr) \,ds
\\
&&\qquad = \int_0^t \langle\mathcal{A}_s,
\varphi h \rangle \,ds.
\end{eqnarray*}
This completes the proof.
\end{pf}

Next, we have the following.

%
\begin{prop}\label{PropositionDeriv}
For each $t \geq0$,
\begin{eqnarray*}
&&-\int_{\mathbb R_+} \varphi(y) \,d \Biggl( \sum
_{i=A_0(y-t)+1}^{A_0(y)} \mathbf{1}_{ \{ -\tilde{\tau}_i + \tilde
{\eta}_i > y \} }+ \sum
_{i=1}^{E_{(t-y)-}} \mathbf{1}_{ \{ \eta_i >
y \} } \Biggr)
\\
&&\qquad = E_t \varphi(0) + \int_0^t
\bigl\langle\mathcal{A}_s, \varphi' \bigr\rangle \,ds.
\end{eqnarray*}
\end{prop}

\begin{pf}Let $t \geq0$. Then,
integrating by parts we have that
\begin{eqnarray*}
&&- E_t \varphi(0) - \int_{\mathbb R_+} \varphi(y) \,d
\Biggl(\sum_{i=A_0(y-t)+1}^{A_0(y)} \mathbf{1}_{ \{ -\tilde{\tau}_i +
\tilde
{\eta}_i > y \} }
+ \sum_{i=1}^{E_{(t-y)-}} \mathbf{1}_{ \{ \eta_i
> y \} }
\Biggr)
\\
&&\qquad = \int_{\mathbb R_+} \Biggl(\sum_{i=A_0(y-t)+1}^{A_0(y)}
\mathbf{1}_{ \{ -\tilde{\tau}_i + \tilde{\eta}_i > y \} }+\sum
_{i=1}^{E_{(t-y)-}}
\mathbf{1}_{ \{ \eta_i > y \} } \Biggr) \varphi'(y) \,dy
\\
&&\qquad = \int_{\mathbb R_+} \Biggl(\sum_{i=1}^{A_0(\infty)}
\mathbf{1}_{
\{ \tilde{\tau}_i \geq-y, -\tilde{\tau}_i+\tilde{\eta}_i > y,
\tilde{\tau}_i + y < t \} } + \sum_{i=1}^{E_t}
\mathbf{1}_{ \{ \eta
_i > y, \tau_i + y < t \} } \Biggr) \varphi'(y) \,dy
\\
&&\qquad = \sum_{i=1}^{A_0(\infty)} \int
_{\mathbb R_+} \mathbf{1}_{ \{
\tilde{\tau}_i \geq-y, -\tilde{\tau}_i+\tilde{\eta}_i > y,
\tilde{\tau}_i + y < t \} } \varphi'(y) \,dy
\\
&&\quad\qquad{}+ \sum_{i=1}^{E_t} \int
_{\mathbb R_+} \mathbf{1}_{ \{ \eta_i > y,
\tau_i + y < t \} } \varphi'(y) \,dy
\\
&&\qquad = \sum_{i=1}^{A_0(\infty)} \int
_0^{t} \mathbf{1}_{ \{ 0 \leq s -
\tilde{\tau}_i \leq-\tilde{\tau}_i+ \tilde{\eta}_i \} } \varphi
'(s-\tilde{\tau}_i) \,ds
\\
&&\quad\qquad{}+ \sum_{i=1}^{E_t} \int
_0^t \mathbf{1}_{ \{ 0 \leq s - \tau_i
\leq\eta_i \} }
\varphi'(s-\tau_i) \,ds
\\
&&\qquad = \int_0^t \Biggl(\sum
_{i=1}^{A_0(\infty)} \mathbf{1}_{ \{ 0 \leq
s - \tilde{\tau}_i \leq-\tilde{\tau}_i+\tilde{\eta}_i \}
}\varphi'(s-\tilde{\tau}_i) + \sum
_{i=1}^{E_t} \mathbf{1}_{ \{ 0
\leq s - \tau_i \leq\eta_i \} }
\varphi'(s-\tau_i) \Biggr) \,ds
\\
&&\qquad = \int_0^t \biggl( \int_{\mathbb R_+}
\varphi'(u)\,d\mathcal{A}_s(u) \biggr) \,ds
\\
&&\qquad = \int_0^t \bigl\langle\mathcal{A}_s,
\varphi' \bigr\rangle \,ds.
\end{eqnarray*}
This completes the proof.
\end{pf}

Now note that combining Propositions~\ref{PropositionHazard} and
\ref{PropositionDeriv} with system equation (\ref{AgesPhiForm}), one
finds that for each $t \geq0$ and $\varphi\in\mathcal{S}$,
%
%
\begin{eqnarray}\label{AgesSystemEquation}
\langle\mathcal{A}_t, \varphi\rangle &=& \langle\mathcal{A}_0, \varphi\rangle+ \bigl\langle\mathcal{E}_t -
\mathcal{D}_t^0- \mathcal{D}_t, \varphi\bigr\rangle
\nonumber\\[-8pt]\\[-8pt]
&&{}- \int_0^t \langle\mathcal{A}_s, h \varphi\rangle \,ds + \int_0^t \bigl\langle
\mathcal{A}_s, \varphi' \bigr\rangle \,ds,\nonumber
\end{eqnarray}
where we define the $\mathcal{S}'$-valued process $\mathcal
{E}=(\mathcal{E}_t)_{t \geq0}$ to be such that $\langle\mathcal
{E}_t, \varphi\rangle=
E_t \varphi(0)$ for each $\varphi\in\mathcal{S}$ and $t \geq0$. In
general, we refer to (\ref{AgesSystemEquation}) as the semi-martingale
decomposition of $\mathcal{A}$. This will become clear in
Section~\ref{SecMartingale} where we show that the process $( \mathcal{D}_t^0+
\mathcal{D}_t)_{t \geq0}$ is a martingale.

\subsection{Residual service time process}\label{SubsecResidualEquations}

We next move on to analyzing the residual
service time process $\mathcal{R}$. As in Section~\ref{SubsecAgeEquations},
we assume that we have a $G/\mathit{GI}/\infty$ queue in which customers arrive
to the system according to a general arrival process
$(E_t) _{t \geq0} \in\mathbb{D}([0,\infty),\mathbb{R})$, where we
assume that $E_0=0$, $\mathbb{P}$-a.s. For each $i \geq1$, we denote by
\begin{eqnarray*}
\tau_i &=& \inf\{t \geq0 \dvtx  E_t \geq i\}
\end{eqnarray*}
the time of the $i$th customer arrival to the system after time $t=0$
and we let
$\eta_i$ be the service time
of the $i$th customer to arrive to the system after time
$t=0$. We assume that $\{\eta_i, i \geq1\}$ is an i.i.d. sequence of
nonnegative, mean 1 random variables with cumulative distribution
function $F$ and probability density function~$f$. We also assume in
this subsection
that the hazard rate function $h$ of $F$ satisfies Assumption~\ref
{boundedhazardassumption} of Section~\ref{SubsecAgeEquations}.

Now let $R=(R_t)_{t \geq0} \in\mathbb{D}([0,\infty),\mathbb{D})$
be such that for each $t \geq0$ and $y \in\mathbb{R}$, the quantity
$R_t(y)$ denotes the number of customers
in the system at time $t \geq0$ that
have less than or equal to $y \in\mathbb R$ units of service
remaining. For $y < 0$, we interpret $R_t(y)$ as the number of
customers who have departed from
the system by time $t+y$. Thus, $(R_t(y))_{y \in\mathbb{R}}$ not only
keeps tracks of the\vadjust{\goodbreak} residual service times of those customers present
in the system at time $t$, but it also records the departure times of
all customers who have departed from the system by time $t$. We assume
that at time $t=0$ there are $R_0(y)$ customers in the system that have less
than or equal to $y$ units of service time remaining. By first
principles, it then follows that for each $t \geq0$ and $y \in\mathbb
{R}$ we may write
%
%
\begin{equation}
\label{ResidualFirstPrinciple} R_t(y) = R_0(t+y) + \sum
_{i=1}^{E_t} \mathbf{1}_{ \{ \eta_i - (t -
\tau_i) \leq y \} }.
\end{equation}

The following proposition now presents an alternative expression for
the right-hand side of
(\ref{ResidualFirstPrinciple}).

%
\begin{prop}\label{PropAlternativeFirstPrinciple}
For each $t \geq0$ and $y \in\mathbb{R}$,
\begin{eqnarray}
R_t(y)
\nonumber
&=& R_0(y) + \bigl(R_0(t+y)-R_0(y)
\bigr)
\nonumber\\[-8pt]\\[-8pt]\nonumber
&&{} + \sum_{i=1}^{E_t} \mathbf{1}_{ \{ \eta_i \leq y \} }
+ \sum_{i=1}^{E_t} \mathbf{1}_{ \{ \eta_i > y, (\tau_i+\eta_i) - t \leq y
\} }.
\label{ResidualsDDForm}
\end{eqnarray}
\end{prop}

\begin{pf}
By (\ref{ResidualFirstPrinciple}),
%
%
\begin{eqnarray}\label{FirstEqPropResidualFirst}
R_t(y) &=& R_0(t+y) + \sum_{i=1}^{E_t}
\mathbf{1}_{ \{ \eta_i - (t -
\tau_i) \leq y \} }
\nonumber\\[-8pt]\\[-8pt]\nonumber
&=& R_0(y)+ \bigl(R_0(t+y)-R_0(y) \bigr) +
\sum_{i=1}^{E_t} \mathbf{1}_{ \{ \eta_i - (t - \tau_i) \leq y \}}.
\end{eqnarray}
However, note that
%
%
\begin{equation}
\label{indicatorSub} \mathbf{1}_{ \{ \eta_i - (t - \tau_i) \leq y \} }
= \mathbf{1}_{ \{
\eta_i \leq y \} } +
\mathbf{1}_{ \{ \eta_i > y, \eta_i - (t - \tau
_i) \leq y \} }.
\end{equation}
Substituting (\ref{indicatorSub}) into
(\ref{FirstEqPropResidualFirst}) and summing over $E_t$, completes
the proof.
\end{pf}

We now give an intuitive description for each of the terms appearing in
(\ref{ResidualsDDForm}). The first term represents the number of
customers in the system at time $t=0$ with less than or equal to $y$
units of service time remaining. The second term represents the number
of customers in the system at time $t=0$ with greater than $y$ units of
total service time but at time $t \geq0$ have less than or equal to $y$
units of service time remaining. The third and fourth terms in (\ref
{ResidualsDDForm})
have analogous descriptions but for those customers that arrive to the system
after time $t=0$.

Now let $G = (G_t)_{t \geq0}
\in\mathbb{D}([0,\infty),\mathbb{D})$ be defined by setting
\[
G_t(y) = \sum_{i=1}^{E_t}
\bigl(\mathbf{1}_{ \{ \eta_i \leq y \}
}-F(y) \bigr),\qquad t \geq0, y \in\mathbb{R}.\vadjust{\goodbreak}
\]
By Proposition~\ref{PropAlternativeFirstPrinciple}, it then follows
that for each $t \geq0$
and $y \in\mathbb{R}$, $R_t(y)$ may be written as
%
%
\begin{eqnarray}\label{ResidualsDDForm2}
R_t(y) &=& R_0(y) + \bigl(R_0(t+y)-R_0(y)
\bigr) + G_t(y)+ E_t F(y)
\nonumber\\[-8pt]\\[-8pt]\nonumber
&&{} + \sum_{i=1}^{E_t} \mathbf{1}_{ \{ \eta_i > y, (\tau_i+\eta_i) - t \leq y\} }.
\end{eqnarray}

The above representation for $R_t(y)$ will be useful in a moment.
However, we first proceed to define tempered distribution-valued
versions of the processes defined above. In particular, we let
$\mathcal{R}=(\mathcal{R}_t)_{t \geq0}$ be the $\mathcal
{S}'$-valued process associated with $R$ such
that for each $t \geq0$ and $\varphi\in\mathcal{S}$ we have that
\[
\langle\mathcal{R}_t,\varphi\rangle= \int_{\mathbb{R}}
\varphi(y) \,dR_t(y)
\]
and we let
$\mathcal{G}=(\mathcal{G}_t)_{t \geq0}$ be the $\mathcal
{S}'$-valued process associated with $G$ such
that for each $t \geq0$ and $\varphi\in\mathcal{S}$ we have that
\[
\langle\mathcal{G}_t,\varphi\rangle= \int_{\mathbb{R}}
\varphi(y) \,dG_t(y).
\]
We also define the elements $\mathcal{F}$ and $\mathcal{F}_e$ of
$\mathcal{S}'$ by setting
\begin{eqnarray*}
\langle\mathcal{F}, \varphi\rangle&\equiv& \int_0^{\infty
}
\varphi(x)\,dF(x),\qquad\varphi\in\mathcal{S},
\end{eqnarray*}
and
\begin{eqnarray*}
\langle\mathcal{F}_e, \varphi\rangle&\equiv& \int
_0^{\infty
}\varphi(x) \bigl(1-F(x)\bigr)\,dx,\qquad\varphi
\in\mathcal{S}.
\end{eqnarray*}

Now note that integrating test functions $\varphi\in\mathcal{S}$
against each of the terms in~(\ref{ResidualsDDForm2}), it follows that for each $\varphi\in
\mathcal{S}$ and $t \geq0$,
%
%
\begin{eqnarray}\label{ResidualsPhiForm}
\qquad \langle\mathcal{R}_t, \varphi\rangle &=& \langle
\mathcal{R}_0,\varphi\rangle+ \int_\mathbb R
\varphi(y) \,d\bigl(R_0(t+y)-R_0(y)\bigr)
\nonumber\\[-8pt]\\[-8pt]\nonumber
&&{} + \langle\mathcal{G}_t, \varphi\rangle+ E_t \langle
\mathcal{F}, \varphi\rangle+ \int_\mathbb R\varphi(y) \,d \Biggl(
\sum_{i=1}^{E_t} \mathbf{1}_{ \{ \eta_i > y, (\tau_i+\eta_i) - t
\leq y \} }
\Biggr).
\end{eqnarray}

The following proposition now allows us to simplify the right-hand side of~(\ref{ResidualsPhiForm}).

\begin{prop}\label{PropositionDeriv2}
For each $t \geq0$,
%
%
\begin{eqnarray}\label{ResidualProp}
\qquad && \int_\mathbb R\varphi(y) \,d\bigl(R_0(t+y)-R_0(y)
\bigr) + \int_\mathbb R\varphi(y)\, d \Biggl( \sum
_{i=1}^{E_t} \mathbf{1}_{ \{ \eta_i >
y, (\tau_i+\eta_i) - t \leq y \} } \Biggr)
\nonumber\\[-8pt]\\[-8pt]\nonumber
&&\qquad = -\int_0^t \bigl\langle
\mathcal{R}_s, \varphi' \bigr\rangle \,ds.
\end{eqnarray}
\end{prop}

\begin{pf}
The proof parallels the proof of Proposition
\ref{PropositionDeriv}. Let $t \geq0$. Then, integrating by parts,
we have that
\begin{eqnarray*}
&&- \int_\mathbb R \varphi(y) \,d\bigl(R_0(t+y) -
R_0(y)\bigr) - \int_\mathbb R \varphi(y) \,d \Biggl(
\sum_{i=1}^{E_t} \mathbf{1}_{ \{ \eta_i >
y, (\tau_i+\eta_i) - t \leq y \} }
\Biggr)
\\
&&\qquad = \int_\mathbb R \Biggl(\sum_{i=1}^{R_0(\infty)}
\mathbf{1}_{ \{
y < \tilde{\eta}_i \leq t + y \} } \Biggr) \varphi'(y) \,dy
+ \int_\mathbb R \Biggl(\sum_{i=1}^{E_t}
\mathbf{1}_{ \{ \eta_i >
y, (\tau_i+\eta_i) - t \leq y \} } \Biggr) \varphi'(y) \,dy
\\
&&\qquad = \sum_{i=1}^{R_0(\infty)} \int
_\mathbb R \mathbf{1}_{ \{ y <
\tilde{\eta}_i \leq t + y \} } \varphi'(y)
\,dy + \sum_{i=1}^{E_t} \int
_\mathbb R \mathbf{1}_{ \{ \eta_i > y, (\tau_i+\eta_i) - t
\leq y \} } \varphi'(y)
\,dy
\\
&&\qquad = \sum_{i=1}^{R_0(\infty)} \int
_\mathbb R \mathbf{1}_{ \{ \tilde
{\eta}_i - t \leq y < \tilde{\eta}_i \} } \varphi'(y)
\,dy + \sum_{i=1}^{E_t} \int
_\mathbb R \mathbf{1}_{ \{ (\tau_i+\eta_i) - t
\leq y < \eta_i \} } \varphi'(y)
\,dy
\\
&&\qquad =7 \sum_{i=1}^{R_0(\infty)} \int
_0^t \varphi'( \tilde{
\eta}_i - s ) \,ds + \sum_{i=1}^{E_t}
\int_0^{t} \mathbf{1}_{ \{ \tau_i
\leq s \} }
\varphi'(\tau_i+\eta_i-s) \,ds
\\
&&\qquad = \int_0^t \Biggl(\sum
_{i=1}^{R_0(\infty)} \varphi'( \tilde{\eta
}_i - s ) + \sum_{i=1}^{E_t}
\mathbf{1}_{ \{ \tau_i \leq s \} } \varphi'(\tau_i+
\eta_i-s) \Biggr) \,ds
\\
&&\qquad = \int_0^t \biggl(\int_{\mathbb R_+}
\varphi' \,dR_s(u) \biggr) \,ds
\\
&&\qquad = \int_0^t \bigl\langle\mathcal{R}_s,
\varphi' \bigr\rangle \,ds.
\end{eqnarray*}
This completes the proof.
\end{pf}

Substituting (\ref{ResidualProp}) into (\ref{ResidualsPhiForm}), one
now obtains that for each $t \geq0$ and $\varphi\in\mathcal{S}$,
%
%
\begin{equation}
\label{ResidualsSystemEquation} \langle\mathcal{R}_t, \varphi\rangle=
\langle
\mathcal{R}_0,\varphi\rangle+ \langle\mathcal{G}_t,
\varphi\rangle+ E_t \langle\mathcal{F}, \varphi\rangle-\int
_0^t \bigl\langle\mathcal{R}_s,
\varphi' \bigr\rangle \,ds.
\end{equation}
We refer to (\ref{ResidualsSystemEquation}) as the semi-martingale
decomposition of $\mathcal{R}$. This is due to the fact that in
Section~\ref{SecMartingale} it
will be shown that the process $\mathcal{G}$ in
(\ref{ResidualsSystemEquation}) is a martingale. We also point out the
similarity of (\ref{ResidualsSystemEquation}) with (4) of
\cite{DM08b}.

\section{Regulator map result}\label{SecRegulatorMap}

Let $B\dvtx \mathcal{S} \mapsto\mathcal{S}$ be a continuous linear
operator and for each $ \mu\in\mathbb{D}([0,T],\mathcal{S}')$,
consider the solution $\nu\in\mathbb{D}([0,T],\mathcal{S}')$ to the
integral equation
%
%
\begin{equation}
\label{integralequation} \langle\nu_t, \varphi\rangle= \langle
\mu_t, \varphi\rangle+ \int_0^t
\langle\nu_s, B \varphi\rangle \,ds,\qquad t \in[0,T], \varphi\in
\mathcal{S}.
\end{equation}
In this section, we first show that under
some mild restrictions on $B$, (\ref{integralequation}) defines a
continuous function $\Psi_B\dvtx  \mathbb{D}([0,T],\mathcal{S}')
\rightarrow\mathbb{D}([0,T],\mathcal{S}')$ mapping $ \mu$ to $\nu
$. We then proceed in
Sections~\ref{secageregulator}~and~\ref{secresidualsregulator}
to study particular continuous linear
operators associated with the age process and the residual service time
process, respectively.

We begin with the following definition from
\cite{KallianpurPerezAbreuMap}.

%
\begin{definition}\label{defco1strong}
A family $( S_t )_{t \geq0}$ of linear operators on $\mathcal{S}$ is said
to be a strongly-continuous $(C_0,1)$ semigroup if the following three
conditions are satisfied:
\begin{longlist}[(3)]
\item[(1)] $S_0 = I$,\label{C01part1} where $I$ is the identity operator, and, for all
$s,t \geq0$, $S_sS_t = S_{s+t}$.
\item[(2)] The\label{C01part2} map $t \rightarrow S_t $ is
continuous in the strong topology of $L(\mathcal{S},\mathcal{S})$.
That is, if $t_n \rightarrow t$ in
$\mathbb R_{+}$, then for any bounded subset $K \subset\mathcal{S}$
and $p \geq
1$,
\[
\sup_{\varphi\in K} \llVert S_{t_n} \varphi- S_t
\varphi\rrVert_p \rightarrow0.
\]
\item[(3)] For\label{C01part3} each $q \geq0$, there exist numbers $M_q$, $\sigma_q$ and
$p \geq q$ such that for all $\varphi\in\mathcal{S}$ and $t \geq0$,
\[
\llVert S_t \varphi\rrVert_q \leq M_q
e^{\sigma_q t} \llVert\varphi\rrVert_p.
\]
\end{longlist}
\label{C01def}
\end{definition}

%
\begin{rem}
Note that condition~\hyperref[C01part2]{(2)} of Definition~\ref{defco1strong}
is stronger
than the corresponding condition for a $(C_0,1)$ semigroup as given,
for example,
in \cite{KallianpurPerezAbreuMap}. The weaker
definition in \cite{KallianpurPerezAbreuMap} only requires that the
map $t \rightarrow S_t \varphi$ be
continuous in the weak topology of $L(\mathcal{S},\mathcal{S})$. That
is, if
$t_n \rightarrow t$ in $\mathbb R_{+}$, then for all $\varphi\in
\mathcal{S}$
and $m \geq1$, $\llVert S_{t_n} \varphi- S_t \varphi\rrVert _m
\rightarrow
0$.
\end{rem}

Recall now that for a family $( S_t )_{t \geq0}$ of linear operators
on $\mathcal{S}$, the
infinitesimal generator $B$ of $( S_t )_{t \geq0}$ is defined to be
such that
\begin{eqnarray*}
B \varphi&=& \lim_{t \rightarrow0} \frac{S_t \varphi- \varphi}{t} \qquad
\mbox{in }
\mathcal{S},
\end{eqnarray*}
for all such $\varphi\in\mathcal{S}$ that the limit on the
right-hand side above exits. We refer to such $\varphi\in\mathcal
{S}$ as $\mathcal{D}(B)$, the domain of $B$. We now have the following
result, which is our main result regarding the integral equation (\ref
{integralequation}).

%
\begin{teo}\label{ThmRegulatorMap}
Let $B \in L(\mathcal{S},\mathcal{S})$ be the infinitesimal generator
of a strong\-ly-continuous
$(C_0,1)$ semigroup $( S_t )_{t \geq0}$. Then, for each $
\mu\in\mathbb{D}([0,T],\mathcal{S}')$, the equation
{(\ref{integralequation})} has a unique solution given by
%
%
\begin{equation}
\label{integralsolution} \langle\nu_t, \varphi\rangle= \langle
\mu_t, \varphi\rangle+ \int_0^t
\langle\mu_s, S_{t-s} B \varphi\rangle \,ds,\qquad t \in[0,T],
\varphi\in\mathcal{S}.
\end{equation}
Furthermore, {(\ref{integralsolution})} defines a continuous function
$\Psi_B\dvtx \mathbb{D}([0,T],\mathcal{S}') \rightarrow
\mathbb{D}([0,T],\break  \mathcal{S}')$ mapping $ \mu$ to $\nu$.
\end{teo}

\begin{pf}
That $\Psi_B$ is a well-defined function from $ \mathbb
{D}([0,T],\mathcal{S}')$ to
$\mathbb{D}([0,T],\break \mathcal{S}')$ and the form of the solution
(\ref{integralsolution}) follows from Theorem 2.1 of
\cite{KallianpurPerezAbreuMap} (see also Corollary 2.2 of \cite
{KallianpurPerezAbreuMap}).

We now show that $\Psi_B$ is continuous. By (\ref{integralsolution}),
it suffices to show that the function mapping $\mathbb
{D}([0,T],\mathcal{S}')$ to $\mathbb{D}([0,T],\mathcal{S}')$ defined
by $\mu\mapsto
\int_0^\cdot B^* S_{\cdot- s}^* \mu_s \,ds$, where $B^*$ and $S_t^{*}$
denote the adjoint operators of $B$ and $S_t$,
respectively, is continuous. Let $( \mu^n )_{n \geq1}$ be a
sequence converging to $\mu$ in $\mathbb{D}([0,T],\mathcal{S}')$.
Then, by Definition~\ref{defstrongskor} and the comment below it
(see also Theorem 2.4.1 of \cite{KallianpurXiong}), there exists a
sequence $(\lambda^n)_{n \geq1}$ of strictly increasing
homeomorphisms of the interval $[0,T]$ such
that for each bounded set $K \subset\mathcal{S}$,
%
%
\begin{eqnarray}\label{Dconvergeregulator}
\sup_{0 \leq t \leq T} \sup_{\varphi\in K}\bigl
\llvert\bigl\langle\mu_t^n - \mu_{\lambda^n_t},
\varphi\bigr\rangle\bigr\rrvert\rightarrow0\quad\mbox{and}\quad\sup
_{0 \leq t \leq T} \bigl\llvert\lambda^n_t - t
\bigr\rrvert\rightarrow0
\nonumber\\[-12pt]\\[-12pt]
\eqntext{\mbox{as }n \rightarrow\infty.}
\end{eqnarray}
Moreover, it suffices to
consider homeomorphisms $(\lambda^n)_{n \geq1}$ that are absolutely
continuous with respect to Lebesgue measure on $[0,T]$ having
corresponding derivatives $(\dot{\lambda}^n)_{n \geq1}$ satisfying
$\llVert
\dot{\lambda}^n - 1 \rrVert _T \rightarrow0$ as $n \rightarrow
\infty$
(see pages 112--114 of~\cite{Billingsley99}). It then follows that for each
$t \in[0,T]$ and $\varphi\in\mathcal{S}$ we may write
%
%
\begin{eqnarray} \label{strongcontinequality}
&&\biggl\llvert\biggl\langle\int_0^t B^*
S_{t - s}^* \mu^n_s \,ds - \int
_0^{\lambda^n_t} B^* S_{\lambda^n_t - s}^*
\mu_s \,ds, \varphi\biggr\rangle\biggr\rrvert
\nonumber
\\
&&\qquad  = \biggl\llvert\biggl\langle\int_0^t B^*
S_{t - s}^* \mu^n_s \,ds - \int
_0^t B^* S_{\lambda^n_t - \lambda^n_s}^* \mu_{\lambda^n_s}
\dot{\lambda}^n_s \,ds, \varphi\biggr\rangle\biggr
\rrvert\nonumber
\\
&&\qquad  \leq \bigl\llVert\dot{\lambda}^n - 1 \bigr\rrVert
_T \biggl\llvert\biggl\langle\int_0^t
B^* S_{\lambda^n_t - \lambda^n_s}^* \mu_{\lambda
^n_s} \,ds, \varphi\biggr\rangle\biggr
\rrvert\nonumber
\\
&&\quad\qquad{}  + \biggl\llvert\biggl\langle\int_0^t
B^* \bigl( S_{t-s}^* - S_{\lambda^n_t - \lambda^n_s }^* \bigr) \mu
_{\lambda^n_s}
\,ds, \varphi\biggr\rangle\biggr\rrvert
\\
&&\quad\qquad{}+ \biggl\llvert\biggl\langle\int_0^t B^*
S_{t-s}^* \bigl( \mu^n_s - \mu
_{\lambda^n_s} \bigr) \,ds, \varphi\biggr\rangle\biggr\rrvert
\nonumber
\\
&&\qquad  = \bigl\llVert\dot{\lambda}^n - 1 \bigr\rrVert_T
\biggl\llvert\int_0^t \langle
\mu_{\lambda^n_s}, S_{\lambda^n_t - \lambda^n_s} B \varphi\rangle \,ds
\biggr\rrvert
\nonumber
\\
&&\qquad\quad{} + \biggl\llvert\int_0^t \bigl
\langle\mu_{\lambda^n_s}, ( S_{t-s} - S_{\lambda^n_t - \lambda^n_s } )
B \varphi
\bigr\rangle \,ds \biggr\rrvert
+ \biggl\llvert\int_0^t
\bigl\langle\mu^n_s - \mu_{\lambda
^n_s},
S_{t-s} B \varphi\bigr\rangle \,ds \biggr\rrvert.\hspace*{-15pt}\nonumber
\end{eqnarray}
In order to complete the proof, it now suffices to show that for each
bounded subset $K \subset\mathcal{S}$, the following three limits hold:
%
%
\begin{eqnarray}
\qquad \sup_{t \in[0,T]} \sup_{\varphi\in K} \bigl\llVert\dot{
\lambda}^n - 1 \bigr\rrVert_T \biggl\llvert\int
_0^t \langle\mu_{\lambda^n_s},
S_{\lambda^n_t - \lambda^n_s} B \varphi\rangle \,ds \biggr\rrvert &\rightarrow&0 \qquad\mbox{as }n \rightarrow\infty\label{limit1},
\\
\sup_{t \in[0,T]} \sup_{\varphi\in K} \biggl\llvert\int
_0^t \bigl\langle\mu_{\lambda^n_s}, (
S_{t-s} - S_{\lambda^n_t - \lambda
^n_s } ) B \varphi\bigr\rangle \,ds\biggr\rrvert
&\rightarrow&0 \qquad\mbox{as }n \rightarrow\infty\label{limit2},
\\
\sup_{t \in[0,T]} \sup_{\varphi\in K} \biggl\llvert\int
_0^t \bigl\langle\mu^n_s
- \mu_{\lambda^n_s}, S_{t-s} B \varphi\bigr\rangle \,ds \biggr\rrvert
&\rightarrow&0 \qquad\mbox{as }n \rightarrow\infty\label{limit3}.
\end{eqnarray}

We begin with (\ref{limit1}).
First note that for each bounded subset $K \subset\mathcal{S}$,
%
%
\begin{equation}
\label{strongbounded1} \sup_{s \in[0,T]} \sup_{\varphi\in K} \bigl
\llvert\langle\mu_s, \varphi\rangle\bigr\rrvert< \infty.
\end{equation}
Now let $K \subset\mathcal{S}$ be an arbitrary bounded set. We show
that the set $K' \equiv\{ S_u B \varphi, u \in
[0,T], \varphi\in K \}$ is bounded in $\mathcal{S}$ as well. Indeed,
note that by Definition~\ref{defco1strong}, we have that for each
$q \geq0$ there exist
$M_q$, $\sigma_q$ and $p \geq q$ such that
%
%
\begin{equation}
\label{strongbounded2} \sup_{\varphi\in K'}\llVert\varphi\rrVert
_{q}= \sup_{u \in
[0,T],\varphi\in K} \llVert S_u B
\varphi\rrVert_q \leq M_q e^{\sigma_q T} \sup
_{\varphi\in K} \llVert B \varphi\rrVert_p < \infty,
\end{equation}
where the final inequality follows from the fact that continuous
linear operators from $\mathcal{S}$ to $\mathbb{R}$ map bounded sets
of $\mathcal{S}$ to bounded sets of $\mathbb{R}$ (see Theorem 1.2.1 of~\cite{KallianpurXiong}). Thus, $K' = \{ S_u B \varphi, u \in
[0,T], \varphi\in K \}$ is bounded and so by (\ref{strongbounded1}),
\begin{eqnarray*}
\sup_{s \in[0,T]} \sup_{u \in[0,T],\varphi\in K } \bigl\llvert\langle
\mu_s, S_u B \varphi\rangle\bigr\rrvert&<&\infty.
\end{eqnarray*}
It therefore follows since $\llVert \dot{\lambda}^n - 1 \rrVert _T
\rightarrow0 $ as $n \rightarrow\infty$, that
%
%
\begin{eqnarray}\label{limit1proof}
&& \sup_{t \in[0,T]} \sup_{\varphi\in K} \bigl\llVert
\dot{\lambda}^n - 1 \bigr\rrVert_T \biggl\llvert\int
_0^t \langle\mu_{\lambda^n_s},
S_{\lambda^n_t - \lambda^n_s} B \varphi\rangle \,ds \biggr\rrvert
\nonumber\\[-8pt]\\[-8pt]\nonumber
&& \qquad \leq\bigl\llVert\dot{\lambda}^n - 1 \bigr\rrVert
_T T \sup_{s
\in[0,T]} \sup_{u \in[0,T],\varphi\in K }
\bigl\llvert\langle\mu_s, S_u B \varphi\rangle
\bigr\rrvert \,ds \rightarrow0,
\end{eqnarray}
as $n \rightarrow\infty$, which implies (\ref{limit1}).

We next prove the limit (\ref{limit2}). First note that by Lemma 2.2 of
\cite{KallianpurSurvey}, there exist $\theta\geq0$ and $q \geq1$
such that
%
%
\begin{eqnarray}\label{eq2}
&& \sup_{t \in[0,T]} \sup_{\varphi\in K} \biggl\llvert
\int_0^t \bigl\langle\mu_{\lambda^n_s}, (
S_{t-s} - S_{\lambda^n_t - \lambda
^n_s } ) B \varphi\bigr\rangle \,ds\biggr\rrvert\nonumber
\\
&&\qquad \leq T \sup_{s\in[0,T]}\sup_{0 \leq w \leq v \leq T } \sup
_{\varphi\in K} \bigl\llvert\bigl\langle\mu_s,
(S_{v - w} - S_{\lambda
^n_v - \lambda^n_w}) B \varphi\bigr\rangle\bigr\rrvert
\\
&&\qquad \leq T \theta\sup_{0 \leq w \leq v \leq T } \sup_{\varphi\in K} \bigl
\llVert(S_{v - w} - S_{\lambda^n_v -
\lambda^n_w}) B \varphi\bigr\rrVert
_q.\nonumber
\end{eqnarray}
Next, note that for $w \leq v$ we may write
\begin{eqnarray*}
S_{v - w} - S_{\lambda^n_v -
\lambda^n_w} &=& \mathbf{1}_{ \{ (v - w)-(\lambda^n_v - \lambda
^n_w) \geq0 \} }
(S_{(v - w)-(\lambda^n_v -
\lambda^n_w) }-I ) S_{\lambda^n_v -
\lambda^n_w}
\\
&&{}+\mathbf{1}_{ \{ (\lambda^n_v - \lambda^n_w)-(v - w) \geq0 \} }
(I-S_{(\lambda^n_v -
\lambda^n_w)-(v - w)}) S_{v -w}.
\end{eqnarray*}
Thus, recalling the definition of the set $K' = \{ S_u B \varphi, u
\in
[0,T], \varphi\in K \}$, it follows that
%
%
\begin{eqnarray}\label{useimply}
&& \sup_{0 \leq w \leq v \leq T } \sup_{\varphi\in K} \bigl\llVert
(S_{v
- w} - S_{\lambda^n_v -
\lambda^n_w}) B \varphi\bigr\rrVert_q\nonumber
\\
&&\qquad \leq \sup_{0 \leq w \leq v \leq T } \sup_{\varphi\in K} \bigl\llVert
\mathbf{1}_{ \{ (v - w)-(\lambda^n_v - \lambda^n_w) \geq0 \} } (S_{(v -
w)-(\lambda^n_v -
\lambda^n_w) }-I ) S_{\lambda^n_v -
\lambda^n_w}B \varphi\bigr
\rrVert_q\nonumber
\\
&&\quad\qquad{}+ \sup_{0 \leq w \leq v \leq T } \sup_{\varphi\in K} \bigl\llVert
\mathbf{1}_{ \{ (\lambda^n_v - \lambda^n_w)-(v - w) \geq0 \} }
(I-S_{(\lambda^n_v -
\lambda^n_w)-(v - w)}) S_{v -
w}B \varphi\bigr
\rrVert_q
\\
&&\qquad \leq \sup_{0 \leq w \leq v \leq T } \sup_{\varphi\in K'} \bigl\llVert
\mathbf{1}_{ \{ (v - w)-(\lambda^n_v - \lambda^n_w) \geq0 \} } (S_{(v -
w)-(\lambda^n_v -
\lambda^n_w) }-I ) \varphi\bigr\rrVert
_q\nonumber
\\
&&\quad\qquad{}+ \sup_{0 \leq w \leq v \leq T } \sup_{\varphi\in K'} \bigl\llVert
\mathbf{1}_{ \{ (\lambda^n_v - \lambda^n_w)-(v - w) \geq0 \} }
(I-S_{(\lambda^n_v -
\lambda^n_w)-(v - w)}) \varphi\bigr\rrVert
_q.\hspace*{-15pt}\nonumber
\end{eqnarray}
However, note that since $\llVert \dot{\lambda}^n-1\rrVert _T
\rightarrow0$ as $n \rightarrow\infty$, it follows that
\begin{eqnarray*}
\sup_{0 \leq w \leq v \leq T}\bigl\llvert(v - w)-\bigl(\lambda^n_v
- \lambda^n_w\bigr) \bigr\rrvert&\rightarrow&0\qquad
\mbox{as }n \rightarrow\infty.
\end{eqnarray*}
Thus, since by (\ref{strongbounded2}) the set $K' \subset\mathcal
{S}$ is bounded, it follows by part~\hyperref[C01part2]{(2)} of Definition~\ref{defco1strong} that
\begin{eqnarray*}
&& \sup_{0 \leq w \leq v \leq T } \sup_{\varphi\in K'} \bigl\llVert
\mathbf{1}_{ \{ (v - w)-(\lambda^n_v - \lambda^n_w) \geq0 \} } (S_{(v -
w)-(\lambda^n_v -
\lambda^n_w) }-I ) \varphi\bigr\rrVert
_q
\\
&&\quad {}+ \sup_{0 \leq w \leq v \leq T } \sup_{\varphi\in K'} \bigl\llVert
\mathbf{1}_{ \{ (\lambda^n_v - \lambda^n_w)-(v - w) \geq0 \} }
(I-S_{(\lambda^n_v -
\lambda^n_w)-(v - w)}) \varphi\bigr\rrVert
_q
\\
&&\qquad \rightarrow 0\qquad\mbox{as }n \rightarrow\infty,
\end{eqnarray*}
which, by (\ref{eq2}) and (\ref{useimply}), implies (\ref{limit2}).

Finally, (\ref{limit3}) follows from the fact that
%
%
\begin{eqnarray}\label{eq1}
&&\sup_{t \in[0,T]} \sup_{\varphi\in K} \biggl\llvert\int
_0^t \bigl\langle\mu^n_s
- \mu_{\lambda^n_s}, S_{t-s} B \varphi\bigr\rangle \,ds\biggr\rrvert
\nonumber\\[-8pt]\\[-8pt]\nonumber
&&\qquad \leq T \sup_{s \in[0,T]}\sup_{u \in[0,T], \varphi\in K} \bigl\llvert
\bigl\langle\mu^n_s - \mu_{\lambda^n_s},
S_u B \varphi\bigr\rangle\bigr\rrvert\rightarrow0,
\end{eqnarray}
as $n \rightarrow\infty$, where the final convergence follows from
(\ref{Dconvergeregulator}) and (\ref{strongbounded2}). This completes
the proof.
\end{pf}

\subsection{Age process}\label{secageregulator}
Now let $B^\mathcal{A}$ be the linear operator defined on
$\mathcal{S}$ such that
%
%
\begin{equation}
B^\mathcal{A} \varphi= \varphi' - h\varphi\qquad\mbox{for
all }\varphi\in\mathcal{S}. \label{defBA}
\end{equation}
Our main result in this subsection is to verify that $B^{\mathcal{A}}$
generates a strongly-continuous $(C_0,1)$ semigroup. This will then be
useful in Sections~\ref{SubsecAgesFluid}~and~\ref{SubsecAgesDiffusion}, where we prove our fluid and diffusion limits,
respectively, for the age process. In particular, by Theorem~\ref{ThmRegulatorMap} of the preceding subsection and (\ref{AgesSystemEquation}) of Section~\ref{SubsecAgeEquations}, this will
then allow us to write
%
%
\begin{equation}
\label{ageregulatorrep} \mathcal{A} = \Psi_{B^\mathcal{A}}\bigl(
\mathcal{A}_0+
\mathcal{E} - \bigl(\mathcal{D}^0 + \mathcal{D}\bigr) \bigr),
\end{equation}
where the map $\Psi_{B^\mathcal{A}}\dvtx  \mathbb{D}([0,T],\mathcal{S}')
\mapsto\mathbb{D}([0,T],\mathcal{S}')$
is a continuous map. We begin with the following lemma. Its proof may
be found in the \hyperref[appen]{Appendix}.

%
\begin{lem}\label{hazardlemma}
(1)~For each $n \geq1$ and $t \geq0$,
%
%
\begin{equation}
\label{ccdfratiobound} \sup_{x \geq0} \biggl\llvert\biggl(
\frac{\bar{F}(x+t)}{\bar{F}(x)} \biggr)^{(n)} \biggr\rrvert< \infty.
\end{equation}

(2)~For each $T > 0$, there exists a sequence $(M_n)_{ n \geq0 }$
such that for each \mbox{$s,t \in[0,T]$},
%
%
\begin{equation}
\label{ccdfratiobounddifference} \sup_{x \geq0} \biggl\llvert\biggl(
\frac{\bar{F}(x+t)}{\bar{F}(x)} - \frac{\bar{F}(x+s)}{\bar{F}(x)} \biggr
)^{(n)} \biggr\rrvert\leq
M_n \llvert t-s\rrvert.
\end{equation}
\end{lem}

\begin{pf}
See the \hyperref[appen]{Appendix}.
\end{pf}

Next, we have the following.

%
\begin{lem}\label{shiftboundlemma}
For each $\varphi\in\mathcal{S}$, $t \in\mathbb R$ and $\alpha,
\beta\in\mathbb N$, we have
%
%
\begin{equation}
\label{shiftboundeq} \llVert\tau_t \varphi\rrVert_{\alpha, \beta} \leq
\bigl( 1 \vee\llvert t\rrvert^\alpha\bigr) 2^\alpha\max
_{0 \leq i \leq\alpha} \llVert\varphi\rrVert_{i,\beta}.
\end{equation}
\end{lem}

\begin{pf}
For each $\varphi\in\mathcal{S}$, $t \in\mathbb{R}$ and $\alpha,
\beta\in
\mathbb N$, we have
\begin{eqnarray*}
\llVert\tau_t \varphi\rrVert_{\alpha, \beta} = \bigl\llVert
\varphi( \cdot- t ) \bigr\rrVert_{\alpha, \beta} &=& \sup_{x \in
\mathbb R}
\bigl\llvert x^\alpha\varphi^{(\beta)} (x - t) \bigr\rrvert
\\
&=& \sup_{x \in\mathbb R} \bigl\llvert\bigl[t + ( x - t )
\bigr]^\alpha\varphi^{(\beta)} (x - t) \bigr\rrvert
\\
&=& \sup_{x \in\mathbb R} \Biggl\llvert\sum
_{i=0}^\alpha\pmatrix{\alpha
\cr
i} t^{\alpha- i}
(x - t)^i \varphi^{(\beta)}(x - t) \Biggr\rrvert
\\
& \leq&\sum_{i=0}^\alpha\pmatrix{\alpha
\cr
i} \llvert t\rrvert^{\alpha- i} \llVert\varphi\rrVert_{i,\beta} \leq
\bigl( 1 \vee\llvert t\rrvert^\alpha\bigr) 2^\alpha\max
_{0 \leq i \leq\alpha} \llVert\varphi\rrVert_{i,\beta}.
\end{eqnarray*}
This completes the proof.
\end{pf}

The following is now our main result of this subsection.

%
\begin{prop}\label{PropAgeC01}
The linear operator $B^{\mathcal{A}}$ defined by (\ref{defBA})
generates a strongly-continuous $(C_0,1)$ semigroup $(
S^{\mathcal{A}}_t )_{ t \geq0 }$ given by
%
%
\begin{equation}
\label{semigroupdef} S^{\mathcal{A}}_t \varphi= (1/\bar{F})
\tau_{-t} (\bar{F} \varphi)\qquad\mbox{for all } \varphi\in\mathcal{S}.
\end{equation}
%
\end{prop}

\begin{pf}
We first check that $B^{\mathcal{A}}$ is indeed the infinitesimal generator
of the semi-group given by (\ref{semigroupdef}). In order to do so, it
suffices to check that for each $\alpha,\beta\in\mathbb{N}$,
%
%
\begin{eqnarray}
\lim_{t \rightarrow0}\biggl\llVert\frac{S^{\mathcal{A}}_t
\varphi- \varphi}{t}-
\bigl(\varphi' - h \varphi\bigr) \biggr\rrVert
_{\alpha,\beta}&=& 0. \label{checkzeroalpha}
\end{eqnarray}
We begin by noting that for each $t \geq0, x \in\mathbb{R}$ and
$\varphi\in\mathcal{S}$, we have that
\begin{eqnarray*}
\bigl\llvert S^{\mathcal{A}}_t \varphi(x) \bigr\rrvert&=& \biggl
\llvert\frac
{1-F(x-t)}{1-F(x)} \cdot\tau_t \varphi(x) \biggr\rrvert\leq e
^{\llVert h\rrVert _{\infty}t} \cdot\bigl\llvert\tau_t \varphi
(x)\bigr\rrvert,
\end{eqnarray*}
and so it follows that $S^{\mathcal{A}}_t \varphi\in\mathcal{S}$.
Now let $x \in\mathbb{R}$ be fixed and note that for each $t \geq0$,
we may write
\begin{eqnarray*}
S^{\mathcal{A}}_t \varphi(x) &=& \exp\biggl(-\int
_x^{x+t}h(v)\,dv \biggr)\varphi(x+t).
\end{eqnarray*}
Hence, by Taylor's theorem, expanding in terms of $t$ we obtain that we
may write
%
%
\begin{eqnarray}
S^{\mathcal{A}}_t \varphi(x) &=& \varphi(x) + \bigl(\varphi
'(x)-h(x)\varphi(x)\bigr) t + R(x,t), \label{Taylorshow}
\end{eqnarray}
where the remainder term $R(x,t)$ has the from
\begin{eqnarray*}
R(x,t) &=&\frac{1}{2}\int_0^{t}
\frac{d^2}{du^2} \biggl(\exp\biggl(-\int_x^{x+u}h(v)\,dv
\biggr)\varphi(x+u) \biggr)u \,du.
\end{eqnarray*}
Now differentiating with respect to $x$ in (\ref{Taylorshow}), we
obtain that for each $\alpha,\beta\in\mathbb{N}$,
\begin{eqnarray*}
\lim_{t \rightarrow0}\biggl\llVert\frac{S^{\mathcal{A}}_t
\varphi- \varphi}{t}-
\bigl(\varphi' - h \varphi\bigr) \biggr\rrVert
_{\alpha,\beta}&=& \lim_{t \rightarrow0} \sup_{x \in\mathbb{R}}
\biggl\llvert x^{\alpha} \frac
{R^{(\beta)}(x,t)}{t} \biggr\rrvert,
\end{eqnarray*}
where $R^{(\beta)}(x,t)$ denotes the $\beta$th derivative of $R(x,t)$
with respect to $x$.
Hence, in order to verify (\ref{checkzeroalpha}) it now suffices to
show that for each $\alpha,\beta\in\mathbb{N}$,
\begin{eqnarray*}
\lim_{t \rightarrow0} \sup_{x \in\mathbb{R}}\biggl\llvert
x^{\alpha} \frac{R^{(\beta)}(x,t)}{t} \biggr\rrvert&=&0.
\end{eqnarray*}
First note that for each $x \in\mathbb{R}$ fixed, we may write
%
%
\begin{eqnarray}\label{showremainu}
\qquad &&x^{\alpha}R^{(\beta)}(x,t)
\nonumber\\[-8pt]\\[-8pt]\nonumber
&&\qquad =\frac{1}{2}\int_0^{t} x^{\alpha}
\biggl(\frac{d^2}{du^2} \biggl(\exp\biggl(-\int_x^{x+u}h(v)\,dv
\biggr)\varphi(x+u) \biggr) \biggr)^{(\beta)}u \,du.
\end{eqnarray}
However, since $h \in C_b(\mathbb{R}_{+})$ and $\varphi\in\mathcal
{S}$, it is straightforward to show that for each $\alpha,\beta\in
\mathbb{N}$, there exists a constant $C_{\alpha,\beta} < \infty$
such that for $u$ sufficiently small,
\[
\sup_{ x \in\mathbb{R}} \biggl\llvert x^{\alpha} \biggl(
\frac
{d^2}{du^2} \biggl(\exp\biggl(-\int_x^{x+u}h(v)\,dv
\biggr)\varphi(x+u) \biggr) \biggr)^{(\beta)} \biggr\rrvert
< C_{\alpha,\beta}.
\]
Hence, by (\ref{showremainu}) we obtain that
\begin{eqnarray*}
\lim_{t \rightarrow0} \sup_{x \in\mathbb{R}}\biggl\llvert
x^{\alpha} \frac{R^{(\beta)}(x,t)}{t} \biggr\rrvert&<&\lim_{t
\rightarrow0}
\biggl\llvert\frac{C_{\alpha,\beta}}{2t} \int_0^t u \,du
\biggr\rrvert = \lim_{t \rightarrow0} \frac{C_{\alpha,\beta
}}{4}t = 0.
\end{eqnarray*}
This now completes the verification of the fact that $B^{\mathcal{A}}$
is the infinitesimal generator
of the semigroup $( S^{\mathcal{A}}_t )_{t \geq0}$ given by (\ref
{semigroupdef}).

We next verify that the semigroup $( S^{\mathcal{A}}_t )_{t \geq0}$
is a strongly-continuous $(C_0,1)$ semi-group. It is straightforward to
see that part~\hyperref[C01part1]{(1)} of Definition~\ref{C01def} is satisfied.
We next check that part~\hyperref[C01part2]{(2)} of Definition~\ref{C01def} is
satisfied.
Consider $ 0 \leq s < t$, a~bounded set $K \subset\mathcal{S}$, and
$\alpha, \beta\in\mathbb N$. We then have that we may write
%
%
\begin{eqnarray}\label{longstuff}
&&\sup_{\varphi\in K} \bigl\llVert S_{s}^{\mathcal{A}}
\varphi- S_{t}^{\mathcal{A}} \varphi\bigr\rrVert_{\alpha, \beta}
\nonumber
\\
&&\qquad = \sup_{\varphi\in K} \bigl\llVert(1/\bar{F}) \tau_{-s}
( \bar{F} \varphi) - (1/\bar{F})\tau_{-t} ( \bar{F} \varphi) \bigr
\rrVert_{\alpha, \beta}
\nonumber
\\
&&\qquad = \sup_{\varphi\in K} \sup_{x \in\mathbb R} \biggl\llvert
x^\alpha\biggl(\frac{ \bar{F}(x + s) }{\bar{F}(x)} \varphi(x + s) -
\frac
{\bar{F}(x + t)}{\bar{F}(x)}
\varphi(x + t) \biggr)^{(\beta)} \biggr\rrvert
\\
&&\qquad \leq \sup_{\varphi\in K} \sup_{x \in\mathbb R} \biggl\llvert
x^\alpha\biggl( \biggl( \biggl(\frac{ \bar{F}(x + s) }{\bar{F}(x)}-
\frac
{\bar{F}(x+t)}{\bar{F}(x)}
\biggr)\varphi(x + s) \biggr)^{(\beta
)} \biggr) \biggr\rrvert
\nonumber
\\
&&\quad\qquad{} + \sup_{\varphi\in K} \sup_{x \in\mathbb R} \biggl\llvert
x^\alpha\biggl(\frac{\bar{F}(x + t)}{\bar{F}(x)} \bigl(\varphi(x+s) -
\varphi(x + t)
\bigr) \biggr)^{(\beta)} \biggr\rrvert. \nonumber
\end{eqnarray}
We now handle each of the terms in (\ref{longstuff}) separately.

For the first term in (\ref{longstuff}), note that by Lemmas~\ref
{hazardlemma} and~\ref{shiftboundlemma} we have that
%
%
\begin{eqnarray}\label{longstufffirstterm}
\qquad &&\sup_{\varphi\in K} \sup_{x \in\mathbb R} \biggl\llvert
x^\alpha\biggl[ \biggl( \biggl(\frac{ \bar{F}(x + s) }{\bar{F}(x)}-
\frac{\bar
{F}(x+t)}{\bar{F}(x)}
\biggr)\varphi(x + s) \biggr)^{(\beta)} \biggr] \biggr\rrvert
\nonumber
\\
&&\qquad = \sup_{\varphi\in K} \sup_{x \in\mathbb R} \sum
_{i=0}^\beta\pmatrix{\beta
\cr
i} \biggl\llvert
x^\alpha\biggl(\frac{ \bar{F}(x + s)
}{\bar{F}(x)} - \frac{\bar{F}(x+t)}{\bar{F}(x)}
\biggr)^{(\beta
-i)} \varphi^{(i)}(x+s) \biggr\rrvert
\nonumber\\[-8pt]\\[-8pt]\nonumber
&&\qquad =\llvert t-s\rrvert\sum_{i=0}^\beta
M_{\beta-i} \pmatrix{\beta
\cr
i} \sup_{\varphi\in K} \sup
_{x \in\mathbb R} \bigl\llvert x^\alpha\varphi^{(i)}(x+s)
\bigr\rrvert
\nonumber
\\
&&\qquad \leq\llvert t-s\rrvert\bigl( 1 \vee\llvert t\rrvert^\alpha\bigr)
2^\alpha\sum_{i=0}^\beta
M_{\beta-i} \pmatrix{\beta
\cr
i} \sup_{\varphi\in K} \max
_{0 \leq j \leq\alpha} \llVert\varphi\rrVert_{j, i}. \nonumber
\end{eqnarray}

We next focus on the second term in (\ref{longstuff}). For each $n
\geq1$, denote the left-hand side of (\ref{ccdfratiobound}) by
$L_n$. By the mean value theorem, for each $x \in\mathbb R$ there exists
an $r_x \in[s,t]$ such that
%
%
\begin{eqnarray}\label{longstuffsecondterm}
\qquad&&\sup_{\varphi\in K} \sup_{x \in\mathbb R} \biggl\llvert
x^\alpha\biggl(\frac{\bar{F}(x + t)}{\bar{F}(x)} \bigl(\varphi(x+s) -
\varphi(x + t)
\bigr) \biggr)^{(\beta)} \biggr\rrvert
\nonumber
\\
&&\qquad = \sup_{\varphi\in K} \sup_{x \in\mathbb R} \sum
_{i=0}^\beta\pmatrix{\beta
\cr
i} \biggl\llvert
x^\alpha\frac{\bar{F}(x + t)}{\bar
{F}(x)}^{(\beta-i)} \bigl(\varphi^{(i)}(x+s)
- \varphi^{(i)}(x + t) \bigr) \biggr\rrvert
\nonumber
\\
&&\qquad = \sup_{\varphi\in K} \sup_{x \in\mathbb R} \sum
_{i=0}^\beta\pmatrix{\beta
\cr
i} L_{\beta-i}
\bigl\llvert x^\alpha\bigl(\varphi^{(i)}(x+s) -
\varphi^{(i)}(x + t) \bigr) \bigr\rrvert
\\
&&\qquad \leq \llvert t - s \rrvert\sup_{\varphi\in K} \sup
_{x \in
\mathbb R} \sum_{i=0}^\beta
\pmatrix{\beta
\cr
i} L_{\beta-i} \bigl\llvert x^\alpha
\varphi^{(i+1)}(x+r_x) \bigr\rrvert
\nonumber
\\
&&\qquad \leq \llvert t - s \rrvert\bigl( 1 \vee\llvert t\rrvert^\alpha
\bigr) 2^\alpha\sum_{i=0}^\beta
\pmatrix{\beta
\cr
i} L_{\beta-i} \sup_{\varphi\in K} \max
_{0 \leq j \leq\alpha} \llVert\varphi\rrVert_{j,i+1},\nonumber
\end{eqnarray}
where the final inequality follows as a consequence of Lemma~\ref
{shiftboundlemma}.
Combining (\ref{longstufffirstterm}) and (\ref{longstuffsecondterm})
with (\ref{longstuff}) and taking the limit as $s \rightarrow t$
now yields part~\hyperref[C01part2]{(2)} of Definition~\ref{C01def}.

We now complete the proof by verifying that part~\hyperref[C01part3]{(3)} of
Definition~\ref{C01def} is satisfied. First note that for
each $\varphi\in\mathcal{S}$, $t \geq0$ and $\alpha, \beta
\in\mathbb N$, we may write
\begin{eqnarray*}
\bigl\llVert S_t^{\mathcal{A}} \varphi\bigr\rrVert_{\alpha, \beta}
&=& \sup_{x \in\mathbb R} \biggl\llvert x^\alpha\biggl(
\frac{ \bar{F}(x+t)
}{\bar{F}(x)} \varphi( x + t ) \biggr)^{(\beta)} \biggr\rrvert
\\
&=& \sup_{x \in\mathbb R} \Biggl\llvert x^\alpha\sum
_{i=0}^\beta\pmatrix{\beta
\cr
i} \biggl(
\frac{\bar{F}(x+t)}{\bar{F}(x)} \biggr)^{(\beta
-i)} \varphi^{(i)}(x+t) \Biggr\rrvert
\\
&\leq&\sum_{i=0}^\beta\pmatrix{\beta
\cr
i}
L_{\beta-i} \sup_{x
\in\mathbb R} \bigl\llvert x^\alpha
\varphi^{(i)}(x+t) \bigr\rrvert
\\
&\leq&\bigl( 1 \vee\llvert t\rrvert^\alpha\bigr) 2^\alpha\sum
_{i=0}^\beta\pmatrix{\beta
\cr
i}
L_{\beta- i } \max_{0 \leq j \leq\alpha} \llVert\varphi\rrVert
_{j, i }
\\
&\leq&\bigl( 1 \vee\llvert t\rrvert^\alpha\bigr) 2^{\alpha+ \beta} \max
_{0
\leq i \leq\beta} L_i \max_{0 \leq j \leq\alpha} \max
_{0 \leq i
\leq
\beta} \llVert\varphi\rrVert_{j, i },
\end{eqnarray*}
where the final inequality above follows from Lemma~\ref{shiftboundlemma}.
Part~\hyperref[C01part3]{(3)} of Definition~\ref{C01def} now follows from (\ref
{lem133}) and (\ref{lem134}) of Section~\ref{SubsecNotation}
and the fact that $\llVert \cdot\rrVert _p \leq\llVert \cdot
\rrVert _{p+1}$
for each $p \in\mathbb{N}$. This completes the proof.
\end{pf}

\subsection{Residual service time process}\label{secresidualsregulator}

Now let $B^{\mathcal{R}}$ be the linear operator defined on $\mathcal
{S}$ such that
%
%
\begin{equation}
B^{\mathcal{R}}\varphi= - \varphi'\qquad\mbox{for all } \varphi
\in\mathcal{S}. \label{deflinresdefine}
\end{equation}
In this subsection, we verify that $B^{\mathcal{R}}$ generates a
strongly-continuous
$(C_0,1)$ semi-group. This will be useful in Sections~\ref{SubsecResidualsFluid}~and~\ref{SubsecResidualsDiffusion},
where we prove our fluid and diffusion limits, respectively, for the
residual service time process. In particular, by (\ref
{ResidualsSystemEquation}) of Section~\ref{SubsecResidualEquations}
and Theorem~\ref{ThmRegulatorMap}, this then implies that we may write
%
%
\begin{equation}
\label{residualsregulatorrep} \mathcal{R} = \Psi_{B^{\mathcal{R}}}(
\mathcal{R}_0+ E
\mathcal{F}+\mathcal{G} ),
\end{equation}
where the map $\Psi_{B^\mathcal{R}}\dvtx  \mathbb{D}([0,T],\mathcal{S}')
\mapsto\mathbb{D}([0,T],\mathcal{S}')$
is a continuous map.

Our main result of this subsection is the following.

%
\begin{prop}\label{PropResidualC01}
The linear operator $B^{\mathcal{R}}$ defined by (\ref
{deflinresdefine}) generates the strongly-continuous $(C_0,1)$
semigroup $( \tau_t )_{t \geq0}$.
\end{prop}

\begin{pf}First note that it is clear by Lemma~\ref{shiftboundlemma}
that for each $t \geq0$ and $\varphi\in\mathcal{S}$, we have that
$\tau_t \varphi\in\mathcal{S}$. We next check that for
each $\alpha, \beta\in\mathbb{N}$, we have the convergence
%
%
\begin{eqnarray}
\lim_{t \rightarrow0}\biggl\llVert\frac{\tau_t \varphi-
\varphi
}{t} -
\bigl(-\varphi'\bigr) \biggr\rrVert_{\alpha,\beta}
&=&0. \label{sufficeeasygenerator}
\end{eqnarray}
This will then be sufficient to verify that $B^{\mathcal{R}}$ as
defined by (\ref{deflinresdefine}) generates the semigroup $( \tau
_t )_{t \geq0}$. Let $x \in\mathbb{R}$ and $\varphi\in\mathcal
{S}$ be fixed. It then follows
by Taylor's theorem, expanding in terms of $t$, that we may write
%
%
\begin{eqnarray}
\tau_t \varphi(x) &=& \varphi(x) - \varphi'(x)t +
R(x,t), \label{tautexpansion}
\end{eqnarray}
where the remainder term $R(x,t)$ is given by
%
%
\begin{eqnarray}
R(x,t) &=& \frac{1}{2} \int_0^{-t}
\varphi''(x+u) (-t-u) \,du. \label{formremainder}
\end{eqnarray}
Now differentiating in (\ref{tautexpansion}) with respect to $x$, we
obtain that
\begin{eqnarray*}
\lim_{t \rightarrow0}\biggl\llVert\frac{\tau_t \varphi-
\varphi
}{t} -
\bigl(-\varphi'\bigr) \biggr\rrVert_{\alpha,\beta}
&=& \lim_{t
\rightarrow0} \sup_{x \in\mathbb{R}}\biggl\llvert
x^{\alpha}\frac
{R^{(\beta)}(t,x)}{t} \biggr\rrvert,
\end{eqnarray*}
where $R^{(\beta)}(x,t)$ denotes the $\beta$th derivative of $R$ with
respect to $x$. Thus, in order to prove (\ref
{sufficeeasygenerator}), it now suffices to check that
%
%
\begin{eqnarray}
\lim_{t \rightarrow0} \sup_{x \in\mathbb{R}}\biggl\llvert
x^{\alpha
}\frac{R^{(\beta)}(t,x)}{t} \biggr\rrvert&=& 0. \label{limcheck}
\end{eqnarray}
First note that by (\ref{formremainder}) we may write
\begin{eqnarray*}
x^{\alpha}R^{(\beta)}(x,t) &=& \frac{x^{\alpha}}{2} \int
_0^{-t} \varphi^{(\beta+2)}(x+u) (-t-u) \,du
\\
&=&\frac{1}{2} \int_0^{-t}
x^{\alpha} \varphi^{(\beta
+2)}(x+u) (-t-u) \,du.
\end{eqnarray*}
However, by Lemma~\ref{shiftboundlemma} there exists a constant
$C_{\alpha,\beta} < \infty$ such that for sufficiently small $u$,
\begin{eqnarray*}
\sup_{x \in\mathbb{R}} \bigl\llvert x^{\alpha}
\varphi^{(\beta
+2)}(x+u)\bigr\rrvert&< &C_{\alpha,\beta}.
\end{eqnarray*}
We therefore obtain that
\begin{eqnarray*}
\lim_{t \rightarrow0} \sup_{x \in\mathbb{R}}\biggl\llvert
x^{\alpha
}\frac{R^{(\beta)}(t,x)}{t} \biggr\rrvert& < & \lim_{t \rightarrow0}
\biggl\llvert\frac{C_{\alpha,\beta}}{2t} \int_0^{-t}
(-t-u)\,du \biggr\rrvert = \lim_{t \rightarrow0} \frac{C_{\alpha,\beta
}}{4}t = 0,
\end{eqnarray*}
thus proving (\ref{limcheck}). This completes our verification of the
fact that $B^{\mathcal{R}}$ as defined by (\ref{deflinresdefine})
generates the semigroup $( \tau_t )_{t \geq0}$.

We next proceed to verify that $( \tau_t )_{ t \geq0 }$ is a
strongly-continuous $(C_0,1)$ semigroup.
In order to check this fact, we will verify that $( \tau_t )_{ t \geq
0 }$ satisfies parts~\hyperref[C01part1]{(1)} through~\hyperref[C01part3]{(3)} of
Definition~\ref{C01def}. It is clear that $( \tau_t )_{ t \geq0 }$
satisfies part~\hyperref[C01part1]{(1)} of Definition~\ref{C01def}. We next
check that $( \tau_t )_{ t \geq0 }$ satisfies part~\hyperref[C01part2]{(2)}
of Definition~\ref{C01def}. Let $s < t$, $K \subset\mathcal{S}$ be a
bounded set, and let $\alpha, \beta\in\mathbb N$. Then, by the mean
value theorem, for each $x \in\mathbb R$ there exists an $r_x \in
[x-t,x-s]$ such that\looseness=-1
\begin{eqnarray*}
&& \sup_{\varphi\in K} \llVert\tau_s \varphi-
\tau_t \varphi\rrVert_{\alpha, \beta}
\nonumber
\\
&&\qquad = \sup_{\varphi\in K} \sup_{x \in\mathbb R} \bigl\llvert
x^\alpha\bigl( \varphi^{(\beta)}(x - s) - \varphi^{(\beta)}(x
- t) \bigr) \bigr\rrvert
\\
&&\qquad = \sup_{\varphi\in K} \sup_{x \in\mathbb R} \bigl\llvert
\bigl( r_x + ( x - r_x ) \bigr)^\alpha\bigl(
\varphi^{(\beta)}(x - s) - \varphi^{(\beta)}(x - t) \bigr) \bigr\rrvert
\\
&&\qquad = \sup_{\varphi\in K} \sup_{x \in\mathbb R} \Biggl\llvert
\sum_{i=0}^\alpha\pmatrix{\alpha
\cr
i}
r_x^{\alpha- i} (x - r_x)^i \bigl(
\varphi^{(\beta)}(x - s) - \varphi^{(\beta)}(x - t) \bigr) \Biggr\rrvert
\\
&&\qquad = \llvert t - s \rrvert\sup_{\varphi\in K} \sup
_{x \in\mathbb
R} \Biggl\llvert\sum_{i=0}^\alpha
\pmatrix{\alpha
\cr
i} r_x^{\alpha- i} (x - r_x)^i
\varphi^{(\beta+1)}(x - r_x) \Biggr\rrvert
\\
&&\qquad \leq \llvert t - s \rrvert\sum_{i=0}^\alpha
\pmatrix{\alpha
\cr
i} (t+1)^{\alpha- i} \sup_{\varphi\in K} \llVert
\varphi\rrVert_{i,\beta+1}.
\end{eqnarray*}
Part~\hyperref[C01part2]{(2)} of Definition~\ref{C01def} now follows from the
bound (\ref{lem133}) and the fact that $K$ is a bounded set.

We now complete the proof by verifying that $(\tau_t)_{t \geq0}$
satisfies part~\hyperref[C01part3]{(3)} of Definition~\ref{C01def}. Note that
for each $\varphi\in\mathcal{S}$, $t \geq0$ and $\alpha, \beta\in
\mathbb N$, we have that
\begin{eqnarray*}
\llVert\tau_t \varphi\rrVert_{\alpha, \beta} &=& \bigl\llVert
\varphi( \cdot- t ) \bigr\rrVert_{\alpha, \beta}
\\
&=& \sup_{x \in\mathbb R} \bigl\llvert x^\alpha
\varphi^{(\beta)} (x - t) \bigr\rrvert
\\
&=& \sup_{x \in\mathbb R} \bigl\llvert\bigl[t + ( x - t )
\bigr]^\alpha\varphi^{(\beta)} (x - t) \bigr\rrvert
\\
&=& \sup_{x \in\mathbb R} \Biggl\llvert\sum
_{i=0}^\alpha\pmatrix{\alpha
\cr
i} t^{\alpha- i}
(x - t)^i \varphi^{(\beta)}(x - t) \Biggr\rrvert
\\
&\leq&\sum_{i=0}^\alpha\pmatrix{\alpha
\cr
i}
t^{\alpha- i} \llVert\varphi\rrVert_{i,\beta}
\\
&\leq&\bigl( 1 \wedge t^\alpha\bigr) 2^\alpha\max
_{0 \leq i \leq\alpha} \llVert\varphi\rrVert_{i,\beta}.
\end{eqnarray*}
Part~\hyperref[C01part3]{(3)} of Definition~\ref{C01def} now follows as a
result of the bounds (\ref{lem133}) and (\ref{lem134}).
\end{pf}

\section{Martingale results}\label{SecMartingale}

In this section, we show that the process $\mathcal{D}^0+\mathcal{D}$ defined
in Section~\ref{SubsecAgeEquations} and the process
$\mathcal{G}$ defined in Section~\ref{SubsecResidualEquations} are
both $\mathcal{S}'$-valued martingales. The fact that
$\mathcal{D}^0+\mathcal{D}$ is an $\mathcal{S}'$-valued martingale
will ultimately be used together with the
martingale functional central limit theorem \cite{EK86} and the
continuous mapping
theorem \cite{Billingsley99} in Sections~\ref{SubsecAgesFluid}~and~\ref{SubsecAgesDiffusion}
in order to prove our fluid and diffusion limits, respectively, for the
age process.
The fact that $\mathcal{G}$ is an \mbox{$\mathcal{S}'$-}valued martingale is
not necessarily needed in
order to prove limit theorems for the residual service time process but
may be used to show that the residual service time process is in fact a
Markov process.
We begin by studying $\mathcal{D}^0+\mathcal{D}$.

\subsection{Age process}

In this subsection, we show that the process $\mathcal{D}^0 + \mathcal{D}$
defined in Section~\ref{SubsecAgeEquations} is an $\mathcal{S}'$-valued
martingale with respect to
the filtration $(\mathcal{F}^{\mathcal{A}}_t)_{t \geq0}$ defined by
\begin{eqnarray*}
\mathcal{F}^{\mathcal{A}}_t &=& \sigma\bigl\{ \mathbf{1}_{ \{ \tilde
{\eta}_i \leq s - \tilde{\tau}_i \} },
s \leq t, i = 1,2,\ldots,A_0(\infty)\bigr\}
\\
&&{} \vee\sigma\bigl\{\tilde{
\tau}_i, i=1,\ldots,A_0(\infty)\bigr\}
\\
&&{}\vee\sigma\{ \mathbf{1}_{ \{ \eta_i \leq s - \tau_i \} }, s \leq t, i
= 1,2,\ldots,E_t\}
\\
&&{} \vee\sigma\{E_s, s \leq t\} \vee
\mathcal{N}.
\end{eqnarray*}
Moreover, we explicitly identify the tensor quadratic
variation of $\mathcal{D}^{0}+\mathcal{D}$. The following is our main
result of this subsection.

\begin{prop}\label{PropMartingales}
The process $\mathcal{D}^0 + \mathcal{D}$ is an
$\mathcal{S}'$-valued
$\mathcal{F}^{\mathcal{A}}_t$-martingale with tensor quadratic
variation process given for all $\varphi, \psi\in
\mathcal{S}$ by
%
%
\begin{eqnarray}\label{DD0QV}
\lll \mathcal{D}^0 + \mathcal{D} \rrr _t(
\varphi, \psi) &=& \sum_{i=1}^{A_0(\infty)} \int
_0^{ \tilde{\eta}_i \wedge t } \varphi( x - \tilde{\tau}_i )
\psi( x - \tilde{\tau}_i ) h_{\tilde{\tau
}_i}(x) \,dx
\nonumber\\[-8pt]\\[-8pt]\nonumber
&&{} + \sum_{i=1}^{E_t} \int
_0^{ \eta_i \wedge( t - \tau_i )^+ } \varphi(x) \psi(x) h(x) \,dx.
\end{eqnarray}
\end{prop}

\begin{pf}Let $\varphi\in\mathcal{S}$. We claim that $(\langle
(\mathcal{D}^0 + \mathcal{D} )_{t}, \varphi\rangle)_{t \geq0}$ is
an $\mathbb{R}$-valued \mbox{$\mathcal{F}^{\mathcal{A}}_t$-}martingale,
which is sufficient to show that $\mathcal{D}^0 + \mathcal{D}$ is an
\mbox{$\mathcal{S}'$-}valued
\mbox{$\mathcal{F}^{\mathcal{A}}_t$-}mar\-tingale. Let $t \geq0$. We first
show that $E[\llvert\langle(\mathcal{D}^0 + \mathcal{D})_t, \varphi
\rangle\rrvert] < \infty$. Note that by (\ref{D0def}) and (\ref
{Ddef}) we may write
%
%
\begin{eqnarray}\label{lessinfty}
&&\mathbb E \bigl[\bigl\llvert\bigl\langle\bigl(\mathcal{D}^0 +
\mathcal{D}\bigr)_t, \varphi\bigr\rangle\bigr\rrvert\bigr]\nonumber
\\
&&\qquad \leq \mathbb E \Biggl[ \Biggl\llvert\sum_{i=1}^{A_0(\infty)}
\int_0^{t} \varphi(x-\tilde{\tau}_i)
\,d \biggl(\mathbf{1}_{ \{ \tilde{\eta
}_i \leq x \} }- \int_0^{ \tilde{\eta}_i \wedge x }
h_{\tilde{\tau
}_i}(u) \,du \biggr) \Biggr\rrvert\Biggr]
\nonumber
\\
&&\quad\qquad{}+ \mathbb E \Biggl[ \Biggl\llvert\sum_{i=1}^{E_t}
\int_0^{(t - \tau
_i)^+} \varphi(x) \,d \biggl(
\mathbf{1}_{ \{ \eta_i \leq x \} }- \int_0^{ \eta_i \wedge x } h(u) \,du
\biggr) \Biggr\rrvert\Biggr]
\nonumber
\\
&&\qquad \leq \mathbb E \Biggl[ \sup_{0 \leq s < \infty}\bigl\llvert\varphi(s)
\bigr\rrvert\sum_{i=1}^{A_0(\infty)} \biggl(
\mathbf{1}_{ \{ \tilde
{\eta}_i \leq t \} } + \int_0^{ \tilde{\eta}_i \wedge t }
h_{\tilde{\tau}_i}(u) \,du \biggr) \Biggr]
\\
&&\quad\qquad{}+\mathbb E \Biggl[ \sup_{0 \leq s < \infty}\bigl\llvert\varphi(s)
\bigr
\rrvert\sum_{i=1}^{E_t} \biggl(
\mathbf{1}_{ \{ \eta_i \leq( t -
\tau_i )^+ \} } + \int_0^{ \eta_i \wedge(t-\tau_i)^+ } h(u) \,du
\biggr) \Biggr]
\nonumber
\\
&&\qquad \leq \sup_{0 \leq s < \infty} \bigl\llvert\varphi(s) \bigr\rrvert
\bigl(
1 + t \llVert h\rrVert_{\infty} \bigr) \bigl(\mathbb E\bigl[
A_0(\infty)\bigr]+ \mathbb E[ E_t]\bigr)
\nonumber
\\
&&\qquad <\infty,
\nonumber
\end{eqnarray}
where the final inequality follows from the assumptions made on
$\mathbb E[ A_0(\infty)]$ and $\mathbb E[ E_t]$
in Section~\ref{SubsecAgeEquations}. Thus, $E[\llvert\langle(\mathcal
{D}^0 + \mathcal{D})_t, \varphi\rangle\rrvert] < \infty$ as desired.

Next, we show that $(\langle(\mathcal{D}^0 + \mathcal{D} )_{t},
\varphi\rangle)_{t \geq0}$ possesses the martingale property with
respect to the filtration $(\mathcal{F}_t^{\mathcal{A}})_{t \geq0}$.
That is, we show that for each $0 \leq s \leq t$,
%
%
\begin{eqnarray}
\mathbb E\bigl[\bigl\langle\bigl(\mathcal{D}^0 + \mathcal{D}
\bigr)_{t}, \varphi\bigr\rangle|\mathcal{F}_s^{\mathcal{A}}
\bigr]&=&\bigl\langle\bigl(\mathcal{D}^0 + \mathcal{D}
\bigr)_{s}, \varphi\bigr\rangle. \label{zeromart}
\end{eqnarray}
First note that by (\ref{D0def}) and (\ref{Ddef}), we may write
%
%
\begin{eqnarray}
\bigl\langle\bigl(\mathcal{D}^0 + \mathcal{D} \bigr)_{t},
\varphi\bigr\rangle&=& \sum_{i=1}^{\infty}
\mathbf{1}_{ \{ i \leq A_0(\infty) \} } \bigl\langle\mathcal{D}^{0,i}_t,
\varphi\bigr\rangle+ \sum_{i=1}^{\infty} \bigl
\langle\mathcal{D}^{i}_t, \varphi\bigr\rangle,
\label{decompose}
\end{eqnarray}
where, for each $i \geq1$, we set
\begin{eqnarray*}
&&\mathbf{1}_{ \{ i \leq A_0(\infty) \} }\bigl\langle\mathcal
{D}^{0,i}_t, \varphi\bigr\rangle
\\
&&\qquad =\mathbf{1}_{ \{ i \leq A_0(\infty) \} } \int_0^{t }
\varphi(x-\tilde{\tau}_i) \,d \biggl(\mathbf{1}_{ \{ \tilde{\eta}_i \leq
x \} }- \int
_0^{ \tilde{\eta}_i \wedge x } h_{\tilde{\tau}_i}(u) \,du \biggr)
\end{eqnarray*}
and
%
%
\begin{eqnarray}
\bigl\langle\mathcal{D}^{i}_t, \varphi\bigr\rangle&=&
\int_0^{(t - \tau
_i)^+} \varphi(x) \,d \biggl(
\mathbf{1}_{ \{ \eta_i \leq x \} } - \int_0^{\eta_i \wedge x} h(u) \,du
\biggr). \label{seecaldi}
\end{eqnarray}
We now show that for each $i \geq1$,
%
%
\begin{eqnarray}
\mathbb E\bigl[\bigl\langle\mathcal{D}^i_t, \varphi\bigr
\rangle|\mathcal{F}_s^{\mathcal{A}} \bigr]&=&\bigl\langle
\mathcal{D}^{i}_{s}, \varphi\bigr\rangle. \label{firstmart}
\end{eqnarray}
The proof that
%
%
\begin{eqnarray}
&&\mathbb E\bigl[\mathbf{1}_{ \{ i \leq A_0(\infty) \} } \bigl\langle
\mathcal
{D}^{0,i}_t, \varphi\bigr\rangle|\mathcal{F}_s^{\mathcal{A}}
\bigr]
\\
&&\qquad =\mathbf{1}_{ \{ i \leq A_0(\infty) \} }\mathbb E\bigl[ \bigl\langle
\mathcal
{D}^{0,i}_t, \varphi\bigr\rangle|\mathcal{F}_s^{\mathcal{A}}
\bigr]=\mathbf{1}_{ \{ i \leq A_0(\infty) \} }\bigl\langle\mathcal
{D}^{0,i}_{s},
\varphi\bigr\rangle, \label{secondmart}
\end{eqnarray}
for each $i \geq1$, is similar and will not be included. For each $i
\geq1$ and $y \geq0$, set
%
%
\begin{equation}
\label{Di} D^i_t(y) = \mathbf{1}_{ \{ \eta_i \leq( t - \tau_i )^+
\wedge y \}
} -
\int_0^{ \eta_i \wedge(t-\tau_i)^{+} \wedge y } h(u) \,du.
\end{equation}
We now claim that in order to show (\ref{firstmart}), it suffices to
show that $\mathbb E[D^i_t(y)|\break \mathcal{F}^{\mathcal{A}}_s ]=D^i_s(y)$
for each $y \geq0$. This is true since it will then follow that
\begin{eqnarray*}
\mathbb E \bigl[ \bigl\langle\mathcal{D}^i_t, \varphi
\bigr\rangle| \mathcal{F}^{\mathcal{A}}_s \bigr] &=& -\mathbb E
\biggl[\int_{\mathbb R_+} D^i_t(y)
\varphi'(y) \,dy \Big| \mathcal{F}^{\mathcal
{A}}_s \biggr]
\\
&=&- \int_{\mathbb R_+} \mathbb E \bigl[D^i_t(y)
| \mathcal{F}^{\mathcal{A}}_s \bigr] \varphi'(y) \,dy
\\
&=& -\int_{\mathbb R_+} D^i_s(y)
\varphi'(y) \,dy,
\\
&=& \bigl\langle\mathcal{D}^i_s, \varphi\bigr\rangle.
\end{eqnarray*}
First note that since $y \wedge(t - \tau_i
)^{+} = ( t \wedge(\tau_i + y ) - \tau_i )^{+}$, we may write
\begin{eqnarray*}
D_t^i(y) &=& D^i_{ t \wedge( \tau_i + y ) }(\infty).
\label{marteq}
\end{eqnarray*}
Next note that since by the assumptions in Section~\ref
{SubsecAgeEquations}, we have that $\mathbb{E}[\tau_i ] < \infty$,
it is straightforward to verify that $\tau_i + y$ is an $\mathcal
{F}^{\mathcal{A}}_t$-stopping
time for each $y \geq0$. Thus, by Problem 3.2.4 of \cite
{KaratzasShreve}, it suffices to
show that $D^i(\infty)=(D^i_t(\infty))_{t \geq0}$ is an $\mathbb
{R}$-valued $\mathcal{F}^{\mathcal{A}}_t$-martingale.
First note that since $\llVert h\rrVert _{\infty} < \infty$, it is
straightforward to show that $\mathbb E[\llvert D^i_t(\infty)\rrvert]
< \infty$ for each $t \geq0$. Next note that using the independence
of $\tau_i$ and $\eta_i$, one may also verify that
\begin{eqnarray*}
\mathbb E\bigl[D^i_t(\infty)| \mathbf{1}_{ \{ \tau_i \leq s \} },
\mathbf{1}_{ \{ \eta_i \leq s - \tau_i \} } \bigr] &=&D^i_s(\infty).
\end{eqnarray*}
Hence, since $\eta_i$ is assumed to be independent of $A_0, \{\tilde
{\eta}_k,k=1,\ldots,A_0(\infty)\},\break E=(E_t)_{t \geq0}$ and $\eta_k,k
\neq i$, we obtain
that
\begin{eqnarray*}
E\bigl[D^i_t(\infty) |\mathcal{F}^{\mathcal{A}}_s
\bigr] &=& E\bigl[D^i_t(\infty)| \mathbf{1}_{ \{ \tau_i \leq s \} },
\mathbf{1}_{ \{ \eta_i \leq s -
\tau_i \} } \bigr]
\\
&=&E\bigl[D^i_s(
\infty) |\mathcal{F}^{\mathcal{A}}_s \bigr]
\\
&=&D^i_s(\infty),
\end{eqnarray*}
and so $D^i(\infty) $ is an $\mathbb{R}$-valued $\mathcal
{F}^{\mathcal{A}}_t$-martingale as desired. This completes the proof
of (\ref{firstmart}). The proof of (\ref{secondmart}) is similar.

We now show that (\ref{firstmart}) and (\ref{secondmart}) imply
(\ref{zeromart}).
We claim that for each $k \geq1$, the sum
\[
\sum_{i=1}^{k} \mathbf{1}_{ \{ i \leq A_0(\infty) \} }
\bigl\langle\mathcal{D}^{0,i}_t, \varphi\bigr\rangle+\sum
_{i=1}^{k} \bigl\langle
\mathcal{D}^i_t, \varphi\bigr\rangle
\]
is dominated uniformly over $k \geq1$ by an integrable random
variable. By Lebesgue's dominated
convergence theorem for conditional expectations \cite
{chung2001course}, this then implies that
%
%
\begin{eqnarray} \label{interchangemart}
\mathbb E \bigl[ \langle\mathcal{D}_t, \varphi\rangle|
\mathcal{F}_s^{\mathcal{A}} \bigr] &=& \mathbb E \Biggl[ \sum
_{i=1}^{\infty} \bigl\langle\mathcal
{D}^i_t, \varphi\bigr\rangle\Big| \mathcal{F}_s^{\mathcal{A}}
\Biggr]\nonumber
\\
&=& \sum_{i=1}^{\infty}\mathbb E \bigl[
\bigl\langle\mathcal{D}^i_t, \varphi\bigr\rangle|
\mathcal{F}_s^{\mathcal{A}} \bigr]
\\
&=& \sum_{i=1}^{\infty} \bigl\langle
\mathcal{D}^i_s, \varphi\bigr\rangle= \langle
\mathcal{D}_s, \varphi\rangle,\nonumber
\end{eqnarray}
as desired. However, to obtain the bound is straightforward since
\begin{eqnarray*}
&&\Biggl\llvert\sum_{i=1}^{k}
\mathbf{1}_{ \{ i \leq A_0(\infty) \} } \bigl\langle\mathcal{D}^{0,i}_t,
\varphi\bigr\rangle+ \sum_{i=1}^{k} \bigl
\langle\mathcal{D}^i_t, \varphi\bigr\rangle\Biggr\rrvert
\\
&&\qquad \leq \Biggl\llvert\sum_{i=1}^{k}
\mathbf{1}_{ \{ i \leq A_0(\infty) \}
} \int_0^{t } \varphi(x-
\tilde{\tau}_i) \,d \biggl(\mathbf{1}_{ \{
\tilde{\eta}_i \leq x \} }- \int
_0^{ \tilde{\eta}_i \wedge x } h_{\tilde{\tau}_i}(u) \,du \biggr) \Biggr
\rrvert
\\
&&\quad\qquad{}+ \Biggl\llvert\sum_{i=1}^{k} \int
_0^{(t - \tau_i)^+} \varphi(x) \,d \biggl(\mathbf{1}_{ \{ \eta_i \leq x
\} }-
\int_0^{ \eta_i \wedge
x } h(u) \,du \biggr) \Biggr\rrvert
\\
&&\qquad \leq \sup_{0 \leq s < \infty}\bigl\llvert\varphi(s) \bigr\rrvert
\bigl(
1 + t \llVert h\rrVert_{\infty} \bigr) \bigl(A_0(\infty)
+E_t\bigr),
\end{eqnarray*}
and as in (\ref{lessinfty}) we have that
\begin{eqnarray}
\sup_{0 \leq s < \infty} \bigl\llvert\varphi(s) \bigr\rrvert\bigl( 1
+ t
\llVert h\rrVert_{\infty} \bigr) \bigl(\mathbb E\bigl[ A_0(
\infty)\bigr]+ \mathbb E[ E_t]\bigr)
\nonumber
&<&\infty.
\nonumber
\end{eqnarray}
We have therefore shown that $(\langle(\mathcal{D}^0 + \mathcal{D}
)_{t}, \varphi\rangle)_{t \geq0}$ is an $\mathbb{R}$-valued
$\mathcal{F}^{\mathcal{A}}_t$-mar\-tingale for each $\varphi\in
\mathcal{S}$,
which implies that $\mathcal{D}^0 + \mathcal{D}$ is an $\mathcal
{S}'$-valued $\mathcal{F}^{\mathcal{A}}_t$-martingale.

We next proceed to calculate the tensor quadratic variation of
$\mathcal{D}^0 + \mathcal{D}$.
Let $i \geq1$ and recall that by (\ref{seecaldi}) we have that for
each $\varphi\in
\mathcal{S}$ and $t \geq0$,
%
%
\begin{equation}
\bigl\langle\mathcal{D}^i_t, \varphi\bigr\rangle =
\int_0^{(t - \tau
_i)^+} \varphi(x) \,d \biggl(
\mathbf{1}_{ \{ \eta_i \leq x \} } - \int_0^{\eta_i \wedge x} h(u) \,du
\biggr).
\end{equation}
Therefore, as on page 259 of \cite{KrPu97}, it follows that
%
%
\begin{eqnarray} \label{DniQV}
\lll\,\bigl\langle\mathcal{D}^i, \varphi\bigr\rangle\,\rrr_t
&=& \int_0^{(t - \tau
_i)^+ } \varphi(x)^2 \,d \biggl
\langle\mathbf{1}_{ \{ \eta_i \leq x
\} } - \int_0^{\eta_i \wedge x}
h(u) \,du \biggr\rangle
\nonumber\\[-8pt]\\[-8pt]
& =& \int_0^{ \eta_i \wedge( t - \tau_i )^+ } \varphi(x)^2 h(x)
\,dx,\nonumber
\end{eqnarray}
which implies that
\begin{eqnarray*}
\lll\mathcal{D}^i \rrr_t( \varphi, \psi) &=& \bigl\langle \bigl
\langle\mathcal{D}^i, \varphi\bigr\rangle, \bigl\langle
\mathcal{D}^i, \psi\bigr\rangle\bigr\rangle_t
\\
&=& \frac{1}{4} \bigl(\lll\,\bigl\langle\mathcal{D}^i, \varphi+
\psi\bigr\rangle\,\rrr_t - \lll\,\bigl\langle\mathcal{D}^i,
\varphi-\psi\bigr\rangle\,\rrr_t \bigr)
\\
&=& \frac{1}{4} \biggl(\int_0^{ \eta_i \wedge( t - \tau_i )^+ } \bigl(
\varphi(x) + \psi(x) \bigr)^2 h(x) \,dx
\\
&&\hspace*{13pt}{} - \int_0^{ \eta_i \wedge( t - \tau
_i )^+ }
\bigl( \varphi(x) - \psi(x)\bigr)^2 h(x) \,dx \biggr)
\\
&=& \int_0^{ \eta_i \wedge( t - \tau_i )^+ } \varphi(x) \psi(x) h(x) \,dx,
\end{eqnarray*}
where the second equality in the above follows from the polarization
identity and the third
equality follows from (\ref{DniQV}). In a similar manner, one may also
show that for all $i \geq1$,
\[
\lll\mathbf{1}_{ \{ i \leq A_0(\infty) \} }\mathcal{D}^{0,i} \rrr_t(
\varphi,\psi) = \mathbf{1}_{ \{ i \leq A_0(\infty) \} }\int_0^{ \tilde
{\eta}_i \wedge t }
\varphi( x - \tilde{\tau}_i )\psi( x - \tilde{\tau}_i )
h_{\tilde{\tau}_i}(x) \,dx,
\]
for all $\varphi, \psi\in\mathcal{S}$.

We now claim that in order to show that the tensor quadratic variation of
$\mathcal{D}^0 + \mathcal{D}$ is given by (\ref{DD0QV}), it suffices
to show the following three facts:
\begin{longlist}[(2)]
\item[(1)] $\mathcal{D}^i$ is\label{ClaimMM1} orthogonal to $\mathcal{D}^j$ for $i \neq
j$,
\item[(2)] $\mathbf{1}_{ \{ i \leq A_0(\infty) \} }\mathcal{D}^{0,i}$ is\label{ClaimMM2}
orthogonal to $\mathbf{1}_{ \{ j \leq A_0(\infty) \} }\mathcal
{D}^{0,j}$ for $i \neq j$,
\item[(3)] $\mathbf{1}_{ \{ i \leq A_0(\infty) \} }\mathcal{D}^{0,i}$ is\label{ClaimMM3}
orthogonal to $\mathcal{D}^j$ for all $i,j \geq1$.
\end{longlist}
The fact that the tensor quadratic variation of $\mathcal{D}^0 +
\mathcal{D}$ is given by (\ref{DD0QV})
can then be shown in the following manner. For each $k \geq1,\varphi
\in\mathcal{S}$ and $t \geq0$, let
\begin{eqnarray*}
\bigl\langle\bigl(\mathcal{D}^0 + \mathcal{D} \bigr)_{t}^k,
\varphi\bigr\rangle&=& \sum_{i=1}^{k}
\mathbf{1}_{ \{ i \leq A_0(\infty) \} } \bigl\langle\mathcal{D}^{0,i}_t,
\varphi\bigr\rangle+ \sum_{i=1}^{k} \bigl
\langle\mathcal{D}^{i}_t, \varphi\bigr\rangle,
\end{eqnarray*}
and set $(\mathcal{D}^0 + \mathcal{D})^{k}=((\mathcal{D}^0 +
\mathcal{D})^{k}_t, t \geq0) $. It is then clear that claims~\hyperref[ClaimMM1]{(1)} through~\hyperref[ClaimMM3]{(3)} above imply that
\begin{eqnarray*}
&&\lll\bigl(\mathcal{D}^0 + \mathcal{D}\bigr)^{k}
\rrr_t(\varphi, \psi)
\\
&&\qquad =\sum_{i=1}^{k} \mathbf{1}_{ \{ i \leq A_0(\infty) \} }
\int_0^{
\tilde{\eta}_i \wedge t } \varphi( x - \tilde{
\tau}_i ) \psi( x - \tilde{\tau}_i ) h_{\tilde{\tau}_i}(x)
\,dx
\nonumber
\\
&&\quad\qquad{} + \sum_{i=1}^{k} \int
_0^{ \eta_i \wedge( t - \tau_i )^+ } \varphi(x) \psi(x) h(x) \,dx.
\end{eqnarray*}
Moreover, using similar arguments as above and the simple inequality
$(x_1+x_2)^2 \leq4(x_1^2+x_2^2)$, it is straightforward to show that
for each $k \geq1$,
one has $\mathbb{P}$-a.s. the bound
\begin{eqnarray*}
&& \Biggl\llvert\bigl\langle\bigl(\mathcal{D}^0 + \mathcal{D}
\bigr)^{k}_t,\varphi\bigr\rangle\bigl\langle\bigl(
\mathcal{D}^0 + \mathcal{D}\bigr)^{k}_t,\psi
\bigr\rangle
\\
&&\quad{}- \Biggl( \sum_{i=1}^{k}
\mathbf{1}_{ \{ i \leq A_0(\infty)
\} } \int_0^{ \tilde{\eta}_i \wedge t } \varphi( x
- \tilde{\tau}_i ) \psi( x - \tilde{\tau}_i )
h_{\tilde{\tau}_i}(x) \,dx
\nonumber
\\
&&\hspace*{90pt}{}+ \sum_{i=1}^{k}
\int_0^{ \eta_i \wedge( t -
\tau_i )^+ } \varphi(x) \psi(x) h(x) \,dx \Biggr)
\Biggr\rrvert
\\
&&\qquad \leq 4 \Bigl(\sup_{0 \leq s < \infty}\bigl(\bigl\llvert\varphi
(s)\bigr
\rrvert+\bigl\llvert\psi(s)\bigr\rrvert\bigr) \Bigr)^2 \bigl(1+t
\llVert h\rrVert_{\infty
}\bigr)^2 \bigl(E_t^{2}+A_0^2(
\infty)\bigr)
\\
&&\quad\qquad{}+ \sup_{0 \leq s < \infty}\bigl(\bigl\llvert\varphi(s)\bigr\rrvert
\bigl
\llvert\psi(s)\bigr\rrvert\bigr) \bigl(1+t \llVert h\rrVert_{\infty}
\bigr) \bigl(E_t+A_0(\infty)\bigr).
\end{eqnarray*}
However, by the assumptions in Section~\ref{SubsecAgeEquations}, one has
that $\mathbb E[E_t^{2}+A_0^2(\infty)] < \infty$ and $\mathbb
E[E_t+A_0(\infty)] < \infty$. Hence, using the dominated convergence
theorem for conditional expectations \cite{chung2001course}, it
follows that for $0 \leq s \leq t$,
\begin{eqnarray*}
&& \mathbb E\bigl[\bigl\langle\bigl(\mathcal{D}^0 + \mathcal{D}
\bigr)_t, \varphi\bigr\rangle\bigl\langle\bigl(
\mathcal{D}^0 + \mathcal{D}\bigr)_t, \varphi\bigr\rangle-
\lll\bigl(\mathcal{D}^0 + \mathcal{D}\bigr)\rrr_t(
\varphi, \psi) | \mathcal{F}^{A}_s \bigr]
\\
&&\qquad = \mathbb E\Bigl[\lim_{k \rightarrow\infty}\bigl( \bigl\langle\bigl(
\mathcal{D}^0 + \mathcal{D}\bigr)_t^k,
\varphi\bigr\rangle\bigl\langle\bigl(\mathcal{D}^0 + \mathcal{D}
\bigr)_t^k, \psi\bigr\rangle- \lll\bigl(
\mathcal{D}^0 + \mathcal{D}\bigr)^k \rrr_t(
\varphi, \psi)\bigr) | \mathcal{F}^{A}_s \Bigr]
\\
&&\qquad =\lim_{k \rightarrow\infty} \mathbb E\bigl[ \bigl\langle\bigl(
\mathcal{D}^0 + \mathcal{D}\bigr)_t^k,
\varphi\bigr\rangle\bigl\langle\bigl(\mathcal{D}^0 + \mathcal{D}
\bigr)_t^k, \psi\bigr\rangle- \lll\bigl(
\mathcal{D}^0 + \mathcal{D}\bigr)^k \rrr_t(
\varphi, \psi) | \mathcal{F}^{A}_s \bigr]
\\
&&\qquad =\lim_{k \rightarrow\infty} \bigl( \bigl\langle\bigl(\mathcal{D}^0
+ \mathcal{D}\bigr)_s^k, \varphi\bigr\rangle\bigl
\langle\bigl(\mathcal{D}^0 + \mathcal{D}\bigr)_s^k, \psi\bigr\rangle-
\lll\bigl(\mathcal{D}^0 + \mathcal{D}
\bigr)^k \rrr_s(\varphi, \psi)\bigr)
\\
&&\qquad = \bigl\langle\bigl(\mathcal{D}^0 + \mathcal{D}
\bigr)_s, \varphi\bigr\rangle\bigl\langle\bigl(
\mathcal{D}^0 + \mathcal{D}\bigr)_s, \psi\bigr\rangle-
\lll\bigl(\mathcal{D}^0 + \mathcal{D}\bigr)\rrr_s(
\varphi, \psi).
\end{eqnarray*}
This then implies that the tensor quadratic variation of $\mathcal
{D}^0 + \mathcal{D}$ is given by (\ref{DD0QV}). We now proceed to
prove claims~\hyperref[ClaimMM1]{(1)} through~\hyperref[ClaimMM3]{(3)}, which is
sufficient to complete the proof.

We begin with claim~\hyperref[ClaimMM1]{(1)}. Let $\varphi, \psi\in\mathcal
{S}$ and $i \neq j$. We show that $(\langle\mathcal{D}^i_t,\break \varphi
\rangle\langle\mathcal{D}^j_t,\psi\rangle)_{t \geq0}$
is an $\mathbb{R}$-valued $\mathcal{F}^{\mathcal{A}}_t$-martingale,
which is sufficient
to show that $\mathcal{D}^i$ is orthogonal to $\mathcal{D}^j$. First
note that it is clear as in (\ref{lessinfty}) that for each $t \geq
0$, we have that $\mathbb E[\llvert\langle\mathcal{D}^i_t,\varphi
\rangle\langle\mathcal{D}^j_t,\psi\rangle\rrvert] < \infty$.
Next, let $0 \leq s \leq t$. By the independence of $\eta_i$
from $A_0, \{\tilde{\eta}_k,k=1,\ldots,A_0(\infty)\},E=(E_t)_{t
\geq0}$ and $\eta_k,k \neq i$,
and, similarly, the independence of $\eta_j$
from $A_0, \{\tilde{\eta}_k,k=1,\ldots,A_0(\infty)\},E=(E_t)_{t
\geq0}$ and $\eta_k,k \neq j$,
it follows that
\begin{eqnarray*}
&&\mathbb E\bigl[\bigl\langle\mathcal{D}^i_t,\varphi
\bigr\rangle\bigl\langle\mathcal{D}^j_t,\psi\bigr
\rangle|\mathcal{F}^{\mathcal{A}}_s \bigr]
\\
&&\qquad =\mathbb E\bigl[\bigl\langle\mathcal{D}^i_t,\varphi
\bigr\rangle\bigl\langle\mathcal{D}^j_t,\psi\bigr
\rangle| \mathbf{1}_{ \{ \tau_i \leq s \} }, \mathbf{1}_{ \{ \tau_j
\leq s \} },
\mathbf{1}_{ \{ \eta_i \leq s - \tau_i
\} },\mathbf{1}_{ \{ \eta_j \leq s - \tau_j \} } \bigr].
\end{eqnarray*}
However, by the independence of $\eta_i$ from $\eta_j$ and $\tau_j$,
and, similarly, the independence of $\eta_j$ from $\eta_i$ and $\tau
_i$, we have that
\begin{eqnarray*}
&&\mathbb E\bigl[\bigl\langle\mathcal{D}^i_t,\varphi
\bigr\rangle\bigl\langle\mathcal{D}^j_t,\psi\bigr
\rangle| \mathbf{1}_{ \{ \tau_i \leq s \} }, \mathbf{1}_{ \{ \tau_j
\leq s \} },
\mathbf{1}_{ \{ \eta_i \leq s - \tau_i
\} },\mathbf{1}_{ \{ \eta_j \leq s - \tau_j \} } \bigr]
\\
&&\qquad =\mathbb E\bigl[\bigl\langle\mathcal{D}^i_t,\varphi
\bigr\rangle| \mathbf{1}_{
\{ \tau_i \leq s \} }, \mathbf{1}_{ \{ \eta_i \leq s - \tau_i \} } \bigr]
\mathbb E\bigl[ \bigl\langle\mathcal{D}^j_t,\psi\bigr
\rangle| \mathbf{1}_{ \{
\tau_j \leq s \} }, \mathbf{1}_{ \{ \eta_j \leq s - \tau_j \} } \bigr]
\\
&&\qquad =\bigl\langle\mathcal{D}^i_s,\varphi\bigr\rangle
\bigl\langle\mathcal{D}^j_s,\psi\bigr\rangle,
\end{eqnarray*}
where the final equality follows from (\ref{firstmart}) and the fact
that for $k \geq1$,
%
%
\begin{eqnarray}
\mathbb E\bigl[\bigl\langle\mathcal{D}^k_t,\varphi\bigr
\rangle| \mathbf{1}_{ \{
\tau_k \leq s \} }, \mathbf{1}_{ \{ \eta_k \leq s - \tau_k \} } \bigr
]&=&\mathbb
E\bigl[\bigl\langle\mathcal{D}^k_t,\varphi\bigr\rangle|
\mathcal{F}^{\mathcal{A}}_s \bigr]. \label{afterfact}
\end{eqnarray}
Thus, it is clear that $(\langle\mathcal{D}^i_t,\varphi\rangle
\langle\mathcal{D}^j_t,\psi\rangle)_{t \geq0}$ possesses the
martingale property and so $(\langle\mathcal{D}^i_t,\varphi\rangle
\langle\mathcal{D}^j_t,\psi\rangle)_{t \geq0}$ is an $\mathbb
{R}$-valued $\mathcal{F}^{\mathcal{A}}_t$-martingale,\vspace*{1pt} and hence
$\mathcal{D}^i$ is orthogonal to $\mathcal{D}^j$. The proof of claim~\hyperref[ClaimMM2]{(2)} above follows similarly. The proof of
claim~\hyperref[ClaimMM3]{(3)} above follows in a similar manner as well. In particular,
let $i,j \geq1$. We show that $(\mathbf{1}_{ \{ i \leq A_0(\infty)
\} }\langle\mathcal{D}^{0,i}_t,\varphi\rangle\langle\mathcal
{D}^j_t,\psi\rangle)_{t \geq0}$
is an $\mathbb{R}$-valued $\mathcal{F}^{\mathcal{A}}_t$-martingale
for each $\varphi, \psi\in\mathcal{S}$, which is sufficient
to show that $\mathbf{1}_{ \{ i \leq A_0(\infty) \} }\mathcal
{D}^{0,i}$ is orthogonal to $\mathcal{D}^j$. For each $t \geq0$, it
is clear as in (\ref{lessinfty}) that $\mathbb E[\llvert\mathbf{1}_{
\{ i \leq A_0(\infty) \} }\langle\mathcal{D}^{0,i}_t,\varphi\rangle
\langle\mathcal{D}^j_t,\psi\rangle\rrvert ] < \infty$. Next, note
that since $\tilde{\eta}_i$ is independent of
$A_0, \{\tilde{\eta}_k,k=1,\ldots,A_0(\infty); k \neq i\},E=(E_t)_{t
\geq0}$ and $\eta_k,k \geq1$,
and, similarly, $\eta_j$ is independent of
$A_0, \{\tilde{\eta}_k,k=1,\ldots,A_0(\infty)\},E=(E_t)_{t \geq0}$
and $\eta_k,k \neq j$, we have
that
\begin{eqnarray*}
&&\mathbb E\bigl[\mathbf{1}_{ \{ i \leq A_0(\infty) \} }\bigl\langle
\mathcal
{D}^{0,i}_t,\varphi\bigr\rangle\bigl\langle
\mathcal{D}^j_t,\psi\bigr\rangle|\mathcal{F}^{\mathcal{A}}_s
\bigr]
\\
&&\qquad =\mathbb E\bigl[\mathbf{1}_{ \{ i \leq A_0(\infty) \} }\bigl\langle
\mathcal
{D}^{0,i}_t,\varphi\bigr\rangle\bigl\langle
\mathcal{D}^j_t,\psi\bigr\rangle| \tilde{
\tau}_i, \mathbf{1}_{ \{ \tau_j \leq s \} }, \mathbf{1}_{
\{ \tilde{\eta}_i \leq s - \tilde{\tau}_i \} },
\mathbf{1}_{ \{
\eta_j \leq s - \tau_j \} } \bigr].
\end{eqnarray*}
However, by the independence of $\tilde{\eta}_i$ from $\eta_j$ and
$\tau_j$ and, similarly, the independence of $\eta_j$ from $\tilde
{\eta}_i$ and $\tilde{\tau}_i$, we have that
\begin{eqnarray*}
&&\mathbb E\bigl[\mathbf{1}_{ \{ i \leq A_0(\infty) \} }\bigl\langle
\mathcal
{D}^{0,i}_t,\varphi\bigr\rangle\bigl\langle
\mathcal{D}^j_t,\psi\bigr\rangle| \tilde{
\tau}_i, \mathbf{1}_{ \{ \tau_j \leq s \} }, \mathbf{1}_{
\{ \tilde{\eta}_i \leq s - \tilde{\tau}_i \} },
\mathbf{1}_{ \{
\eta_j \leq s - \tau_j \} } \bigr]
\\
&&\qquad =\mathbb E\bigl[\mathbf{1}_{ \{ i \leq A_0(\infty) \} }\bigl\langle
\mathcal
{D}^{0,i}_t,\varphi\bigr\rangle| \tilde{
\tau}_i, \mathbf{1}_{ \{
\tilde{\eta}_i \leq s - \tilde{\tau}_i \} } \bigr]\mathbb E\bigl[ \bigl
\langle
\mathcal{D}^j_t,\psi\bigr\rangle| \mathbf{1}_{ \{ \tau_j \leq s \} },
\mathbf{1}_{ \{ \eta_j \leq s - \tau_j \} } \bigr]
\\
&&\qquad =\mathbf{1}_{ \{ i \leq A_0(\infty) \} }\bigl\langle\mathcal{D}^{0,i}_s,
\varphi\bigr\rangle\bigl\langle\mathcal{D}^j_s,\psi
\bigr\rangle,
\end{eqnarray*}
where the final equality follows by (\ref{firstmart}), (\ref
{secondmart}), (\ref{afterfact}) and the fact that for $k \geq1$,
\begin{eqnarray*}
\mathbb E\bigl[\mathbf{1}_{ \{ k \leq A_0(\infty) \} }\bigl\langle
\mathcal
{D}^{0,k}_t,\varphi\bigr\rangle| \tilde{
\tau_k}, \mathbf{1}_{ \{
\tilde{\eta}_k \leq s - \tilde{\tau}_k \} } \bigr] &=&\mathbb E\bigl[
\mathbf{1}_{ \{ k \leq A_0(\infty) \} }\bigl\langle\mathcal{D}^{0,k}_t,
\varphi\bigr\rangle| \mathcal{F}^{\mathcal{A}}_s \bigr].
\end{eqnarray*}
Thus, it is clear that $(\mathbf{1}_{ \{ i \leq A_0(\infty) \}
}\langle\mathcal{D}^{0,i}_t,\varphi\rangle\langle\mathcal
{D}^j_t,\psi\rangle)_{t \geq0}$ possesses the martingale property
and so $(\mathbf{1}_{ \{ i \leq A_0(\infty) \} }\langle\mathcal
{D}^{0,i}_t,\varphi\rangle\langle\mathcal{D}^j_t,\psi\rangle)_{t
\geq0}$ is an $\mathbb{R}$-valued $\mathcal{F}^{\mathcal
{A}}_t$-martingale, and hence $\mathcal{D}^{0,i}$ is orthogonal to
$\mathcal{D}^j$. This proves claim~\hyperref[ClaimMM3]{(3)}, which completes
the proof.
\end{pf}

\subsection{Residuals}

In this subsection, we show that the process $\mathcal{G}$ defined in
Section~\ref{SubsecResidualEquations} is a martingale. This fact may be useful
in future work where one wishes to show that the residual service time
process is a Markov process. Let $(\mathcal{F}^{\mathcal{G}}_t)_{t
\geq0}$ be the
natural filtration generated by $\mathcal{G}$. We then have the
following result.

\begin{prop}\label{PropMartingalesTwo}
The process $\mathcal{G}$ is an $\mathcal{S}'$-valued
$\mathcal{F}^{\mathcal{G}}_t$-martingale with tensor optional
quadratic variation process given for all $\varphi, \psi\in
\mathcal{S}$ by
%
%
\begin{equation}
\label{GOQV} [ \mathcal{G} ]_t(\varphi, \psi) = \sum
_{i=1}^{E_t} \bigl(\varphi(\eta_i)-
\langle\mathcal{F},\varphi\rangle\bigr) \bigl(\psi(\eta_i)-\langle
\mathcal{F},\psi\rangle\bigr).
\end{equation}
\end{prop}

\begin{pf}
Let $\varphi\in\mathcal{S}$. We first show that $(\langle\mathcal
{G}_t,\varphi\rangle)_{t \geq0} $ is an $\mathbb{R}$-valued
$\mathcal{F}^{\mathcal{G}}_t$-martin\-gale. Define the filtration
$(\mathcal{H}_k)_{k \geq1} $ by setting $\mathcal{H}_k = \sigma\{E_t,
t \geq0\} \vee\sigma\{ \eta_1, \eta_2,\ldots,\eta_k \} \vee
\mathcal{N}$ for each $k \geq1$. Next, define the discrete-time
$\mathbb{D}$-valued process $(G^k)_{k
\geq1}$ by
%
%
\begin{equation}
G^k(y) = \sum_{i=1}^{k}
\bigl(\mathbf{1}_{ \{ \eta_i \leq y \}
}-F(y) \bigr),\qquad y \geq0,
\end{equation}
and, for convenience, let $G^k(y)=0$ for $y < 0$. Then let $(\mathcal
{G}^k)_{k \geq1}$ be the
$\mathcal{S}'$-valued process associated with $G^k$. Since $\varphi
\in\mathcal{S}$ is bounded, it is clear that $E[\llvert\langle
\mathcal{G}^k, \varphi\rangle\rrvert ] < \infty$ for each $k \geq
1$. Moreover, by the independence of the service times from the arrival
process, one has that $(\langle
\mathcal{G}^k, \varphi\rangle)_{k \geq1}$ possesses the martingale
property with respect to
$(\mathcal{H}_k)_{k \geq1} $. Hence, $(\langle
\mathcal{G}^k, \varphi\rangle)_{k \geq1}$
is an $\mathbb{R}$-valued
$\mathcal{H}_k$-martingale. However, since for each $t \geq0$ we
have by the assumptions in Section~\ref{SubsecAgeEquations} that
$\mathbb E[E_t] < \infty$, it is straightforward to see that $E_t$ is
a stopping time with respect to the filtration
$(\mathcal{H}_k)_{k \geq1}$. Thus, the filtration
$(\mathcal{H}_{E_t})_{t \geq0}$ is well defined and, furthermore,
it follows by the optional sampling theorem \cite{KaratzasShreve} that
$(\langle
\mathcal{G}_t, \varphi\rangle)_{t \geq0}=(\langle
\mathcal{G}^{E_t}, \varphi\rangle)_{t \geq0}$ is an
\mbox{$\mathcal{H}_{E_t}$-}martingale. The result now follows since any
martingale is a martingale relative to its natural filtration.

The form of the tensor optional quadratic variation (\ref{GOQV}) is
immediate by Theorem 3.3 of \cite{PTW}.
\end{pf}

\section{Fluid limits}\label{SecFluidLimits}

In this section, we provide our main fluid limit results. We begin in
Section~\ref{SubsecAgesFluid} by studying the age process and in
Section~\ref{SubsecResidualsFluid} we study the residual service time process.
Our setup in both subsections is the same. In particular, we consider a
sequence of $G/\mathit{GI}/\infty$ queues indexed by $n \geq1$, where the
arrival rate to the system grows large with $n$ while the service time
distribution does not change with $n$.

\subsection{Ages}\label{SubsecAgesFluid}

We begin by studying the age process $\mathcal{A}$ defined in
Section~\ref{SubsecAgeEquations}. For each $n \geq1$, define the
fluid scaled quantities
%
%
\begin{eqnarray}\label{AgeFluidDef}
\bar{\mathcal{A}}^n_0 &\equiv&
\frac{\mathcal{A}^n_0}{n},\qquad\bar{E}^n \equiv\frac{E^n}{n},\qquad
\bar{\mathcal{D}}^{0,n} \equiv\frac{\mathcal{D}^{0,n}}{n},
\nonumber\\[-8pt]\\[-8pt]
\bar{\mathcal{D}}^n &\equiv&\frac
{\mathcal{D}^n}{n},
\qquad
\bar{\mathcal{A}}^n \equiv\frac{\mathcal{A}^n}{n},\nonumber
\end{eqnarray}
and set $\bar{\mathcal{E}}^n \equiv\bar{E}^n \delta_0$. Using
(\ref{AgesSystemEquation}), Theorem~\ref{ThmRegulatorMap} and
Proposition~\ref{PropAgeC01}, it is straightforward to show that one
may write
%
%
\begin{eqnarray}
\bar{\mathcal{A}}^n &=& \Psi_{B^{\mathcal{A}}}\bigl( \bar{\mathcal
{A}}_0^n+ \bar{\mathcal{E}}^n-\bigl(\bar{
\mathcal{D}}^{0,n}+\bar{\mathcal{D}}^n\bigr)\bigr),
\label{inreg}
\end{eqnarray}
where\vspace*{1pt} the map $\Psi_{B^{\mathcal{A}}}\dvtx  \mathbb{D}([0,T],\mathcal
{S}') \mapsto\mathbb{D}([0,T],\mathcal{S}')$ is continuous. We now
prove that if $( \bar{\mathcal{A}}^n_0+ \bar{\mathcal{E}}^n
)_{n \geq1}$ weakly converges, then so too does $( \bar{\mathcal
{A}}^n_0+ \bar{\mathcal{E}}^n
-(\bar{\mathcal{D}}^{0,n} +
\bar{\mathcal{D}}^n))_{n \geq1}$.

\begin{prop}\label{Dfluidlimit}
If $\bar{\mathcal{A}}^n_0 + \bar{\mathcal{E}}^n \Rightarrow
\bar{\mathcal{A}}_0 +\bar{\mathcal{E}} $ in $ \mathbb
{D}([0,T],\mathcal{S}')$ as $n
\rightarrow\infty$, then
\begin{eqnarray*}
\label{agefluidlimitprop} \bar{\mathcal{A}}^n_0 + \bar{
\mathcal{E}}^n -\bigl( \bar{\mathcal{D}}^{0,n} + \bar{
\mathcal{D}}^n\bigr) &\Rightarrow& \bar{\mathcal{A}}_0 +
\bar{\mathcal{E}} \qquad\mbox{in } \mathbb{D}\bigl([0,T],\mathcal{S}'
\bigr) \mbox{ as }n \rightarrow\infty.
\end{eqnarray*}
\end{prop}

\begin{pf}
We first note that by Theorem~\ref{Mitoma}, it is sufficient to show
that if $\bar{\mathcal{A}}^n_0 + \bar{\mathcal{E}}^n \Rightarrow
\bar{\mathcal{A}}_0 +\bar{\mathcal{E}} $ as $n \rightarrow\infty
$, then
%
%
\begin{equation}
\label{D0Dfluidlimit} \bar{\mathcal{D}}^{0,n} + \bar{\mathcal{D}}^n
\Rightarrow0 \qquad\mbox{in } \mathbb{D}\bigl([0,T],\mathcal{S}'
\bigr) \mbox{ as }n \rightarrow\infty.
\end{equation}
Let $T > 0$ and $0 \leq t \leq T$. Then, for each $\varphi\in\mathcal
{S}$, we have by Proposition~\ref{PropMartingales} that
%
%
\begin{eqnarray}
&&\bigl\llvert\lll\bar{\mathcal{D}}^{0,n} + \bar{\mathcal{D}}^n
\rrr_t (\varphi,\varphi)\bigr\rrvert
\nonumber
\\
&&\qquad =\Biggl\llvert\frac{1}{n^2} \Biggl( \sum_{i=1}^{A^n_0(\infty)}
\int_0^{ \tilde{\eta}_i \wedge t } \varphi^2\bigl( x -
\tilde{\tau}^n_i \bigr) h_{\tilde{\tau}^n_i}(x) \,dx
\nonumber\\[-8pt]\\[-8pt]\nonumber
&&\hspace*{60pt}{}+ \sum_{i=1}^{E^n_t}
\int_0^{ \eta_i
\wedge(t-\tau^n_i)^+ } \varphi^2( x )h(x) \,dx
\Biggr) \Biggr\rrvert
\nonumber
\\
&&\qquad \leq \frac{\llVert h \rrVert _{\infty} }{n^2} \sum
_{i=1}^{A^n_0(\infty)}
\int_0^{ t } \varphi^2\bigl( x -
\tilde{\tau}^n_i \bigr) \,dx +\frac{\llVert \varphi^2 h \rrVert
_{\infty}}{n}
\bar{E}^n_T. \label{D0DQVfluidlimit}
\end{eqnarray}
Thus, from (\ref{D0DQVfluidlimit}) we obtain that for each $0 \leq t
\leq T$,
%
%
\begin{eqnarray}\label{D0DQVfluidlimit23}
&& \bigl\llvert\lll\bar{\mathcal{D}}^{0,n} + \bar{\mathcal{D}}^n
\rrr_t (\varphi,\varphi)\bigr\rrvert\nonumber
\\
&&\qquad \leq\frac{\llVert h \rrVert _{\infty} }{n^2} \sum_{i=1}^{A^n_0(\infty)}
\int_0^{ t } \varphi^2\bigl( x -
\tilde{\tau}^n _i \bigr) \,dx +\frac{\llVert \varphi^2 h \rrVert
_{\infty}}{n} \bar
{E}^n_T\nonumber
\\
&&\qquad = \frac{\llVert h \rrVert _{\infty} }{n^2} \int_0^{ t } \sum
_{i=1}^{A^n_0(\infty)} \varphi^2\bigl( x - \tilde{
\tau}^n_i \bigr) \,dx +\frac{\llVert \varphi^2 h \rrVert _{\infty}}{n}
\bar{E}^n_T
\\
&&\qquad = \frac{\llVert h \rrVert _{\infty} }{n} \int_0^{ t } \bigl\langle
\bar{\mathcal{A}}^n_0, \tau_{-x}
\varphi^2 \bigr\rangle \,dx +\frac
{\llVert \varphi^2 h \rrVert _{\infty}}{n} \bar{E}^n_T
\nonumber
\\
&&\qquad \leq\frac{\llVert h \rrVert _{\infty} t}{n} q_{K}\bigl( \bar
{\mathcal{A}}^n_0
\bigr) +\frac{\llVert \varphi^2 h \rrVert _{\infty
}}{n} \bar{E}^n_T,
\nonumber
\end{eqnarray}
where the set $K$ is given by $K=\{\tau_{-x}\varphi^2, 0 \leq x \leq
t\}$. By Lemma~\ref{shiftboundlemma}, the set $K$ is bounded in
$\mathcal{S,}$ and hence $q_{K}$ is a seminorm on $\mathcal{S}'$
by Definition~\ref{defstrongdual}. This then implies that $q_K$ is a
continuous function on $\mathcal{S}'$. Hence, since by assumption
\mbox{$\bar{\mathcal{A}}_0^n \Rightarrow\bar{\mathcal{A}}_0$} and $\bar
{E}^n_T \Rightarrow\bar{E}_T$, it follows by Slutsky's theorem that
\begin{eqnarray*}
\frac{\llVert h \rrVert _{\infty} t}{n} q_{K}\bigl( \bar{\mathcal{A}}^n_0
\bigr) +\frac{\llVert \varphi^2 h \rrVert _{\infty}}{n} \bar{E}^n_T
&\Rightarrow&0
\qquad\mbox{as }n \rightarrow\infty.
\end{eqnarray*}
By (\ref{D0DQVfluidlimit23}), this then implies that
%
%
\begin{eqnarray}
\lll\bar{\mathcal{D}}^{0,n} + \bar{\mathcal{D}}^n \rrr(
\varphi,\varphi) &\Rightarrow&0\qquad\mbox{in }\mathbb{D}\bigl([0,T],
\mathbb R
\bigr)\mbox{ as }n \rightarrow\infty. \label{convquad}
\end{eqnarray}

We now verify that parts (1) and (2) of Theorem~\ref{Mitoma} are satisfied
for the sequence $(\bar{\mathcal{D}}^{0,n} + \bar{\mathcal
{D}}^n)_{n \geq1}$, with the limit point being the function which is
identically $0$. We begin with part~(1). Using\vspace*{1pt} the fact that
the maximum jump of both $\langle\bar{\mathcal{D}}^{0,n} + \bar{\mathcal
{D}}^n, \varphi\rangle$ and $\ll \bar{\mathcal{D}}^{0,n} +
\bar{\mathcal{D}}^n \gg(\varphi,\varphi)$ is bounded over the
interval $[0,T]$ uniformly in $n$, we
obtain by (\ref{convquad}) and the martingale FCLT (see Theorem 7.1.4
of \cite{EK86} or \cite{WhittMGFCLT}) that
%
%
\begin{equation}
\bigl\langle\bar{\mathcal{D}}^{0,n} + \bar{\mathcal{D}}^n,
\varphi\bigr\rangle\Rightarrow0 \qquad\mbox{in }\mathbb{D}\bigl
([0,T],\mathbb R
\bigr) \mbox{ as }n \rightarrow\infty. \label{cond1convergence}
\end{equation}
Thus, part~(1) of Theorem~\ref{Mitoma} holds. We next check that
condition~(2) holds. Let $m \geq1$
and let $t_1,\ldots,t_m \in[0,T]$ and $\varphi_1, \ldots, \varphi
_m \in\mathcal{S}$. By (\ref{cond1convergence}), we have that for
each $1 \leq i \leq m$,
\[
\bigl\langle\bigl(\bar{\mathcal{D}}^{0,n} + \bar{\mathcal{D}}^n
\bigr)_{t_i}, \varphi_i \bigr\rangle\Rightarrow0 \qquad
\mbox{in }\mathbb{R} \mbox{ as }n \rightarrow\infty.
\]
However, by Theorem 3.9 of \cite{Billingsley99} this now implies that
\[
\bigl(\bigl\langle\bigl(\bar{\mathcal{D}}^{0,n} + \bar{\mathcal
{D}}^n\bigr)_{t_1}, \varphi_1 \bigr\rangle,\ldots,\bigl\langle\bigl(\bar{\mathcal{D}}^{0,n} + \bar{
\mathcal{D}}^n\bigr)_{t_m}, \varphi_m \bigr
\rangle\bigr) \Rightarrow(0,\ldots,0) \qquad\mbox{in }\mathbb{R}^m,
\]
as $ n \rightarrow\infty$. Thus, we have shown that part~(2) of Theorem~\ref{Mitoma} holds and so (\ref{D0Dfluidlimit})
is proven. This completes the proof.
\end{pf}

We are now in a position to prove the main result of this subsection.
We have the following.

\begin{teo}\label{AgeFluidLimitTheorem}
If $\bar{\mathcal{A}}^n_0 + \bar{\mathcal{E}}^n \Rightarrow
\bar{\mathcal{A}}_0 +\bar{\mathcal{E}} $ in $ \mathbb
{D}([0,T],\mathcal{S}')$ as $n
\rightarrow\infty$, then
\[
\bar{\mathcal{A}}^n \Rightarrow\bar{\mathcal{A}} \qquad\mbox{in }
\mathbb{D}\bigl([0,T],\mathcal{S}'\bigr) \mbox{ as }n
\rightarrow\infty,
\]
where $\bar{\mathcal{A}}$ is the unique solution to the integral
equation
%
%
\begin{eqnarray}\label{AgeFluidLimit}
\langle\bar{ \mathcal{A} }_t, \varphi\rangle=
\langle\bar{ \mathcal{A} }_0, \varphi\rangle+ \langle\bar{
\mathcal{E}}_t,\varphi\rangle- \int_0^t
\langle\bar{ \mathcal{A} }_s, h \varphi\rangle \,ds + \int
_0^t \bigl\langle\bar{ \mathcal{A}
}_s, \varphi' \bigr\rangle \,ds,
\nonumber\\[-8pt]\\[-12pt]
\eqntext{t \in[0,T],}
\end{eqnarray}
for all $\varphi\in\mathcal{S}$.
\end{teo}

\begin{pf}
By the assumption that $\bar{\mathcal{A}}^n_0 + \bar{\mathcal{E}}^n
\Rightarrow
\bar{\mathcal{A}}_0 +\bar{\mathcal{E}} $, it follows immediately by
Proposition~\ref{Dfluidlimit} that
%
%
\begin{equation}\label{usemapshow}
\qquad\bar{\mathcal{A}}_0^n + \bar{\mathcal{E}}^n
-\bigl( \bar{\mathcal{D}}^{0,n}+\bar{\mathcal{D}}^n\bigr)
\Rightarrow\bar{\mathcal{A}}_0 + \bar{\mathcal{E}} \qquad\mbox{in }
\mathbb{D}\bigl([0,T],\mathcal{S}'\bigr) \mbox{ as }n
\rightarrow\infty.
\end{equation}
Next, recall that by (\ref{inreg}) we have that
$\bar{\mathcal{A}}^n=\Psi_{B^{\mathcal{A}}}(\bar{\mathcal{A}}^n_0+
\bar{\mathcal{E}}^n - (\bar{\mathcal{D}}^{0,n} + \bar{\mathcal
{D}}^n))$, where the map $\Psi_{B^{\mathcal{A}}}\dvtx  \mathbb
{D}([0,T],\mathcal{S}') \mapsto\mathbb{D}([0,T],\mathcal{S}')$ is
continuous. The result
now follows by (\ref{usemapshow}) and Proposition~\ref{cmtprop}
applied to $\Psi_{B^\mathcal{A}}$.
\end{pf}

\begin{rem}
Note that one may now use Theorem {\ref{AgeFluidLimitTheorem}} along
with Theorem {\ref{ThmRegulatorMap}} in order to obtain an
explicit expression
for $\bar{\mathcal{A}}$. Similarly, one may obtain explicit expressions
for $\bar{\mathcal{R}}$, $\hat{\mathcal{A}}$ and $\hat{\mathcal{R}}$
in Theorems {\ref{ResidualFluidLimitThm}},
{\ref{AgeDiffusionLimitTheorem}} and
{\ref{ResidualDiffusionLimitTheorem}}, respectively, below.
\end{rem}

We also note that at a heuristic level, one may attempt to substitute
the function $\mathbf{1}_{ \{ x \geq0 \} }$ into the explicit formula
provided by Theorem~\ref{ThmRegulatorMap} for $\bar{\mathcal{A}}$ in
order to obtain an expression for the limiting, fluid scaled total
number of customers in the system. For instance, suppose that $\bar
{\mathcal{A}}_0=0$ so that the system is initially empty and that
$\langle\bar{\mathcal{E}}, \varphi\rangle= \lambda\varphi(0) e
$ for each $\varphi\in\mathcal{S}$. Then, using the form of the
generator $B^{\mathcal{A}}$ from (\ref{defBA}) and the semigroup
$(S^{\mathcal{A}}_t)_{t \geq0}$ from Proposition~\ref{PropAgeC01}, one
obtains after substituting into Theorem~\ref{ThmRegulatorMap} that
heuristically the total number of customers in the system at time $t
\geq0$ is given by
\begin{eqnarray*}
\lambda t - \lambda\int_0^t s f(t-s)\,ds &=&
\lambda\int_0^t \bar{F}(t-s)\,ds.
\end{eqnarray*}

We now conclude this subsection by providing an additional condition on
the arrival process under which a stationary solution to
the fluid limit equation (\ref{AgeFluidLimit}) may be explicitly
found. Note also that our condition in Proposition~\ref
{StationaryFluid} below holds, for example, if the arrival process
to the $n$th system is a renewal process which has been sped up by a
factor of $n$ (as will be the case for the $\mathit{GI}/\mathit{GI}/\infty$ queue).

%
\begin{prop}\label{StationaryFluid}
If $\langle\bar{\mathcal{E}}, \varphi\rangle= \lambda\varphi
(0) e $ for each $\varphi\in\mathcal{S}$, then $\bar{\mathcal{A}} =
\lambda\mathcal{F}_e$ is a stationary solution to the fluid
limit equation (\ref{AgeFluidLimit}).
\end{prop}

\begin{pf}
Substituting $\bar{\mathcal{A}} = \lambda\mathcal{F}_e$ and
$\langle\bar{\mathcal{E}}, \varphi\rangle= \lambda\varphi(0) e
$ into (\ref{AgeFluidLimit}), we see that it
suffices to verify that
\begin{eqnarray*}
&&\lambda\int_{\mathbb R_+} \varphi(y) \,dF_e(y)
\\
&&\qquad = \lambda\int_{\mathbb R_+} \varphi(y) \,dF_e(y) +
\lambda t \varphi(0) - \lambda t \int_{\mathbb R_+} \bigl(h(y)
\varphi(y) - \varphi'(y)\bigr) \,dF_e(y).
\end{eqnarray*}
However, this follows since
\begin{eqnarray*}
&& \lambda t \int_{\mathbb R_+} \bigl(h(y) \varphi(y) -
\varphi'(y)\bigr) \,dF_e(y)
\\
&&\qquad = \lambda t \int
_{\mathbb R_+} \bigl( h(y) \varphi(y) - \varphi'(y)
\bigr) \bar{F}(y) \,dy
\\
&&\qquad = \lambda t \int_{\mathbb R_+} \bigl(f(y) \varphi(y) - \bar{F}(y)
\varphi'(y)\bigr) \,dy
\\
&&\qquad = - \lambda t \int_{\mathbb R_+} \bigl(\bar{F}(y) \varphi(y)
\bigr)' \,dy
\\
&&\qquad = \lambda t \varphi(0).
\end{eqnarray*}
This completes the proof.
\end{pf}

\subsection{Residuals}\label{SubsecResidualsFluid}

We next proceed to analyze the residual service time process~$\mathcal
{R}$ of
Section~\ref{SubsecResidualEquations}. Our setup in this subsection is
the same as that in the previous subsection. However, in addition to
the fluid scaled quantities already defined in Section~\ref{SubsecAgesFluid}, we also now define for each $n \geq1$ the new fluid
scaled quantities
\[
\bar{\mathcal{R}}^n \equiv\frac{\mathcal{R}^n}{n}, \qquad\bar{
\mathcal{R}}_0^n \equiv\frac{\mathcal{R}_0^n}{n}\quad\mbox{and}
\quad\bar{\mathcal{G}} \equiv\frac{\mathcal{G}^n}{n}.
\]
Using (\ref{ResidualsSystemEquation}), Theorem
\ref{ThmRegulatorMap} and Proposition~\ref{PropResidualC01}, it is
now straightforward to show that
%
%
\begin{eqnarray}
\bar{ \mathcal{R}}^n &=& \Psi_{B^{\mathcal{R}}}\bigl( \bar{\mathcal
{R}}^n_0 + \bar{E}^n \mathcal{F}+ \bar{
\mathcal{G}}^n \bigr), \label{showresids}
\end{eqnarray}
where\vspace*{1pt} the map $\Psi_{B^{\mathcal{R}}}\dvtx \mathbb{D}([0,T],\mathcal
{S}') \mapsto\mathbb{D}([0,T],\mathcal{S}')$ is continuous. In our
first result of this subsection, we prove that if $(\bar{\mathcal
{R}}^n_0+ \bar{E}^n \mathcal{F})_{n \geq1}$ weakly converges, then
so too does $(\bar{\mathcal{R}}^n_0+ \bar{E}^n \mathcal{F} +\bar
{\mathcal{G}}^n )_{n \geq1}$.

\begin{prop}\label{Gfluidlimit}
If $ \bar{\mathcal{R}}^n_0 + \bar{E}^n \mathcal{F} \Rightarrow
\bar{\mathcal{R}}_0 + \bar{E}\mathcal{F} $ in $ \mathbb
{D}([0,T],\mathcal{S}')$ as $n
\rightarrow\infty$, then
%
%
\begin{equation}
\label{ResidualsJoint} \bar{\mathcal{R}}^n_0 +
\bar{E}^n \mathcal{F} + \bar{\mathcal{G}}^n \Rightarrow
\bar{\mathcal{R}}_0 + \bar{E} \mathcal{F} \qquad\mbox{in }\mathbb{D}
\bigl([0,T],\mathcal{S}'\bigr) \mbox{ as }n \rightarrow\infty.
\end{equation}
\end{prop}

\begin{pf}We first note that by Theorem~\ref{Mitoma}, it is sufficient
to show that if $ \bar{\mathcal{R}}^n_0 + \bar{E}^n \mathcal{F}
\Rightarrow
\bar{\mathcal{R}}_0 + \bar{E}\mathcal{F} $ as $n \rightarrow\infty
$, then
%
%
\begin{eqnarray}
\bar{\mathcal{G}}^n &\Rightarrow& 0 \qquad\mbox{in }\mathbb{D}
\bigl([0,T],\mathcal{S}'\bigr) \mbox{ as }n \rightarrow\infty.
\end{eqnarray}
Let\vspace*{1pt} $T > 0$ and $0 \leq t \leq T$. We then have by Proposition~\ref
{PropMartingalesTwo} and the assumption that $ \bar{\mathcal{R}}^n_0
+ \bar{E}^n \mathcal{F} \Rightarrow
\bar{\mathcal{R}}_0 + \bar{E}\mathcal{F} $, that for each $\varphi, \psi
\in\mathcal{S}$,
%
%
\begin{eqnarray}\label{eq4}
\bigl\llvert\bigl[ \bar{\mathcal{G}}^n \bigr]_t(
\varphi, \psi)\bigr\rrvert&=& \Biggl\llvert\frac{1}{n^2} \sum
_{i=1}^{E^n_t} \varphi(\eta_i)\psi(
\eta_i) \Biggr\rrvert
\nonumber\\[-8pt]\\[-8pt]
&\leq& \frac{1}{n^2} E^n_T \sup
_{0 \leq s < \infty}\bigl\llvert\varphi(s)\psi(s) \bigr\rrvert
\Rightarrow0 \qquad\mbox{in }\mathbb R\mbox{ as }n \rightarrow
\infty.\nonumber
\end{eqnarray}
The remainder of the proof now proceeds in a similar manner to the
proof of Proposition~\ref{Dfluidlimit}. We omit the details.
\end{pf}

The following is now our main result of this subsection.

\begin{teo}\label{ResidualFluidLimitThm}
If $ \bar{\mathcal{R}}^n_0 + \bar{E}^n \mathcal{F} \Rightarrow
\bar{\mathcal{R}}_0 + \bar{E}\mathcal{F} $ in $ \mathbb
{D}([0,T],\mathcal{S}')$ as $n
\rightarrow\infty$, then
\[
\bar{\mathcal{R}}^n \Rightarrow\bar{\mathcal{R}} \qquad\mbox{in }
\mathbb{D}\bigl([0,T],\mathcal{S}'\bigr) \mbox{ as }n
\rightarrow\infty,
\]
where $\bar{\mathcal{R}}$ is the unique solution to the integral
equation
%
%
\begin{equation}
\label{ResidualFluidLimit} \langle\bar{ \mathcal{R} }_t, \varphi\rangle=
\langle\bar{ \mathcal{R} }_0, \varphi\rangle+ \bar{E}_t
\langle\mathcal{F}, \varphi\rangle- \int_0^t
\bigl\langle\bar{ \mathcal{R} }_s, \varphi' \bigr
\rangle \,ds,\qquad t \in[0,T],
\end{equation}
for all $\varphi\in\mathcal{S}$.
\end{teo}

\begin{pf}
By the assumption that $\bar{\mathcal{R}}^n_0 + \bar{E}^n \mathcal
{F} \Rightarrow
\bar{\mathcal{R}}_0 + \bar{E}\mathcal{F} $, it follows immediately
by Proposition~\ref{Gfluidlimit} that
%
%
\begin{equation}
\label{ResidualsJointAgain} \bar{\mathcal{R}}^n_0 +
\bar{E}^n \mathcal{F} + \bar{\mathcal{G}}^n \Rightarrow
\bar{\mathcal{R}}_0 + \bar{E} \mathcal{F} \qquad\mbox{in }\mathbb{D}
\bigl([0,T],\mathcal{S}'\bigr) \mbox{ as }n \rightarrow\infty.
\end{equation}
Next, recall that by (\ref{showresids}) we have that
$\bar{\mathcal{R}}^n=\Psi_{B^{\mathcal{R}}}(\bar{\mathcal{R}}^n_0+
\bar{E}^n \mathcal{F}+ \bar{\mathcal{G}}^n)$, where the map
$\Psi_{B^{\mathcal{R}}}\dvtx \mathbb{D}([0,T],\mathcal{S}') \mapsto
\mathbb{D}([0,T],\mathcal{S}')$ is continuous. The result now follows
by (\ref{ResidualsJointAgain})
and Proposition~\ref{cmtprop} applied to
$\Psi_{B^{\mathcal{R}}}$.
\end{pf}

\section{Diffusion limits}\label{SecDiffusionLimits}

In this section, we prove our main diffusion limit results. In
Section~\ref{SubsecAgesDiffusion}, we
study the age process and in Section~\ref{SubsecResidualsDiffusion} we
study the residual service time process.
Before we provide our main results, however, we first
must provide the definition of an
$\mathcal{S}'$-valued Wiener process and a generalized $\mathcal{S}'$-valued
Ornstein--Uhlenbeck process. Our definitions are the same as those in
\cite{Langevin}.

%
\begin{definition}
A continuous $\mathcal{S}'$-valued Gaussian process $W = (W_t)_{t
\geq0}$ is called a \emph{generalized $\mathcal{S}'$-valued Wiener
process} with \emph{covariance functional}
\[
K(s,\varphi; t, \psi) = \mathbb E \bigl[\langle W_s, \varphi
\rangle\langle W_t, \psi\rangle\bigr],\qquad s,t \geq0\mbox{ and }
\varphi, \psi\in\mathcal{S},
\]
if it has continuous trajectories and, for each
$s,t \geq0$ and $\varphi, \psi\in\mathcal{S}$, $K(s, \varphi; t,
\psi)$
is of the form
\[
K(s,\varphi; t, \psi) = \int_0^{s \wedge t} \langle
Q_u \varphi, \psi\rangle \,du,
\]
where the operators $Q_u\dvtx  \mathcal{S}\rightarrow\mathcal{S}'$, $u
\geq0$, possess
the following two properties:
\begin{longlist}[(2)]
\item[(1)]$Q_u$ is linear, continuous, symmetric and positive for each $u
\geq0$,
\item[(2)] the function $u \mapsto\langle Q_u \varphi, \psi\rangle$ is in
$\mathbb{D}([0,\infty),\mathbb{R})$ for each $\varphi, \psi\in
\mathcal{S}$.
\end{longlist}
If $Q_u$ does not depend on $u \geq0$, then the process $W$ is called an
\emph{$\mathcal{S}'$-valued Wiener process}.
\end{definition}

Now, using the above definition of a generalized $\mathcal{S}'$-valued Wiener
process, we may provide the following definition of a generalized
$\mathcal{S}'$-valued Ornstein--Uhlenbeck process.

%
\begin{definition}
\label{defou}
An $\mathcal{S}'$-valued process $X = (X_t)_{t \geq0}$ is called a
(\emph{generalized}) \emph{$\mathcal{S}'$-valued Ornstein--Uhlenbeck process}
if for each $\varphi\in\mathcal{S}$ and $t \geq0$,
\[
\langle X_t, \varphi\rangle= \langle X_0, \varphi
\rangle+ \int_0^t \langle X_u, A
\varphi\rangle \,du + \langle W_t, \varphi\rangle,
\]
where $W \equiv(W_t)_{t \geq0}$ is a (generalized) $\mathcal{S}'$-valued
Wiener process and $A\dvtx \mathcal{S} \rightarrow\mathcal{S}$ is a
continuous operator.
\end{definition}

\subsection{Ages}\label{SubsecAgesDiffusion}
In this subsection, we prove our main diffusion limit result for the
age process $\mathcal{A}$ defined in Section~\ref{SubsecAgeEquations}. Our setup is the same as that in
Section~\ref{SecFluidLimits}. That is, we consider a sequence of $G/\mathit{GI}/\infty$
queues indexed by $n \geq1$, where the arrival rate to the system
grows large with $n$ while the service time distribution does not
change with $n$. For the remainder of this subsection, we assume that
$\bar{\mathcal{A}}^n_0 + \bar{\mathcal{E}}^n \Rightarrow\bar
{\mathcal{A}}_0 + \bar{\mathcal{E}}$ as $n \rightarrow\infty$,
where $\bar{\mathcal{A}}_0 + \bar{\mathcal{E}}$ is a nonrandom
quantity. By Theorem~\ref{AgeFluidLimitTheorem} of Section~\ref{SubsecAgesFluid}, this then implies that $\bar{\mathcal{A}}$ is a
nonrandom quantity as well. Setting $\bar{A}^n_0(\infty
)=n^{-1}A^n_0(\infty)$ for each $n \geq1$ and letting $T \geq0$, we
also assume
that the sequences $\{\bar{A}^n_0(\infty), n \geq1\}$ and $\{\bar
{E}^n_T, n \geq1\}$ are uniformly integrable.

Now, for each $n \geq1$, define the diffusion scaled quantities
\begin{eqnarray*}
\hat{\mathcal{A}}^n &\equiv&\sqrt{n} \bigl(\bar{\mathcal{A}}^n
- \bar{\mathcal{A}} \bigr), \qquad\hat{\mathcal{A}}_0^n
\equiv\sqrt{n} \bigl(\bar{\mathcal{A}}_0^n - \bar{
\mathcal{A}}_0 \bigr), \qquad\hat{E}^n \equiv\sqrt{n}
\bigl(\bar{E}^n - \bar{E} \bigr),
\\
\hat{\mathcal{D}}^{0,n} &\equiv&\sqrt{n} \bar{\mathcal{D}}^{0,n},
\qquad\hat{\mathcal{D}}^n \equiv\sqrt{n} \bar{\mathcal{D}}^n,
\end{eqnarray*}
and set $\hat{\mathcal{E}}^n \equiv\hat{E}^n \delta_0$. Then,
recalling the form of the fluid limit $\bar{\mathcal{A}}$ from
Theorem~\ref{AgeFluidLimitTheorem},
note that using system equation (\ref{AgesSystemEquation}) in
conjunction with Theorem~\ref{ThmRegulatorMap} and Proposition~\ref
{PropAgeC01}, one has that for each $n \geq1$,
%
%
\begin{equation}
\hat{\mathcal{A}}^n = \Psi_{B^{\mathcal{A}}}\bigl( \hat{\mathcal
{A}}^n_0 + \hat{\mathcal{E}}^n - \bigl(\hat{
\mathcal{D}}^{0,n} + \hat{\mathcal{D}}^n\bigr) \bigr),
\label{psidifage}
\end{equation}
where the map $\Psi_{B^{\mathcal{A}}}\dvtx  \mathbb{D}([0,T],\mathcal
{S}') \mapsto\mathbb{D}([0,T],\mathcal{S}')$ is continuous. Our
strategy now is to first prove a weak convergence result for the
sequence $( \hat{\mathcal{A}}^n_0 + \hat{\mathcal{E}}^n - (\hat
{\mathcal{D}}^{0,n} + \hat{\mathcal{D}}^n) )_{n \geq1}$
and then to apply Theorem~\ref{cmtprop} together with (\ref
{psidifage}) in order to prove our diffusion limit
result for the sequence $(\hat{\mathcal{A}}^n)_{n \geq1}$.

We begin with the following result. Its proof may be found in the \hyperref[appen]{Appendix}.

\begin{lem}\label{LemDcheck}
If $ \hat{\mathcal{A}}^n_0 + \hat{\mathcal{E}}^n
\Rightarrow\hat{\mathcal{A}}_0 + \hat{\mathcal{E}} $ in $\mathbb
{D}([0,T],\mathcal{S}') $
as $n \rightarrow\infty$, then
%
%
\begin{equation}
\label{DcheckDTwo} \hat{\mathcal{D}}^{0,n} + \hat{\mathcal{D}}^n
\Rightarrow\hat{\mathcal{D}}^0 + \hat{\mathcal{D}} \qquad\mbox{in }
\mathbb{D}\bigl([0,T],\mathcal{S}'\bigr) \mbox{ as }n
\rightarrow\infty,
\end{equation}
where $\hat{\mathcal{D}}^0 + \hat{\mathcal{D}}$ is a generalized
$\mathcal{S}'$-valued Wiener process with covariance functional given
for each $\varphi, \psi\in\mathcal{S}$ and $s,t \geq0$ by
%
%
\begin{equation}
\label{Dcovariance} K_{\hat{\mathcal{D}}^0 + \hat{\mathcal
{D}}}(s,\varphi; t, \psi) = \int
_0^{s \wedge t} \langle\bar{\mathcal{A}}_u,
\varphi\psi h \rangle \,du.
\end{equation}
\end{lem}

\begin{pf} See the \hyperref[appen]{Appendix}.
\end{pf}

We next have the following result, which provides a weak limit for the
sequence $( \hat{\mathcal{A}}^n_0+ \hat{\mathcal{E}}^n - (\hat
{\mathcal{D}}^{0,n} + \hat{\mathcal{D}}^n) )_{n \geq1}$.

%
\begin{prop}\label{Ddiffusionlimit}
If $ \hat{\mathcal{A}}^n_0 + \hat{\mathcal{E}}^n
\Rightarrow\hat{\mathcal{A}}_0 + \hat{\mathcal{E}} $ in $ \mathbb
{D}([0,T],\mathcal{S}') $
as $n \rightarrow\infty$, then
%
%
\begin{eqnarray}\label{jointdiffusion}
\hat{\mathcal{A}}^n_0 + \hat{
\mathcal{E}}^n -\bigl( \hat{\mathcal{D}}^{0,n} + \hat{
\mathcal{D}}^n\bigr) \Rightarrow\hat{\mathcal{A}}_0 +
\hat{\mathcal{E}} -\bigl( \hat{\mathcal{D}}^0 + \hat{\mathcal{D}}
\bigr)
\nonumber\\[-8pt]\\[-8pt]
\eqntext{\mbox{in } \mathbb{D}\bigl([0,T],\mathcal{S}'\bigr)
\mbox{ as }n \rightarrow\infty,}
\end{eqnarray}
where $\hat{\mathcal{D}}^0 + \hat{\mathcal{D}}$ is as given in
Lemma~\ref{LemDcheck} and is independent of $ \hat{\mathcal{A}}_0 +
\hat{\mathcal{E}}$.
\end{prop}

\begin{pf}We will check that the sequence $( \hat{\mathcal{A}}^n_0 +
\hat{\mathcal{E}}^n -( \hat{\mathcal{D}}^{0,n} + \hat{\mathcal
{D}}^n) )_{n \geq1}$ satisfies parts~(1) and~(2) of Theorem~\ref{Mitoma}.
We begin with part~(1). Let $\varphi\in\mathcal{S}$. By assumption,
the sequence $ (
\langle\hat{\mathcal{A}}^n_0 + \hat{\mathcal{E}}^n, \varphi
\rangle)_{n \geq1}$ weakly converges and hence is tight in $ \mathbb
{D}([0,T],\mathbb{R}) $, and, by Lemma~\ref{LemDcheck}, the sequence
$(\langle(\hat{\mathcal{D}}^{0,n} + \hat{\mathcal{D}}^n), \varphi
\rangle)_{n \geq1}$ weakly converges, and hence is tight in $ \mathbb
{D}([0,T],\mathbb{R}) $. Therefore, there must exist a subsequence
$(n)_{n \geq1}$ along which we have the joint convergence
\begin{eqnarray*}
\bigl( \bigl\langle\hat{\mathcal{A}}^n_0 + \hat{
\mathcal{E}}^n, \varphi\bigr\rangle, \bigl\langle\bigl(\hat{
\mathcal{D}}^{0,n} + \hat{\mathcal{D}}^n\bigr), \varphi\bigr
\rangle\bigr) &\Rightarrow& \bigl( \langle\check{\mathcal{A}}_0 +
\check{\mathcal{E}}, \varphi\rangle, \bigl\langle\bigl(\check{ \mathcal
{D}}^{0} + \check{\mathcal{D}}\bigr), \varphi\bigr\rangle\bigr)
\end{eqnarray*}
in $\mathbb{D}^2([0,T],\mathbb{R})$ as $n \rightarrow
\infty$. Now note that clearly $\langle\check{\mathcal{A}}_0 +
\check{\mathcal{E}}, \varphi\rangle$ has the same distribution as
$\langle\hat{\mathcal{A}}_0 + \hat{\mathcal{E}}, \varphi\rangle
$ and, similarly, $\langle(\check{ \mathcal{D}}^{0} + \check
{\mathcal{D}}), \varphi\rangle$ has the same distribution as
$\langle(\hat{ \mathcal{D}}^{0} + \hat{\mathcal{D}}), \varphi
\rangle$. We now verify that
$\langle\check{\mathcal{A}}_0 + \check{\mathcal{E}}, \varphi
\rangle$ and $\langle(\check{ \mathcal{D}}^{0} + \check{\mathcal
{D}}), \varphi\rangle$ are independent of one another. This will then
imply the convergence
\begin{eqnarray*}
\bigl\langle\hat{\mathcal{A}}^n_0 + \hat{
\mathcal{E}}^n, \varphi\bigr\rangle- \bigl\langle\bigl(\hat{
\mathcal{D}}^{0,n} + \hat{\mathcal{D}}^n\bigr), \varphi\bigr
\rangle&\Rightarrow& \langle\hat{\mathcal{A}}_0 + \hat{\mathcal{E}},
\varphi\rangle- \bigl\langle\bigl(\hat{ \mathcal{D}}^{0} + \hat{
\mathcal{D}}\bigr), \varphi\bigr\rangle
\end{eqnarray*}
in $ \mathbb{D}([0,T],\mathbb{R})$ as $n \rightarrow
\infty$, along the given subsequence. However, since the subsequence
was arbitrary, this will then imply convergence along the entire
sequence, thus verifying part~(1) of Theorem~\ref{Mitoma}.

Let $t_1,t_2 \in[0,T]$ with $t_1 \leq t_2$ and let $a_1, a_2,b_1,b_2
\in\mathbb{R}$ and let $x,y \in\mathbb{R}$. We will show that
%
%
\begin{eqnarray}\label{desiredresult}
&& \mathbb{P}\bigl(a_1 \bigl\langle\hat{\mathcal{A}}^n_0
+ \hat{\mathcal{E}}^n_{t_1}, \varphi\bigr\rangle+
a_2 \bigl\langle\hat{\mathcal{A}}^n_0 + \hat{
\mathcal{E}}^n_{t_2}, \varphi\bigr\rangle\leq x,\nonumber
\\
&&\hspace*{10pt} b_1 \bigl\langle\bigl(\hat{ \mathcal{D}}^{0,n}_{t_1}
+ \hat{\mathcal{D}}_{t_1}^n\bigr), \varphi\bigr
\rangle+b_2 \bigl\langle\bigl(\hat{ \mathcal{D}}^{0,n}_{t_2}
+ \hat{\mathcal{D}}_{t_2}^n\bigr), \varphi\bigr\rangle\leq
y\bigr)
\nonumber\\[-8pt]\\[-8pt]\nonumber
&&\qquad \rightarrow \mathbb{P}\bigl(a_1 \langle\hat{
\mathcal{A}}_0 + \hat{\mathcal{E}}_{t_1}, \varphi\rangle+
a_2 \langle\hat{\mathcal{A}}_0 + \hat{
\mathcal{E}}_{t_2}, \varphi\rangle\leq x\bigr)
\\
&&\hspace*{32pt}{}\times\mathbb{P}\bigl(b_1 \bigl\langle\bigl(\hat{
\mathcal{D}}^{0}_{t_1} + \hat{\mathcal{D}}_{t_1}
\bigr), \varphi\bigr\rangle+b_2 \bigl\langle\bigl(\hat{
\mathcal{D}}^{0}_{t_2} + \hat{\mathcal{D}}_{t_2}
\bigr), \varphi\bigr\rangle\leq y\bigr)
\nonumber
\end{eqnarray}
as $n \rightarrow\infty$. The analogous proof for $t_1,\ldots,t_m
\in[0,T]$ with $m> 2$ follows similarly. This will then be sufficient
to show that $\langle\check{\mathcal{A}}_0 + \check{\mathcal{E}},
\varphi\rangle$ and $\langle(\check{ \mathcal{D}}^{0} + \check
{\mathcal{D}}), \varphi\rangle$ are independent of one another.
First note that we may write
\begin{eqnarray*}
&&\mathbb{P}\bigl(a_1 \bigl\langle\hat{\mathcal{A}}^n_0
+ \hat{\mathcal{E}}^n_{t_1}, \varphi\bigr\rangle+
a_2 \bigl\langle\hat{\mathcal{A}}^n_0 + \hat{
\mathcal{E}}^n_{t_2}, \varphi\bigr\rangle\leq x,
\\
&&\hspace*{11pt} b_1 \bigl\langle\bigl(\hat{ \mathcal{D}}^{0,n}_{t_1}
+ \hat{\mathcal{D}}_{t_1}^n\bigr), \varphi\bigr
\rangle+b_2 \bigl\langle\bigl(\hat{ \mathcal{D}}^{0,n}_{t_2}
+ \hat{\mathcal{D}}_{t_2}^n\bigr), \varphi\bigr\rangle\leq
y\bigr)
\\
&&\qquad = \mathbb E [\mathbf{1}_{ \{ a_1 \langle\hat{\mathcal
{A}}^n_0 + \hat{\mathcal{E}}^n_{t_1}, \varphi\rangle+ a_2 \langle
\hat{\mathcal{A}}^n_0 + \hat{\mathcal{E}}^n_{t_2}, \varphi\rangle
\leq x \} } \mathbf{1}_{ \{ b_1 \langle(\hat{ \mathcal
{D}}^{0,n}_{t_1} + \hat{\mathcal{D}}_{t_1}^n), \varphi\rangle+b_2
\langle(\hat{ \mathcal{D}}^{0,n}_{t_2} + \hat{\mathcal
{D}}_{t_2}^n), \varphi\rangle\leq y \} } ].
\end{eqnarray*}
However, by the tower property of conditional expectations \cite
{chung2001course}, we have that
\begin{eqnarray*}
&& \mathbb E [\mathbf{1}_{ \{ a_1 \langle\hat{\mathcal{A}}^n_0
+ \hat{\mathcal{E}}^n_{t_1}, \varphi\rangle+ a_2 \langle\hat
{\mathcal{A}}^n_0 + \hat{\mathcal{E}}^n_{t_2}, \varphi\rangle\leq
x \} } \mathbf{1}_{ \{ b_1 \langle(\hat{ \mathcal{D}}^{0,n}_{t_1} +
\hat{\mathcal{D}}_{t_1}^n), \varphi\rangle+b_2 \langle(\hat{
\mathcal{D}}^{0,n}_{t_2} + \hat{\mathcal{D}}_{t_2}^n), \varphi
\rangle\leq y \} } ]
\\
&&\qquad = \mathbb E \bigl[\mathbf{1}_{ \{ a_1 \langle\hat{\mathcal
{A}}^n_0 + \hat{\mathcal{E}}^n_{t_1}, \varphi\rangle+ a_2 \langle
\hat{\mathcal{A}}^n_0 + \hat{\mathcal{E}}^n_{t_2}, \varphi\rangle
\leq x \} }
\\
&&\hspace*{42pt}{} \times\mathbb E \bigl[ \mathbf{1}_{ \{ b_1 \langle(\hat
{ \mathcal{D}}^{0,n}_{t_1} + \hat{\mathcal{D}}_{t_1}^n), \varphi
\rangle+b_2 \langle(\hat{ \mathcal{D}}^{0,n}_{t_2} + \hat{\mathcal
{D}}_{t_2}^n), \varphi\rangle\leq y \} }|
\mathcal{A}^n,\mathcal{E}^n \bigr] \bigr].
\end{eqnarray*}
We now claim that
%
%
\begin{eqnarray}\label{conuseberry2}
&&\mathbb E \bigl[ \mathbf{1}_{ \{ b_1 \langle(\hat{ \mathcal
{D}}^{0,n}_{t_1} + \hat{\mathcal{D}}_{t_1}^n), \varphi\rangle+b_2
\langle(\hat{ \mathcal{D}}^{0,n}_{t_2} + \hat{\mathcal
{D}}_{t_2}^n), \varphi\rangle\leq y \} }| \mathcal{A}^n,
\mathcal{E}^n \bigr]
\nonumber\\[-8pt]\\[-8pt]
&&\qquad \stackrel{\mathbb{P}} {\rightarrow}\mathbb{P}\bigl(b_1 \bigl
\langle\bigl(\hat{ \mathcal{D}}^{0}_{t_1} + \hat{
\mathcal{D}}_{t_1}\bigr), \varphi\bigr\rangle+b_2 \bigl
\langle\bigl(\hat{ \mathcal{D}}^{0}_{t_2} + \hat{\mathcal
{D}}_{t_2}\bigr), \varphi\bigr\rangle\leq y \bigr)
\nonumber
\end{eqnarray}
as $n \rightarrow\infty$. Then, since by assumption
\begin{eqnarray*}
&&\mathbb E [\mathbf{1}_{ \{ a_1 \langle\hat{\mathcal{A}}^n_0
+ \hat{\mathcal{E}}^n_{t_1}, \varphi\rangle+ a_2 \langle\hat
{\mathcal{A}}^n_0 + \hat{\mathcal{E}}^n_{t_2}, \varphi\rangle\leq
x \} } ]
\\
&&\qquad \rightarrow\mathbb{P}\bigl(a_1 \langle\hat{\mathcal{A}}_0
+ \hat{\mathcal{E}}_{t_1}, \varphi\rangle+ a_2 \langle
\hat{\mathcal{A}}_0 + \hat{\mathcal{E}}_{t_2}, \varphi
\rangle\leq x \bigr)
\end{eqnarray*}
as $n \rightarrow\infty$, this will then imply (\ref
{desiredresult}), thus verifying part~(1) of Theorem~\ref{Mitoma}.

In order to see that (\ref{conuseberry2}) holds, first note that
\begin{eqnarray*}
&&\mathbb E \bigl[ \mathbf{1}_{ \{ b_1 \langle(\hat{ \mathcal
{D}}^{0,n}_{t_1} + \hat{\mathcal{D}}_{t_1}^n), \varphi\rangle+b_2
\langle(\hat{ \mathcal{D}}^{0,n}_{t_2} + \hat{\mathcal
{D}}_{t_2}^n), \varphi\rangle\leq y \} }| \mathcal{A}^n,
\mathcal{E}^n \bigr] \label{conuseberry}
\\
&&\qquad =\mathbb{P} \bigl( b_1 \bigl\langle\bigl(\hat{
\mathcal{D}}^{0,n}_{t_1} + \hat{\mathcal{D}}_{t_1}^n
\bigr), \varphi\bigr\rangle+b_2 \bigl\langle\bigl(\hat{
\mathcal{D}}^{0,n}_{t_2} + \hat{\mathcal{D}}_{t_2}^n
\bigr), \varphi\bigr\rangle\leq y | \mathcal{A}^n,
\mathcal{E}^n \bigr).
\end{eqnarray*}
Now recall from (\ref{decompose}) that we may write
\begin{eqnarray*}
&& b_1 \bigl\langle\bigl(\hat{ \mathcal{D}}^{0,n}_{t_1}
+ \hat{\mathcal{D}}_{t_1}^n\bigr), \varphi\bigr
\rangle+b_2 \bigl\langle\bigl(\hat{ \mathcal{D}}^{0,n}_{t_2}
+ \hat{\mathcal{D}}_{t_2}^n\bigr), \varphi\bigr\rangle
\\
&&\qquad =\sum_{i=1}^{A_0^n(\infty)} \bigl( \bigl\langle
\hat{\mathcal{D}}^{0,n,i}_{t_1}, b_1 \varphi\bigr
\rangle+ \bigl\langle\hat{\mathcal{D}}^{0,n,i}_{t_2},
b_2 \varphi\bigr\rangle\bigr)
+\sum_{i=1}^{E^n_{t_2}} \bigl( \bigl\langle
\hat{\mathcal{D}}^{n,i}_{t_1}, b_1 \varphi\bigr
\rangle+ \bigl\langle\hat{\mathcal{D}}^{n,i}_{t_2},
b_2 \varphi\bigr\rangle\bigr).
\end{eqnarray*}
Moreover, given $\mathcal{A}^n$ and $\mathcal{E}^n$, we have that the
random variables $( \langle\hat{\mathcal{D}}^{0,n,i}_{t_1}, b_1
\varphi\rangle+
\langle\hat{\mathcal{D}}^{0,n,i}_{t_2}, b_2 \varphi\rangle
),i=1,\ldots,A^n_0(\infty)$, and $( \langle\hat{\mathcal
{D}}^{n,i}_{t_1}, b_1 \varphi\rangle+
\langle\hat{\mathcal{D}}^{n,i}_{t_2}, b_2 \varphi\rangle
),i=1,\ldots,E^n_{t_2}$, are mutually independent, with mean zero.
In addition, it is straightforward to calculate that
%
%
\begin{eqnarray}\label{calvar}
&&E \Biggl[\sum_{i=1}^{A_0^n(\infty)} \bigl( \bigl
\langle\hat{\mathcal{D}}^{0,n,i}_{t_1}, b_1
\varphi\bigr\rangle+ \bigl\langle\hat{\mathcal{D}}^{0,n,i}_{t_2}, b_2
\varphi\bigr\rangle\bigr)^2 \Big| \mathcal{A}^n_0,
\mathcal{E}^n \Biggr] \nonumber
\\
&&\quad{}+E \Biggl[\sum_{i=1}^{E^n_{t_2}} \bigl( \bigl
\langle\hat{\mathcal{D}}^{n,i}_{t_1}, b_1
\varphi\bigr\rangle+ \bigl\langle\hat{\mathcal{D}}^{n,i}_{t_2}, b_2
\varphi\bigr\rangle\bigr)^2 \Big| \mathcal{A}^n_0,
\mathcal{E}^n \Biggr]
\\
&&\qquad=E \biggl[\int_0^{t_1} \bigl\langle\bar{
\mathcal{A}}^n_u, (b_1 \varphi
+b_2 \varphi)^2 h \bigr\rangle \,du \Big| \mathcal{A}^n_0,
\mathcal{E}^n \biggr].
\nonumber
\end{eqnarray}
Also, note that for each $i=1,\ldots,A^n_0(\infty)$,
%
%
\begin{eqnarray}\label{thirdmomfirst}
&&E \bigl[\bigl( \bigl\langle\hat{\mathcal{D}}^{0,n,i}_{t_1},
b_1 \varphi\bigr\rangle+ \bigl\langle\hat{\mathcal{D}}^{0,n,i}_{t_2},
b_2 \varphi\bigr\rangle\bigr)^3 | \mathcal{A}^n_0,
\mathcal{E}^n \bigr] \nonumber
\\
&&\qquad =E \bigl[\bigl( \bigl\langle\hat{\mathcal{D}}^{0,n,i}_{t_1},
b_1 \varphi\bigr\rangle+ \bigl\langle\hat{\mathcal{D}}^{0,n,i}_{t_2},
b_2 \varphi\bigr\rangle\bigr)^2 \bigl( \bigl\langle
\hat{\mathcal{D}}^{0,n,i}_{t_1}, b_1 \varphi\bigr
\rangle+ \bigl\langle\hat{\mathcal{D}}^{0,n,i}_{t_2},
b_2 \varphi\bigr\rangle\bigr) | \mathcal{A}^n_0,
\mathcal{E}^n \bigr]\hspace*{-10pt}
\nonumber\\[-8pt]\\[-8pt]
&&\qquad \leq\frac{(\llvert b_1\rrvert+\llvert b_2\rrvert)\llVert \varphi
\rrVert _{\infty}(1+t_2 \llVert h\rrVert _{\infty})}{\sqrt
{n}}\nonumber
\\
&&\quad\qquad{}\times E \bigl[\bigl( \bigl\langle\hat{\mathcal{D}}^{0,n,i}_{t_1}, b_1
\varphi\bigr\rangle+ \bigl\langle\hat{\mathcal{D}}^{0,n,i}_{t_2}, b_2
\varphi\bigr\rangle\bigr)^2 | \mathcal{A}^n_0,
\mathcal{E}^n \bigr],
\nonumber
\end{eqnarray}
and, similarly, for each $i=1,\ldots,E^n_{t_2}$,
%
%
\begin{eqnarray}\label{thirdmomsec}
&&E \bigl[\bigl( \bigl\langle\hat{\mathcal{D}}^{n,i}_{t_1},
b_1 \varphi\bigr\rangle+ \bigl\langle\hat{\mathcal{D}}^{0,n,i}_{t_2},
b_2 \varphi\bigr\rangle\bigr)^3 | \mathcal{A}^n_0,
\mathcal{E}^n \bigr] \nonumber
\\
&&\qquad \leq\frac{(\llvert b_1\rrvert+\llvert b_2\rrvert)\llVert \varphi
\rrVert _{\infty}(1+t_2 \llVert h\rrVert _{\infty})}{\sqrt
{n}}
\\
&&\quad\qquad{}\times E \bigl[\bigl( \bigl\langle\hat{\mathcal{D}}^{n,i}_{t_1}, b_1
\varphi\bigr\rangle+ \bigl\langle\hat{\mathcal{D}}^{n,i}_{t_2}, b_2
\varphi\bigr\rangle\bigr)^2 | \mathcal{A}^n_0,
\mathcal{E}^n \bigr].
\nonumber
\end{eqnarray}
Now let $\Phi$ denote the c.d.f. of a standard, normal random variable.
It then follows by (\ref{calvar}), (\ref{thirdmomfirst}), (\ref
{thirdmomsec}) and an application of the Berry--Esseen theorem \cite
{chung2001course} for independent (but not necessarily identically
distributed) random variables that
\begin{eqnarray*}
&&\biggl\llvert\mathbb{P} \biggl(\frac{ b_1 \langle(\hat{ \mathcal
{D}}^{0,n}_{t_1} + \hat{\mathcal{D}}_{t_1}^n), \varphi\rangle+b_2
\langle(\hat{ \mathcal{D}}^{0,n}_{t_2} + \hat{\mathcal
{D}}_{t_2}^n), \varphi\rangle}{ ( \mathbb E [\int_0^{t_1}
\langle\bar{\mathcal{A}}^n_u, (b_1 \varphi+b_2 \varphi)^2 h
\rangle \,du | \mathcal{A}^n_0,\mathcal{E}^n ] )^{1/2} } \leq y |
\mathcal{A}^n,\mathcal{E}^n \biggr)- \Phi(y) \biggr\rrvert
\\
&&\qquad \leq\frac{1}{\sqrt{n}} \frac{(\llvert b_1\rrvert+\llvert b_2\rrvert
)\llVert \varphi\rrVert _{\infty}(1+t_2 \llVert h\rrVert
_{\infty})}{ ( \mathbb E [\int_0^{t_1} \langle\bar
{\mathcal{A}}^n_u, (b_1 \varphi+b_2 \varphi)^2 h \rangle \,du |
\mathcal{A}^n_0,\mathcal{E}^n ] )^{1/2}}.
\end{eqnarray*}
Hence, in order to complete the proof of (\ref{conuseberry2}) and
hence verify part~(1) of Theorem~\ref{Mitoma}, it suffices to show that
%
%
\begin{eqnarray} \label{imply53}
\qquad && E \biggl[\int_0^{t_1} \bigl\langle\bar{
\mathcal{A}}^n_u, (b_1 \varphi
+b_2 \varphi)^2 h \bigr\rangle \,du \Big| \mathcal{A}^n_0,
\mathcal{E}^n \biggr]
\stackrel{\mathbb{P}} {\rightarrow}\int_0^{t_1}
\bigl\langle\bar{\mathcal{A}}_u, (b_1
\varphi+b_2 \varphi)^2 h \bigr\rangle \,du
\nonumber\\[-8pt]\\[-8pt]
\eqntext{\mbox{as }n \rightarrow\infty.}
\end{eqnarray}
However, note that by Theorem~\ref{AgeFluidLimitTheorem} and the
continuity of the integral map \cite{Billingsley99}, we have that
%
%
\begin{eqnarray}
\int_0^{t_1} \bigl\langle\bar{
\mathcal{A}}^n_u, (b_1
\varphi+b_2 \varphi)^2 h \bigr\rangle \,du &\stackrel{
\mathbb{P}} {\rightarrow}&\int_0^{t_1} \bigl\langle
\bar{\mathcal{A}}_u, (b_1 \varphi+b_2
\varphi)^2 h \bigr\rangle \,du, \label{imply56}
\end{eqnarray}
as $n \rightarrow\infty$. Next, note that the uniform integrability
of $\{\bar{A}^n_0(\infty), n \geq1\}$ and $\{\bar{E}^n_T, n \geq1\}
$ implies
the uniform integrability of
%
%
\begin{equation}
\biggl\{ \int_0^{t_1} \bigl\langle\bar{
\mathcal{A}}^n_u, (b_1 \varphi
+b_2 \varphi)^2 h \bigr\rangle \,du, n \geq1 \biggr\}.
\end{equation}
It is then straightforward to show that (\ref{imply56}) implies (\ref
{imply53}), thus completing the verification of part~(1) of Theorem~\ref
{Mitoma}. The proof of the verification of part~(2) of Theorem~\ref
{Mitoma} follows in a similar manner to the above and has been omitted
for the sake of brevity.
This completes the proof.
\end{pf}

The following is now our main result of this section. Its proof is a
straightforward consequence
of Theorem~\ref{cmtprop}, (\ref{psidifage}) and Proposition~\ref
{Ddiffusionlimit}.

%
\begin{teo}\label{AgeDiffusionLimitTheorem}
If $ \hat{\mathcal{A}}_0^n + \hat{\mathcal{E}}^n \Rightarrow
\hat{\mathcal{A}}_0 + \hat{\mathcal{E}} $ in $\mathbb
{D}([0,T],\mathcal{S}') $ as $n
\rightarrow\infty$, then
%
%
\begin{equation}
\hat{\mathcal{A}}^n \Rightarrow\hat{\mathcal{A}}\qquad\mbox{in }
\mathbb{D}\bigl([0,T],\mathcal{S}'\bigr)\mbox{ as } n
\rightarrow\infty, \label{weakdiffage}
\end{equation}
where $\hat{\mathcal{A}}$ is the solution to the stochastic integral equation
%
%
\begin{eqnarray}\label{AgeDiffusionLimit}
\langle\hat{ \mathcal{A} }_t, \varphi\rangle&=&
\langle\hat{\mathcal{A}}_0, \varphi\rangle+ \langle\hat{
\mathcal{E}}_t, \varphi\rangle- \bigl\langle\hat{
\mathcal{D}}^0 + \hat{ \mathcal{D } }, \varphi\bigr\rangle
\nonumber\\[-8pt]\\[-8pt]\nonumber
&&{} - \int
_0^t \langle\hat{ \mathcal{A} }_s,
h \varphi\rangle \,ds + \int_0^t \bigl\langle
\hat{ \mathcal{A} }_s, \varphi' \bigr\rangle \,ds,
\end{eqnarray}
for each $t \in[0,T]$ and $\varphi\in\mathcal{S}$. In addition, if
$\hat{\mathcal{E}}$ is an $\mathcal{S}'$-valued Wiener
process with covariance functional $K_{\hat{\mathcal{E}}}(s,\varphi;
t, \psi) = \sigma^2 (s \wedge t) \varphi(0) \psi(0)$, then $\hat
{\mathcal{A}}$
is a generalized $\mathcal{S}'$-valued Ornstein--Uhlenbeck process driven
by a generalized $\mathcal{S}'$-valued Wiener process with covariance
functional
%
%
\begin{equation}
\label{drivingcovariance}
\qquad K_{\hat{\mathcal{E}} - (\hat{\mathcal{D}}^0 +
\hat{\mathcal
{D}})}(s,\varphi; t, \psi) = \sigma^2 (s
\wedge t) \varphi(0) \psi(0) +\int_0^{s \wedge t} \langle
\bar{\mathcal{A}}_u h, \varphi\psi\rangle \,du.
\end{equation}
\end{teo}

\begin{pf}
First note that by (\ref{psidifage}) we have that
$\hat{\mathcal{A}}^n=\Psi_{B^{\mathcal{A}}}(\hat{\mathcal{A}}^n_0+
\hat{\mathcal{E}}^n -( \hat{\mathcal{D}}^{0,n} +
\hat{\mathcal{D}}^n))$, where the map $\Psi_{B^{\mathcal{A}}}\dvtx
\mathbb{D}([0,T],\mathcal{S}') \mapsto\mathbb{D}([0,T],\mathcal
{S}')$ is a continuous map. The convergence (\ref{weakdiffage}) now
follows by Theorem~\ref{cmtprop} and Proposition~\ref{Ddiffusionlimit}.\vspace*{1pt}

Next, suppose that $\hat{\mathcal{E}}$ is an $\mathcal{S}'$-valued Wiener
process with covariance functional $K_{\hat{\mathcal{E}}}(s,\varphi;
t, \psi) = \sigma^2 (s \wedge t) \varphi(0) \psi(0)$. Then,
combining this
with (\ref{Dcovariance}) and the fact that
$\hat{\mathcal{D}}^0+\hat{\mathcal{D}}$ and $\hat{\mathcal
{A}}_{0}+\hat{\mathcal{E}}$ are
independent from Proposition~\ref{Ddiffusionlimit}, yields
(\ref{drivingcovariance}). Thus, by Definition~\ref{defou}, $\hat
{\mathcal{A}}$ is an
$\mathcal{S}'$-valued Ornstein--Uhlenbeck process.
\end{pf}

Recall now from Proposition~\ref{StationaryFluid} of Section~\ref
{SubsecAgesFluid} that if $\langle\bar{\mathcal{E}}, \varphi
\rangle= \lambda\varphi(0) e $ for each $\varphi\in\mathcal{S}$
and some $\lambda\geq0$, then a stationary solution to
the fluid limit equation (\ref{AgeFluidLimit}) is given by $\bar{\mathcal
{A}} =
\lambda\mathcal{F}_e$. We now\vspace*{1pt} show that under the additional
condition that
$\hat{\mathcal{E}}$ is an $\mathcal{S}'$-valued Wiener
process with covariance functional $K_{\hat{\mathcal{E}}}(s,\varphi;
t, \psi) = \sigma^2 (s \wedge t) \varphi(0) \psi(0)$, then the
resulting limiting diffusion scaled age
process $\hat{\mathcal{A}}$ of Theorem~\ref
{AgeDiffusionLimitTheorem} is a
time-homogeneous Markov process. Our result is the following. Note also
that a similar approach may be used to analyze the
diffusion limit of the residual service time process in Theorem
\ref{ResidualDiffusionLimitTheorem} in the following subsection.

%
\begin{prop}\label{AgeMarkov}
If $\langle\bar{\mathcal{E}}, \varphi\rangle= \lambda\varphi
(0) e $ for each $\varphi\in\mathcal{S}$, $\bar{\mathcal{A}}_0 =
\lambda
\mathcal{F}_e$ and $\hat{\mathcal{E}}$ is an \mbox{$\mathcal{S}'$-}valued Wiener
process with covariance functional $K_{\hat{\mathcal{E}}}(s,\varphi;
t, \psi) = \sigma^2 (s \wedge t) \varphi(0) \psi(0)$, then $\hat
{\mathcal{A}}$ is an \mbox{$\mathcal{S}'$-}valued
Ornstein--Uhlenbeck process driven by an \mbox{$\mathcal{S}'$-}valued Wiener
process with covariance functional given
for each $\varphi, \psi\in\mathcal{S}$ and $s,t \geq0$ by
%
%
\begin{equation}
\label{AgeDrivingCovar} K_{\hat{\mathcal{E}} - ( \hat{\mathcal{D}}^0 +
\hat{\mathcal{D}}
)}(s, \varphi; t, \psi) = (s \wedge t) \bigl
\langle\sigma^2 \delta_0 + \lambda\mathcal{F}, \varphi\psi
\bigr\rangle.
\end{equation}
\end{prop}

\begin{pf}
It is clear that the covariance functional of $\hat{\mathcal{E}} $ is given
by
%
%
\begin{equation}
\label{ECovar} K_{\hat{\mathcal{E}}}(s,\varphi; t, \psi) = (s \wedge t)
\bigl
\langle\sigma^2\delta_0, \varphi\psi\bigr\rangle.
\end{equation}
We now show that the covariance functional of $\hat{\mathcal{D}}^0 +
\hat{\mathcal{D}}$ is given by
%
%
\begin{equation}
\label{DCovar} K_{\hat{\mathcal{D}}^0 + \hat{\mathcal{D}}}(s, \varphi;
t, \psi) = \lambda( s \wedge t )
\langle\mathcal{F}, \varphi\psi\rangle.
\end{equation}
Then, since $\hat{\mathcal{E}} $ and $\hat{\mathcal{D}}^0 + \hat
{\mathcal{D}}$ are independent, summing (\ref{ECovar}) and (\ref
{DCovar}) will prove~(\ref{AgeDrivingCovar}), which will complete the proof.

Note that by Proposition~\ref{StationaryFluid} we have that since by
assumption $\bar{\mathcal{A}}_0 =
\lambda\mathcal{F}_e$, it follows that $\bar{\mathcal{A}} = \lambda
\mathcal{F}_e$ is the unique solution to the fluid limit equation
(\ref{AgeFluidLimit}). Therefore, by Lemma
\ref{LemDcheck} we have that for each $\varphi, \psi\in\mathcal
{S}$ and $s, t \geq0$,
\begin{eqnarray*}
K_{\hat{\mathcal{D}}^0 + \hat{\mathcal{D}}}(s, \varphi; t, \psi) &=& \int
_0^{s \wedge t}
\langle\bar{\mathcal{A}}_u, \varphi\psi h \rangle \,du
\\
&=& \int_0^{s \wedge t} \int_{\mathbb R_+}
\varphi(x) \psi(x) h(x) \,d\bar{A}_u(x) \,du
\\
&=& \lambda(s \wedge t) \int_{\mathbb R_+} \varphi(x) \psi(x) h(x)
\bar{F}(x) \,dx
\\
&=& \lambda(s \wedge t) \int_{\mathbb R_+} \varphi(x) \psi(x) f(x)
\,dx.
\end{eqnarray*}
This proves (\ref{DCovar}), which completes the proof.
\end{pf}

We now note that using (\ref{integralsolution}) of Theorem~\ref
{ThmRegulatorMap} and (\ref{semigroupdef}) of Proposition~\ref{PropAgeC01},
one may obtain an explicit representation of $\hat{\mathcal{A}}$ in
terms of the $\mathcal{S}'$-valued Wiener process given in Proposition
\ref{AgeMarkov} above. Direct calculations may then be used in order
to obtain the
transient and limiting distribution of $\hat{\mathcal{A}}$. In
particular, assuming that $\hat{\mathcal{A}}_0$ is a
Gaussian random variable, one may then show that for each $t \in
[0,T]$, $\hat{\mathcal{A}}_t$ is a Gaussian random
variable with mean
%
%
\begin{equation}
\label{AgeTransientMean} \mathbb E\bigl[ \langle\hat{\mathcal{A}}_t,
\varphi
\rangle\bigr] = \bigl\langle\hat{\mathcal{A}}_0, \bar{F}^{-1}
\tau_{-t} (\varphi\bar{F} ) \bigr\rangle,\qquad\varphi\in\mathcal{S},
\end{equation}
and covariance functional given for each $\varphi, \psi\in\mathcal
{S}$ and $t \in[0,T]$ by
%
%
\begin{eqnarray}\label{AgeTransientCovar}
\mathbb E\bigl[ \langle\hat{\mathcal{A}}_t, \varphi\rangle\langle
\hat{\mathcal{A}}_t, \psi\rangle\bigr] &=&\lambda\bigl\langle
\mathcal{F}_e, \bar{F}^{-1} \tau_{-t} (\varphi
\psi\bar{F} ) \bigl(1 - \bar{F}^{-1} \tau_{-t} \bar{F} \bigr)
\bigr\rangle
\nonumber\\[-8pt]\\[-8pt]\nonumber
&&{}+ \int_0^t \varphi(u) \psi(u) \bigl(\lambda
F(u) + \sigma^2 \bar{F}(u) \bigr) \bar{F}(u) \,du.
\end{eqnarray}
In addition, taking limits as $t \rightarrow\infty$, one also finds
that $\hat{\mathcal{A}}_t$ weakly converges as $t \rightarrow\infty
$ to a Gaussian random variable $\mathcal{A}_{\infty}$ with mean zero
and covariance functional given for each $\varphi, \psi\in\mathcal
{S}$ by
\[
\mathbb E\bigl[ \langle\hat{\mathcal{A}}_\infty, \varphi\rangle\langle
\hat{\mathcal{A}}_\infty, \psi\rangle\bigr] = \bigl\langle
\mathcal{F}_e, \bigl(\lambda F + \sigma^2 \bar{F} \bigr)
\varphi\psi\bigr\rangle.
\]

We now conclude this subsection by noting that one may heuristically
attempt to substitute the test function $\mathbf{1}_{ \{ x \geq0 \}
}$ into the formula for $\hat{\mathcal{A}}$ provided
by Theorem~\ref{ThmRegulatorMap} in order to obtain an expression for
the limiting diffusion scaled total number of customers in the system.
For instance,\vspace*{2pt} suppose that $\hat{\mathcal{A}}_0=0$ and that, as in
the statement of Proposition~\ref{AgeMarkov}, we have that $\langle
\bar{\mathcal{E}}, \varphi\rangle= \lambda\varphi(0) e $ for
each $\varphi\in\mathcal{S}$ and $\bar{\mathcal{A}}_0 = \lambda
\mathcal{F}_e$ and $\hat{\mathcal{E}}$ is an $\mathcal{S}'$-valued Wiener
process with covariance functional $K_{\hat{\mathcal{E}}}(s,\varphi;
t, \psi) = \sigma^2 (s \wedge t) \varphi(0) \psi(0)$. Then, using
the form of the generator $B^{\mathcal{A}}$ from (\ref{defBA}), the
semi-group $(S^{\mathcal{A}}_t)_{t \geq0}$ from Proposition~\ref
{PropAgeC01} and (\ref{AgeDrivingCovar}) of Proposition~\ref
{AgeMarkov}, one obtains after a substitution into Theorem~\ref
{ThmRegulatorMap} that the limiting diffusion scaled number of
customers in the system at time $t \geq0$ is heuristically given by
\begin{eqnarray*}
\hat{B}_t - \int_0^t
\hat{B}_s f(t-s)\,ds &=& \int_0^t
\bar{F}(t-s)\,d\hat{B}_s,
\end{eqnarray*}
where $\hat{B}=(\hat{B}_t)_{t \geq0}$ is a Brownian motion with
infinitesimal variance $\sigma^2+\lambda$.

\subsection{Residuals}\label{SubsecResidualsDiffusion}

We next proceed to study the residual service time process $\mathcal
{R}$ defined in Section~\ref{SubsecResidualEquations}. Our setup is the
same as in Section~\ref{SubsecResidualsFluid}. That is, we consider a
sequence of $G/\mathit{GI}/\infty$ queues indexed by $n$, where the arrival
rate to the $n$th system is of order $n$ and the service time
distribution does not change with $n$. For the remainder of this
subsection, we also assume that $ \bar{\mathcal{R}}^n_0 + \bar{E}^n
\mathcal{F} \Rightarrow
\bar{\mathcal{R}}_0 + \bar{E}\mathcal{F} $ as $n \rightarrow\infty
$, where
$ \bar{\mathcal{R}}_0 + \bar{E}\mathcal{F} $ is a\vspace*{1pt} nonrandom
quantity. By Theorem~\ref{ResidualFluidLimitThm} of Section~\ref
{SubsecResidualsFluid}, this implies that $\bar{\mathcal{R}}$
is nonrandom as well.

Now, for each $n \geq1$, in addition to the diffusion scaled quantities
defined in Section~\ref{SubsecAgesDiffusion}, let us also now define the
diffusion scaled quantities
\[
\hat{\mathcal{R}}^n \equiv\sqrt{n} \bigl(\bar{\mathcal{R}}^n
- \bar{\mathcal{R}} \bigr), \qquad\hat{\mathcal{R}}_0^n
\equiv\sqrt{n} \bigl(\bar{\mathcal{R}}_0^n - \bar{
\mathcal{R}_0} \bigr) \quad\mbox{and}\quad\hat{\mathcal{G}}^n
\equiv\sqrt{n} \bar{\mathcal{G}}^n.
\]
Then, after recalling the form of the fluid limit $\bar{\mathcal{R}}$
from Theorem~\ref{ResidualFluidLimitThm},
note that using system equation (\ref{ResidualsSystemEquation}) in
conjunction with Theorem~\ref{ThmRegulatorMap} and Proposition~\ref
{PropResidualC01}, one has that
%
%
\begin{equation}
\hat{\mathcal{R}}^n = \Psi_{B^{\mathcal{R}}}\bigl( \hat{\mathcal
{R}}^n_0+ \hat{E}^n \mathcal{F}+\hat{
\mathcal{G}}^n \bigr), \label{psidifres}
\end{equation}
where the map $\Psi_{B^{\mathcal{R}}}\dvtx  \mathbb{D}([0,T],\mathcal
{S}') \mapsto\mathbb{D}([0,T],\mathcal{S}')$ is a continuous map.
Our strategy now is to proceed similar
to as in Section~\ref{SubsecAgesDiffusion}. That is, we first prove a weak
convergence result for the sequence $( \hat{\mathcal{R}}^n_0+ \hat
{E}^n \mathcal{F} +\hat{\mathcal{G}}^n )_{n \geq1}$
and then we apply Theorem~\ref{cmtprop} together with (\ref
{psidifres}) in order to obtain a diffusion limit result for the
sequence $(\hat{\mathcal{R}}^n)_{n \geq1}$.

We first show that for each $n \geq1$, the process $\hat{\mathcal
{G}}^n$ may be well approximated by a process which is independent of $
\hat{\mathcal{R}}^n_0+ \hat{E}^n \mathcal{F}$. For each $n \geq1$,
let $\check{\mathcal{G}}^n $ be the $\mathcal{S}'$-valued process
defined for $\varphi\in\mathcal{S}$ by
\begin{eqnarray*}
\bigl\langle\check{\mathcal{G}}^n_t, \varphi\bigr
\rangle&=& \frac{1}{\sqrt
{n}} \sum_{i=1}^{\lfloor n \bar{E}_t \rfloor}
\bigl(\varphi(\eta_i) - \langle\mathcal{F},\varphi\rangle\bigr),\qquad
t\geq0.
\end{eqnarray*}
Note that it is clear that $\check{\mathcal{G}}^n $ is independent of
$ \hat{\mathcal{R}}^n_0+ \hat{E}^n\mathcal{F}$. We now have the
following result. Its proof may be found in the \hyperref[appen]{Appendix}.

%
\begin{lem}\label{Gdiffusionlimitindep}
If $ \hat{\mathcal{R}}^n_0+ \hat{E}^n \mathcal{F} \Rightarrow
\hat{\mathcal{R}}_0 + \hat{E}\mathcal{F} $ in $ \mathbb
{D}([0,T],\mathcal{S}') $ as $n
\rightarrow\infty$, then
%
%
\begin{equation}\label{firstconlem}
\hat{\mathcal{G}}^n -\check{\mathcal{G}}^n \Rightarrow0
\qquad\mbox{in }\mathbb{D}\bigl([0,T],\mathcal{S}'\bigr)\mbox{ as }n
\rightarrow\infty,
\end{equation}
and
%
%
\begin{equation}\label{secconvlem}
\check{\mathcal{G}}^n \Rightarrow\hat{\mathcal{G}}\qquad\mbox{in }
\mathbb{D}\bigl([0,T],\mathcal{S}'\bigr)\mbox{ as }n \rightarrow
\infty,
\end{equation}
where
$\hat{\mathcal{G}}$ is an $\mathcal{S}'$-valued Wiener process with
covariance functional given for $\varphi, \psi\in\mathcal{S}$ and
$s,t \geq0$ by
%
%
\begin{equation}
\label{Gcovariance} K_{\hat{G}}( s, \varphi; t, \psi) = (
\bar{E}_s \wedge\bar{E}_t ) \operatorname{Cov} \bigl(
\varphi(\eta), \psi(\eta) \bigr),
\end{equation}
where $\eta$ is a random variable with c.d.f. $F$.
\end{lem}

\begin{pf}See \hyperref[appen]{Appendix}.
\end{pf}

Using Lemma~\ref{Gdiffusionlimitindep}, we may now prove the
following result on the weak convergence of $( \hat{\mathcal{R}}^n_0+
\hat{E}^n\mathcal{F} +\hat{\mathcal{G}}^n )_{n \geq1}$. We have
the following.

\begin{prop}\label{Gdiffusionlimit}
If $\hat{\mathcal{R}}^n_0+ \hat{E}^n\mathcal{F} \Rightarrow
\hat{\mathcal{R}}_0 + \hat{E}\mathcal{F} $ in $\mathbb
{D}([0,T],\mathcal{S}') $ as $n
\rightarrow\infty$, then
%
%
\begin{equation}
\label{ResidualsJointDiffusion}
\qquad \hat{\mathcal{R}}^n_0 +
\hat{E}^n\mathcal{F} + \hat{\mathcal{G}}^n \Rightarrow
\hat{\mathcal{R}}_0 + \hat{E}\mathcal{F} +\hat{\mathcal{G}} \qquad
\mbox{in }\mathbb{D}\bigl([0,T],\mathcal{S}'\bigr) \mbox{ as }n
\rightarrow\infty,
\end{equation}
where $\hat{\mathcal{G}}$ is as given in Lemma~\ref
{Gdiffusionlimitindep} and is independent of $ \hat{\mathcal{R}}_0+
\hat{E} \mathcal{F} $.
\end{prop}

\begin{pf}
Since $ \check{G}^{n}$ is independent of $\hat{\mathcal{R}}^n_0
+\hat{E}^n\mathcal{F} $ for each $n \geq1 $, it follows Theorem~\ref
{Mitoma} and (\ref{secconvlem}) of Lemma~\ref
{Gdiffusionlimitindep} that we have the convergence
%
%
\begin{equation}\label{Q0AcheckG}
\qquad \hat{\mathcal{R}}^n_0 +
\hat{E}^n\mathcal{F} + \check{G}^{n} \Rightarrow\hat{
\mathcal{R}}_0 + \hat{E} \mathcal{F} + \hat{\mathcal{G}} \qquad
\mbox{in } \mathbb{D}\bigl([0,T],\mathcal{S}'\bigr) \mbox{ as
}n \rightarrow\infty.
\end{equation}
The result now follows by (\ref{Q0AcheckG}), Theorem~\ref{Mitoma},
(\ref{firstconlem}) of
Lemma~\ref{Gdiffusionlimitindep} and the fact that we may write
\begin{eqnarray*}
\hat{\mathcal{R}}^n_0 + \hat{E}^n\mathcal{F}
+ \check{G}^{n} &=& \hat{\mathcal{R}}^n_0 +
\hat{E}^n\mathcal{F} + \check{G}^{n}+ \bigl(
\hat{G}^{n}-\check{G}^{n}\bigr).
\end{eqnarray*}\upqed
\end{pf}

The following is now the main result of this subsection. It provides a
weak limit for the sequence $(\hat{\mathcal{R}}^n)_{n \geq1}$.

%
\begin{teo}\label{ResidualDiffusionLimitTheorem}
If $ \hat{\mathcal{R}}_0^n + \hat{E}^n \mathcal{F} \Rightarrow
\hat{\mathcal{R}}_0 +\hat{E} \mathcal{F} $ in $\mathbb
{D}([0,T],\mathcal{S}')$ as $n
\rightarrow\infty$, then
%
%
\begin{equation}
\hat{\mathcal{R}}^n \Rightarrow\hat{\mathcal{R}} \qquad\mbox{in }
\mathbb{D}\bigl([0,T],\mathcal{S}'\bigr) \mbox{ as }n
\rightarrow\infty, \label{weakdiffres}
\end{equation}
where $\hat{\mathcal{R}}$ is the solution to the stochastic integral equation
%
%
\begin{eqnarray}\label{ResidualDiffusionLimit}
\langle\hat{ \mathcal{R} }_t, \varphi
\rangle=
\langle\hat{\mathcal{R}}_0, \varphi\rangle+ \langle\hat{
\mathcal{G}}_t, \varphi\rangle+ \hat{E}_t \langle
\mathcal{F}, \varphi\rangle- \int_0^t \bigl
\langle\hat{ \mathcal{R} }_s, \varphi' \bigr\rangle
\,ds,
\nonumber\\[-8pt]\\[-8pt]
\eqntext{t \in[0,T], \varphi\in\mathcal{S}.}
\end{eqnarray}
In addition, if $\hat{E}$ is a Brownian
motion with diffusion coefficient $\sigma$, then $\hat{\mathcal{R}}$
is a generalized $\mathcal{S}'$-valued Ornstein--Uhlenbeck process
driven by a generalized \mbox{$\mathcal{S}'$-}valued Wiener process with covariance
functional
%
%
\begin{eqnarray}\label{EFGBrownianCovariance}
&& K_{\hat{E}\mathcal{F} + \hat{G}}( s,
\varphi; t, \psi)
\nonumber\\[-8pt]\\[-8pt]\nonumber
&&\qquad = \sigma^2 (
s \wedge t ) \mathbb E\bigl[ \varphi(\eta) \bigr]\mathbb E\bigl[ \psi(
\eta)\bigr] + (\bar{E}_s \wedge\bar{E}_t ) \operatorname{Cov}
\bigl( \varphi(\eta), \psi(\eta) \bigr),
\end{eqnarray}
where $\eta$ is a random variable with c.d.f. $F$.
\end{teo}

\begin{pf}
First note that by (\ref{psidifres}) we have that
$\hat{\mathcal{R}}^n=\Psi_{B^{\mathcal{R}}}(\hat{\mathcal{R}}^n_0+
\hat{E}^n \mathcal{F}+\hat{\mathcal{G}}^n )$, where the map $\Psi
_{B^{\mathcal{R}}}\dvtx \mathbb{D}([0,T],\mathcal{S}') \mapsto\mathbb
{D}([0,T],\mathcal{S}')$ is a continuous map. The convergence (\ref
{weakdiffres}) now follows by Theorem~\ref{cmtprop} and Proposition
\ref{Gdiffusionlimit}.

Next, note that if $\hat{E}$ is a Brownian motion with diffusion
coefficient $\sigma$, then it is easily checked that
$\hat{E} \mathcal{F}$ is an $\mathcal{S}'$-valued Wiener process
with covariance functional
%
%
\begin{equation}
\label{Ecovariance} K_{\hat{E}\mathcal{F}}( s, \varphi; t, \psi) =
\sigma^2 (
s \wedge t ) \mathbb E\bigl[ \varphi(\eta) \bigr] \mathbb E\bigl[ \psi
(\eta)
\bigr].
\end{equation}
Combining (\ref{Gcovariance}) with (\ref{Ecovariance}) and the fact
that $\hat{E}\mathcal{F}$ and $\hat{\mathcal{G}}$ are independent,
yields~(\ref{EFGBrownianCovariance}).
\end{pf}

%
\begin{rem}
Note that in the special case when the arrival process to the $n$th
system is a Poisson process with rate $\lambda n$, we then have that
$\bar{E} = \lambda e$ and so $\hat{E}$ turns out to be a Brownian
motion with diffusion coefficient $\lambda$. It then follows that
$K_{\hat{E}\mathcal{F}}( s, \varphi; t, \psi) = \lambda( s \wedge
t ) \mathbb E[ \varphi(\eta) \psi(\eta) ]$ and so Theorem {\ref{ResidualDiffusionLimitTheorem}} gives us a version of Theorem 3 of
{\em\cite{DM08b}}.
\end{rem}\vspace*{-12pt}

\begin{appendix}\label{appen}
\section*{Appendix}
In the \hyperref[appen]{Appendix}, we provide the proofs of several supporting lemmas
from the main body of the paper. We
begin with the proof of Lemma \ref{hazardlemma}.

\setcounter{equation}{0}

\begin{pf*}{Proof of Lemma \ref{hazardlemma}}
We prove part~(1) by induction. For each $n \geq0$ and $t \geq0$ fixed,
denote the quantity on
the lefthand side of (\ref{ccdfratiobound}) by $L_n$. For the base
case of
$n=0$, it is straightforward to see that $L_0 \leq1$. Next, for the inductive
step, suppose that (\ref{ccdfratiobound}) holds for $n =
0,\ldots,k-1$, and $t \geq0$. Then, we have that
\begin{eqnarray*}
\biggl\llvert\biggl(\frac{\bar{F}(x+t)}{\bar{F}(x)} \biggr)^{(k)}
\biggr\rrvert&=&
\biggl\llvert\biggl( \biggl(\frac{\bar{F}(x+t)}{\bar{F}(x)} \biggr)^{(1)}
\biggr)^{(k-1)} \biggr\rrvert
\\
&=& \biggl\llvert\biggl( \bigl(h(x) - h(x+t) \bigr) \frac{\bar
{F}(x+t)}{\bar{F}(x)}
\biggr)^{(k-1)} \biggr\rrvert
\\
&\leq&\sum_{i=0}^{k-1} \pmatrix{k-1
\cr
i}
\bigl\llvert\bigl(h(x) - h(x+t) \bigr)^{(k-1-i)} \bigr\rrvert\biggl
\llvert
\biggl(\frac{\bar
{F}(x+t)}{\bar{F}(x)} \biggr)^{(i)} \biggr\rrvert
\\
&\leq&2 \sum_{i=0}^{k-1} \pmatrix{k-1
\cr
i}
\bigl\llVert h^{(k-1-i)} \bigr\rrVert_\infty L_i
\\
&<& \infty.
\end{eqnarray*}
This completes the proof of part~(1).

We next prove part~(2) by induction as well. First recall that for $s,t
\geq0$, we may write
%
%
\begin{eqnarray}
\frac{\bar{F}(x+t)}{\bar{F}(x+s)}&=& \exp\biggl(-\int_{x+s}^{x+t}h(u)\,du
\biggr). \label{inthazard}
\end{eqnarray}
Hence, for the base case of $n=0$, we have that
\begin{eqnarray*}
\sup_{x \geq0} \biggl\llvert\frac{\bar{F}(x+t)}{\bar{F}(x)} -
\frac
{\bar{F}(x+s)}{\bar{F}(x)} \biggr\rrvert&=& \sup_{x \geq0} \biggl\llvert
\frac
{\bar{F}(x+s)}{\bar{F}(x)} \biggl(1 - \frac{\bar{F}(x+t)}{\bar
{F}(x+s)} \biggr) \biggr\rrvert
\\
&\leq&\sup_{x \geq0} \biggl\llvert\frac{\bar{F}(x+s)}{\bar{F}(x)}\biggr
\rrvert
\sup_{x \geq0} \biggl\llvert1 - \frac{\bar{F}(x+t)}{\bar{F}(x+s)
}\biggr\rrvert
\\
&\leq&\sup_{x \geq0} \biggl\llvert1 - \frac{\bar{F}(x+t)}{\bar{F}(x+s)
}\biggr
\rrvert
\\
&\leq&1 - e^{-\llVert h \rrVert _\infty\llvert t - s \rrvert}
\\
&\leq&\llVert h \rrVert_\infty\llvert t-s\rrvert,
\end{eqnarray*}
where the third inequality above follows from (\ref{inthazard}) and
the final inequality follows from the mean value theorem. Next, for the
inductive step, suppose that (\ref{ccdfratiobounddifference}) holds
for $n=0,1,\ldots,k-1$. We then have that
%
%
\begin{eqnarray}\label{ccdfdiffinequality}
&&\sup_{x \geq0} \biggl\llvert\biggl(\frac{\bar{F}(x+t)}{\bar{F}(x)} -
\frac{\bar{F}(x+s)}{\bar{F}(x)} \biggr)^{(k)} \biggr\rrvert
\nonumber
\\
&&\qquad = \sup_{x \geq0} \biggl\llvert\biggl( \bigl(h(x) - h(x+s)
\bigr) \frac
{\bar{F}(x+s)}{\bar{F}(x)} - \bigl(h(x) - h(x+t) \bigr) \frac
{\bar{F}(x+t)}{\bar{F}(x)}
\biggr)^{(k-1)} \biggr\rrvert
\nonumber
\\
&&\qquad = \sup_{x \geq0} \biggl\llvert\biggl( \bigl(h(x+t) - h(x+s)
\bigr) \frac{\bar{F}(x+s)}{\bar{F}(x)}
\nonumber\\[-8pt]\hspace*{-2pt}\\[-8pt]
\hspace*{-2pt}&&\hspace*{56pt}{}- \bigl(h(x) - h(x+t) \bigr) \biggl(\frac
{\bar{F}(x+t)}{\bar{F}(x)}
- \frac{\bar{F}(x+s)}{\bar{F}(x)} \biggr) \biggr)^{(k-1)} \biggr\rrvert
\nonumber
\\
&&\qquad = \sup_{x \geq0} \Biggl\llvert\sum
_{i=0}^{k-1} \pmatrix{k-1
\cr
i} \bigl[ \bigl(h(x+t) -
h(x+s) \bigr) \bigr]^{(k-1-i)} \biggl(\frac{\bar
{F}(x+s)}{\bar{F}(x)}
\biggr)^{(i)}
\nonumber
\\
\hspace*{-2pt}&&\hspace*{48pt}{}- \sum_{i=0}^{k-1} \pmatrix{k-1
\cr i} \bigl(h(x) - h(x+t) \bigr)^{(k-1-i)} \biggl(\frac{\bar{F}(x+t)}{\bar
{F}(x)} -\frac{\bar{F}(x+s)}{\bar{F}(x)} \biggr)^{(i)} \Biggr\rrvert. \nonumber
\end{eqnarray}
Now note that since by Assumption~\ref{boundedhazardassumption} we
have that $h \in C^{\infty}_b(\mathbb{R}_{+})$, it follows that all
of the derivatives of $h$ are bounded and hence
uniformly continuous as well. Using this fact, part~(1) and the inductive
hypothesis it now follows that (\ref{ccdfdiffinequality}) is less than
or equal to
\[
2 \llvert t-s\rrvert\sum_{i=0}^{k-1}
\pmatrix{k-1
\cr
i} \bigl( \bigl\llVert h^{((k-1-i)-1)} \bigr\rrVert
_{\infty} L_i + \bigl\llVert h^{(k-1-i)} \bigr\rrVert
_{\infty} M_i \bigr) \equiv M_k \llvert t-s\rrvert.
\]
This proves part~(2) and completes the proof.
\end{pf*}

We next provide the proof of Lemma~\ref{LemDcheck}

\begin{pf*}{Proof of Lemma~\ref{LemDcheck}}
We must show that
%
%
\begin{equation}
\label{Dlimit} \hat{\mathcal{D}}^{0,n} + \hat{\mathcal{D}}^n
\Rightarrow\hat{\mathcal{D}}^0 + \hat{\mathcal{D}} \qquad\mbox{in }
\mathbb{D}\bigl([0,T],\mathcal{S}'\bigr) \mbox{ as }n
\rightarrow\infty.
\end{equation}
Let $\varphi, \psi\in\mathcal{S}$ and let $t \geq0$. It then
follows by Proposition~\ref{PropMartingales} that we may write
\begin{eqnarray*}
&&\lll\hat{\mathcal{D}}^{0,n} + \hat{\mathcal{D}}^n
\rrr_t (\varphi, \psi)
\nonumber
\\
&&\qquad = \frac{1}{n} \Biggl( \sum_{i=1}^{A^n_0(\infty)}
\int_0^{\tilde
{\eta}_i \wedge t } \varphi\bigl(u - \tilde{
\tau}^n_i\bigr) \psi\bigl( u - \tilde{
\tau}^n_i \bigr) h_{\tilde{\tau}^n_i}(u) \,du
\\
&&\hspace*{81pt}{} + \sum_{i=1}^{E^n_t} \int
_0^{\eta_i \wedge( t -
\tau^n_i ) } \varphi(u) \psi(u) h(u) \,du \Biggr)
\nonumber
\\
&&\qquad = \int_0^{t} \bigl\langle\bar{
\mathcal{A}}^n_s, \varphi\psi h \bigr\rangle \,ds,
\end{eqnarray*}
where the second equality above follows as a result of Proposition~\ref
{PropositionHazard}. However, by the continuity of the integral
mapping on $\mathbb{D}([0,T],\mathbb{R})$, it now follows by Theorem~\ref{AgeFluidLimitTheorem} that
%
%
\begin{eqnarray}
\int_0^{e} \bigl\langle\bar{
\mathcal{A}}^n_s, \varphi\psi h \bigr\rangle \,ds &
\Rightarrow& \int_0^{e} \langle\bar{
\mathcal{A}}_s, \varphi\psi h \rangle \,ds \qquad\mbox{in }\mathbb{D}
\bigl([0,T],\mathbb{R}\bigr)\label{D0DQVdiffusionlimit}
\end{eqnarray}
as $n \rightarrow\infty$.

We now verify that $(\hat{\mathcal{D}}^{0,n} + \hat{\mathcal{D}}^n
)_{n \geq1}$ satisfies parts~(1) and~(2) of Theorem~\ref{Mitoma}. We
begin with part~(1). Let $\varphi= \psi$ in (\ref
{D0DQVdiffusionlimit}) and note that using
the fact that the maximum jumps of both $\langle\hat
{\mathcal{D}}^{0,n} +
\hat{\mathcal{D}}^n, \varphi\rangle$ and $\ll \hat
{\mathcal{D}}^{0,n} +
\hat{\mathcal{D}}^n \gg(\varphi,\varphi)$ over the interval $[0,T]$ are
uniformly bounded in $n$, along with the martingale FCLT \cite{EK86}
and (\ref{D0DQVdiffusionlimit}), yields the limit
\[
\bigl\langle\hat{\mathcal{D}}^{0,n} + \hat{\mathcal{D}}^n,
\varphi\bigr\rangle\Rightarrow\bigl\langle\hat{\mathcal{D}}^0 + \hat{
\mathcal{D}}, \varphi\bigr\rangle\qquad\mbox{in }\mathbb{D}\bigl
([0,T],\mathbb R
\bigr) \mbox{ as }n \rightarrow\infty.
\]
Thus, part~(1) of Theorem~\ref{Mitoma} is satisfied. We next proceed to
verify part~(2) of Theorem~\ref{Mitoma}. Let $m \geq1$ and let
$\varphi_1, \ldots, \varphi_m \in\mathcal{S}$ and $1 \leq
i,j \leq m$. Then, using the
fact that the maximum jumps of both $ (\langle\hat{\mathcal
{D}}^{0,n} + \hat{\mathcal{D}}^n, \varphi_1 \rangle, \ldots,\langle\hat{\mathcal{D}}^{0,n} + \hat{\mathcal{D}}^n, \varphi_m
\rangle)$
and $\langle \hat{\mathcal{D}}^{0,n} + \hat{\mathcal{D}}^n \rangle
(\varphi_i,
\varphi_j)$ are bounded over the interval $[0,T]$, uniformly in $n$,
it follows by
the martingale FCLT \cite{EK86} and (\ref{D0DQVdiffusionlimit}) that
\begin{eqnarray*}
&& \bigl(\bigl\langle\hat{\mathcal{D}}^{0,n} + \hat{\mathcal{D}}^n,
\varphi_1 \bigr\rangle, \ldots,\bigl\langle\hat{
\mathcal{D}}^{0,n} + \hat{\mathcal{D}}^n,
\varphi_m \bigr\rangle\bigr)
\\
&&\qquad  \Rightarrow\bigl(\bigl\langle\hat{
\mathcal{D}}^0 + \hat{\mathcal{D}}, \varphi_1 \bigr
\rangle, \ldots,\bigl\langle\hat{\mathcal{D}}^0 + \hat{\mathcal{D}},
\varphi_m \bigr\rangle\bigr)
\end{eqnarray*}
in $\mathbb{D}^m([0,T], \mathbb R)$ as $n \rightarrow\infty$. This limit
then provides convergence of the finite-dimensional distributions of the
random vector on the lefthand side above, which is sufficient to verify
part~(2) of
Theorem~\ref{Mitoma}. Thus, (\ref{Dlimit}) holds, where
$\hat{\mathcal{D}}^0 + \hat{\mathcal{D}}$ is an
$\mathcal{S}'$-valued Gaussian martingale with
tensor quadratic covariation given by~(\ref{D0DQVdiffusionlimit}).
Equation~(\ref{Dcovariance}) now holds since $\langle\hat{\mathcal{D}}^0 +
\hat{\mathcal{D}}, \varphi\rangle$ has independent increments for
each $\varphi\in\mathcal{S}$.
\end{pf*}

Next, we provide the proof of Lemma~\ref{Gdiffusionlimitindep}.

\begin{pf*}{Proof of Lemma~\ref{Gdiffusionlimitindep}}
We first prove that
%
%
\begin{eqnarray}
\hat{\mathcal{G}}^n &\Rightarrow&\hat{\mathcal{G}} \qquad\mbox{in }
\mathbb{D}\bigl([0,T],\mathcal{S}'\bigr) \mbox{ as }n
\rightarrow\infty. \label{proveme}
\end{eqnarray}
In order to do so, we will verify that parts~(1) and~(2) of Theorem~\ref
{Mitoma} are satisfied. We begin with part~(1). Let $\varphi, \psi\in
\mathcal{S}$ and note that by Proposition~\ref{PropMartingalesTwo},
the functional strong law of large numbers \cite{WhittBook} and the
random time change theorem \cite{Billingsley99}, we have that
%
%
\begin{eqnarray}
\label{GOQVdiffusionlimit} \bigl[ \hat{\mathcal{G}}^n \bigr](\varphi,
\psi) &=&
\frac{1}{n} \sum_{i=1}^{E^n} \bigl(
\varphi(\eta_i)-\langle\mathcal{F},\varphi\rangle\bigr) \bigl(\psi(
\eta_i)-\langle\mathcal{F},\psi\rangle\bigr)
\nonumber\\[-8pt]\\[-8pt]\nonumber
&\Rightarrow& \bar{E} \operatorname{Cov}\bigl(\varphi(\eta),\psi(\eta
)\bigr) \qquad\mbox{in }\mathbb{D}\bigl([0,T],\mathbb R\bigr)\mbox{ as }n
\rightarrow
\infty,
\end{eqnarray}
where $\eta$ is a random variable with c.d.f. $F$. Now letting $\varphi=
\psi$ in (\ref{GOQVdiffusionlimit}) and using the fact that
the maximum jump of $\langle\hat{\mathcal{G}}^n, \varphi\rangle$
over the interval $[0,T]$ is bounded uniformly in $n$, along with the
martingale FCLT \cite{EK86}, yields the limit
\[
\bigl\langle\hat{\mathcal{G}}^n, \varphi\bigr\rangle\Rightarrow
\langle\hat{\mathcal{G}},\varphi\rangle\qquad\mbox{in }\mathbb{D}\bigl([0,T],
\mathbb R\bigr) \mbox{ as }n \rightarrow\infty.
\]
Thus, part~(1) of Theorem~\ref{Mitoma} holds. We next prove that part~(2)
of Theorem~\ref{Mitoma} holds.
Let $m \geq1$ and let $\varphi_1, \ldots, \varphi_m \in\mathcal
{S}$. Then,\vspace*{1pt} using the limit (\ref{GOQVdiffusionlimit}) and the fact
that the maximum jump of $ (\langle\hat{\mathcal{G}}^n, \varphi
_1 \rangle, \ldots,\langle\hat{\mathcal{G}}^n, \varphi_m \rangle
)$ over the interval $[0,T]$
is bounded uniformly in $n$, along with the martingale FCLT \cite
{EK86}, yields the limit
\[
\bigl(\bigl\langle\hat{\mathcal{G}}^n, \varphi_1 \bigr
\rangle, \ldots, \bigl\langle\hat{\mathcal{G}}^n,
\varphi_m \bigr\rangle\bigr) \Rightarrow\bigl(\langle\hat{
\mathcal{G}}, \varphi_1 \rangle, \ldots, \langle\hat{\mathcal{G}},
\varphi_m \rangle\bigr) \qquad\mbox{in }\mathbb{D}^m
\bigl([0,T], \mathbb R\bigr)
\]
as $n \rightarrow\infty$. This limit then provides convergence of the
finite-dimensional distributions of the random vector on the left-hand
side above, which shows that part~(2) of Theorem~\ref{Mitoma} holds.
Thus, (\ref{proveme}) it proven.

In order to complete the proof, it now suffices to show that
%
%
\begin{eqnarray}
\hat{\mathcal{G}}^n - \check{\mathcal{G}}^n&
\Rightarrow&0 \qquad\mbox{in }\mathbb{D}\bigl([0,T],\mathcal{S}'
\bigr) \mbox{ as }n \rightarrow\infty. \label{proveme2}
\end{eqnarray}
However, in order to show (\ref{proveme2}), it suffices by Theorem
\ref{Mitoma} to show that for each $\varphi\in\mathcal{S}$,
%
%
\begin{eqnarray}
\bigl\langle\hat{\mathcal{G}}^n, \varphi\bigr\rangle- \bigl\langle
\check{\mathcal{G}}^n, \varphi\bigr\rangle&\Rightarrow&0 \qquad
\mbox{in }\mathbb{D}\bigl([0,T],\mathbb{R}\bigr) \mbox{ as }n
\rightarrow\infty. \label{proveme3}
\end{eqnarray}
We proceed as follows. In a similar manner to the above, one may show
using the martingale FCLT that
%
%
\begin{eqnarray}
\check{\mathcal{G}}^n &\Rightarrow&\hat{\mathcal{G}} \qquad\mbox{in
}\mathbb{D}\bigl([0,T],\mathcal{S}'\bigr) \mbox{ as }n
\rightarrow\infty. \label{proveme4}
\end{eqnarray}
Hence, for each $\varphi\in\mathcal{S}$, the sequence $(\langle\hat
{\mathcal{G}}^n, \varphi\rangle- \langle\check{\mathcal{G}}^n,
\varphi\rangle)_{n \geq1}$ is tight in $\mathbb{D}([0,T],\break \mathbb
{R})$ and so in order to show (\ref{proveme3}), it suffices to show
that for each $0 \leq t \leq T$,
%
%
\begin{eqnarray}
\bigl\langle\hat{\mathcal{G}}^n_t, \varphi\bigr\rangle-
\bigl\langle\check{\mathcal{G}}^n_t, \varphi\bigr\rangle&
\Rightarrow&0 \qquad\mbox{as }n \rightarrow\infty. \label{proveme5}
\end{eqnarray}
First note that we may write
\begin{eqnarray*}
&& \bigl\langle\hat{\mathcal{G}}^n_t, \varphi\bigr\rangle-
\bigl\langle\check{\mathcal{G}}^n_t, \varphi\bigr
\rangle
\\
&&\qquad = \mathbf{1}_{ \{ E^n_t \geq
\lfloor n \bar{E}_t \rfloor\} } \frac{1}{\sqrt{n}} \Biggl( \sum
_{i= \lfloor n \bar{E}_t \rfloor}^{E^n_t}\bigl(\varphi(\eta_i) -
\langle\mathcal{F},\varphi\rangle\bigr) \Biggr)
\\
&&\quad\qquad{} -\mathbf{1}_{ \{ \lfloor n \bar{E}_t \rfloor\geq E^n_t \} } \frac
{1}{\sqrt{n}} \Biggl( \sum
_{i=E^n_t }^{\lfloor n \bar{E}_t \rfloor
}\bigl(\varphi(\eta_i) -
\langle\mathcal{F},\varphi\rangle\bigr) \Biggr).
\end{eqnarray*}
Now squaring both sides of the above and using the basic identity
$(x_1+x_2)^2 \leq2 (x_1^2+x_2^2)$, it is straightforward to show that
one may write
%
%
\begin{eqnarray}
&&\bigl(\bigl\langle\hat{\mathcal{G}}^n_t, \varphi\bigr
\rangle- \bigl\langle\check{\mathcal{G}}^n_t, \varphi
\bigr\rangle\bigr)^2
\nonumber
\\
&&\qquad \leq \mathbf{1}_{ \{ \bar{E}^n_t \leq2\bar{E}_t \} } \frac
{2}{n} \Biggl( \sum
_{i=E^n_t \wedge\lfloor n \bar{E}_t \rfloor
}^{E^n_t \vee\lfloor n \bar{E}_t \rfloor}\bigl(\varphi(\eta_i) -
\langle\mathcal{F},\varphi\rangle\bigr) \Biggr)^2 \label{firstlast}
\\
&&\quad\qquad{}+ \mathbf{1}_{ \{ \bar{E}^n_t \geq2 \bar{E}_t \} }\frac{2}{n} \Biggl(
\sum
_{i=E^n_t \wedge\lfloor n \bar{E}_t \rfloor}^{E^n_t
\vee\lfloor n \bar{E}_t \rfloor}\bigl(\varphi(\eta_i) -
\langle\mathcal{F},\varphi\rangle\bigr) \Biggr)^2. \label{secolast}
\end{eqnarray}
We now show that each of the terms (\ref{firstlast}) and (\ref
{secolast}) converges to $0$ in probability as $n$ tends to $\infty$,
which implies (\ref{proveme5}) and completes the proof.

We begin with (\ref{firstlast}). First note that by the independence
of $\{\eta_i, i \geq1\}$ from the arrival process $E^n$, and the
i.i.d. nature of the sequence $\{\eta_i, i \geq1\}$, we have that
\begin{eqnarray*}
&&\mathbb E \Biggl[ \frac{2}{n} \mathbf{1}_{ \{ \bar{E}^n_t \leq2
\bar{E}_t \} } \Biggl( \sum
_{i=E^n_t \wedge\lfloor n \bar{E}_t
\rfloor}^{E^n_t \vee\lfloor n \bar{E}_t \rfloor}\bigl(\varphi(
\eta_i) - \langle\mathcal{F},\varphi\rangle\bigr)
\Biggr)^2 \Biggr]
\\
&&\qquad =\frac{2}{n} \mathbb E \Biggl[ \mathbf{1}_{ \{ \bar{E}^n_t \leq2
\bar{E}_t \} } \sum
_{i=E^n_t \wedge\lfloor n \bar{E}_t \rfloor
}^{E^n_t \vee\lfloor n \bar{E}_t \rfloor}\bigl(\varphi(\eta_i) -
\langle\mathcal{F},\varphi\rangle\bigr)^2 \Biggr]
\\
&&\qquad \leq\frac{2}{n} \mathbb E \Biggl[ \sum_{i= (E^n_t \wedge\lceil
2n\bar{E}_t \rceil) \wedge\lfloor n \bar{E}_t \rfloor}^{(E^n_t
\wedge\lceil2n\bar{E}_t \rceil) \vee\lfloor n \bar{E}_t \rfloor
}
\bigl(\varphi(\eta_i) - \langle\mathcal{F},\varphi\rangle
\bigr)^2 \Biggr]
\\
&&\qquad \leq 8 \bigl\llVert\varphi^2\bigr\rrVert
_{\infty} \mathbb E\bigl[\bigl\llvert\bigl(\bar{E}^n_t
\wedge2\bar{E}_t\bigr) - \bar{E}_t \bigr\rrvert\bigr].
\end{eqnarray*}
However, since $\bar{E}^n_t \Rightarrow\bar{E}_t$ as $n \rightarrow
\infty$, it follows that
\begin{eqnarray*}
\mathbb E\bigl[\bigl\llvert\bigl(\bar{E}^n_t \wedge2
\bar{E}_t\bigr) - \bar{E}_t \bigr\rrvert\bigr] &
\rightarrow&0\qquad\mbox{as }n \rightarrow\infty.
\end{eqnarray*}
This then implies that (\ref{firstlast}) converges to $0$ in
probability as $n$ tends to $\infty$, as desired.

We next proceed to (\ref{secolast}). Note that since $\bar{E}^n_t
\Rightarrow\bar{E}_t $ as $n \rightarrow\infty$, it follows that
for each $\varepsilon> 0$:
\begin{eqnarray*}
\lim_{n \rightarrow\infty} \mathbb{P} \Biggl( \mathbf{1}_{ \{ \bar
{E}^n_t \geq2 \bar{E}_t \} }
\frac{2}{n} \Biggl( \sum_{i=E^n_t
\wedge\lfloor n \bar{E}_t \rfloor}^{E^n_t \vee\lfloor n \bar{E}_t
\rfloor}
\bigl(\varphi(\eta_i) - \langle\mathcal{F},\varphi\rangle\bigr)
\Biggr)^2 > \varepsilon\Biggr) &=&0.
\end{eqnarray*}
This shows that (\ref{secolast}) converges to $0$ in probability as
$n$ tends to $\infty$, which completes the proof.
\end{pf*}
\end{appendix}

\section*{Acknowledgements}
The authors would like to thank the referees for their numerous helpful
comments and suggestions which have helped to improve the clarity and
overall exposition of the paper.


%

\printaddresses
\end{document}